\documentclass[12pt]{article}
\usepackage[margin=1.0in]{geometry}
\usepackage[utf8]{inputenc}
\usepackage{tikz}
\usepackage{mathtools}
\usepackage{mathrsfs,amssymb,amsfonts,amsthm,amsmath,amsbsy}
\usepackage{inputenc}
\usepackage[english]{babel}
\usepackage[T1]{fontenc}
\usepackage{hyperref}
\usepackage{float}
\usepackage{ upgreek }
\usepackage{enumitem}
\usepackage{bm}
\usepackage{amssymb}
\usepackage{mdframed}
\usepackage{verbatim}
\usepackage{dsfont}
\usepackage{stmaryrd}
\usepackage{soul}
\usepackage{titlesec} 
\usepackage{listings}
\usepackage{lipsum}
\usepackage[export]{adjustbox}
\usepackage[
backend=biber,
style=numeric,maxnames=99
]{biblatex}
\newtheorem{theorem}{Theorem}[section]
\newtheorem{proposition}[theorem]{Proposition} 
\newtheorem{corollary}[theorem]{Corollary}

\newtheorem{lemma}[theorem]{Lemma}

\newtheorem{definition1}[theorem]{Definition}

\newtheorem{remark1}[theorem]{Remark}

\newcommand{\somme}[3]{\sum_{#1}^{#2} #3}

\newcommand{\Pf}{\mathbb{P}}

\author{ Tanguy Lions\footnote{ENS Lyon, tanguy.lions@ens-lyon.fr}}

\date{}
\title{\textbf{Local limits of uniform triangulations with boundaries in high genus}}
\begin{document}
\maketitle
\begin{abstract}
We study the local limits of uniform random triangulations with boundaries in the regime where the genus is proportional to the number of faces. Budzinski and Louf proved in \cite{Budzinski_2020} that when there are no boundaries, the local limits exist and are the Planar Stochastic Hyperbolic Triangulation (PSHT) introduced in \cite{PSHT}. 

We show that when the triangulations considered have size $n$ and boundaries with total length $p\to +\infty$ and $p=o(n)$, the local limits around a typical boundary edge are the half-plane hyperbolic triangulations defined by Angel and Ray in \cite{Angel_Ray}. This provides, for the first time, a construction of these hyperbolic half-plane triangulations as local limits of large genus triangulations. 

We also prove that under the condition $p = o(n)$, the local limit when rooted on a uniformly chosen oriented edge is given by the PSHT. Contrary to \cite{Budzinski_2020}, the latter does not rely on the Goulden-Jackson recurrence relation, but only on coarse combinatorial estimates. Thus, we expect that the proof can be adapted to local limits in similar models.
\end{abstract}
\begin{figure}[H]
\centering
\includegraphics[scale=0.13]{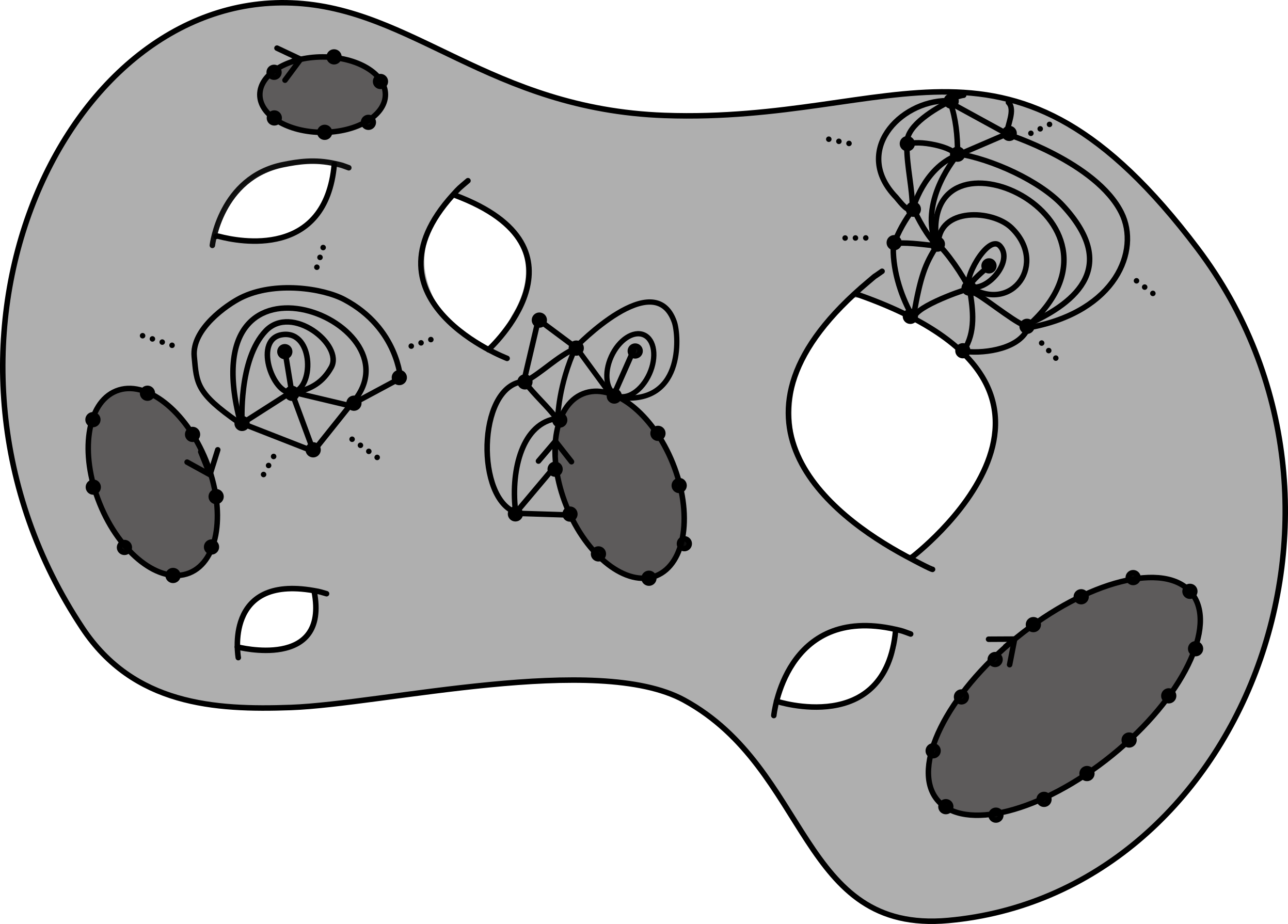}
\end{figure}

\section{Introduction}\label{introduction}
\paragraph{Enumeration of maps.}
Since the enumeration of planar maps by Tutte in the 1960s \cite{tutte}, maps have attracted a lot of interest from different communities within mathematics and theoretical physics. The enumeration of maps in terms of their size $n$ and their genus $g$ is linked to the topological recursion \cite{Eynard} and to solutions of several integrable hierarchies such as the KdV, the KP and $2$-Toda hierarchies \cite{Toda,Solitons,GOULDEN2008932}. These methods provide asymptotics for the enumeration of maps when $g$ is fixed and $n\to +\infty$ \cite{BENDER1986244}. When both $n,g \to +\infty$, these methods have been less successful so far, but the work of Budzinski and Louf \cite{Budzinski_2020} provides asymptotics for the number of triangulations in large genus up to sub-exponential factors.

\paragraph{Random maps.}
Simultaneously, probabilists have also been studying geometric properties of random maps chosen uniformly among a certain class. They particularly focused on certain regimes where the size $n$ becomes large. The planar case $g=0$ is now very well understood in terms of local and scaling limits. Various classes of models have been proved to converge locally to infinite random planar maps such as the UIPT \cite{Angel_UITP} (see also \cite{Stephenson} for type-I UIPT), the UIPQ \cite{krikun2006local,Chassaing_2006} or infinite Boltzmann planar maps \cite{budd2016peeling}. On the other hand, the Brownian sphere \cite{Le_Gall_2013,miermont} appears as the scaling limit of many models of random planar maps \cite{Marzouk,Berry_Albenque,Berry_Albenque2}. For $g \ge 0$, the Brownian sphere has analogues called Brownian surfaces \cite{brownian_surfaces} which have been recently proven in  \cite{bettinelli2022compact} to be the scaling limit of large quadrangulations of fixed genus $g \ge 0$ as $n \to +\infty$. The Brownian sphere is also linked to the Liouville quantum gravity (see \cite{duplantier_miller_sheffield_2021}) approach to random planar geometry.

Another regime that is well understood is when the genus is unconstrained and the map is a random uniform gluing of polygons. In this case, the number of vertices is typically very small, the mean degree of vertices goes to infinity and the genus concentrates close to the maximal possible value (\cite{CHMUTOV201623,unconstrained,Belyi}). Thus, there is no hope for a proper local or scaling limit result in this regime.

\paragraph{Random maps in large genus.}
The investigation of the regimes where $g \to +\infty$ was started much more recently. In the case of unicellular maps (i.e. maps with one face) many results are now known, such as the local convergence to a supercritical random tree \cite{local_unicellular}, the calculation of the diameter \cite{diameter_unicellular} and the convergence of the length spectrum \cite{length_spectrum_unicellular,length_spectrum_unicellular2}. Note that this last result seems to indicate a connection between large genus random maps and large genus random hyperbolic surfaces since the same limit for the length spectrum is obtained in both models \cite{length_spectrum_unicellular2,Mirzakhani_2019,Barazer_2025}. 

Beyond the unicellular case, Budzinski and Louf proved in \cite{Budzinski_2020} that large genus triangulations with a size proportional to the genus locally converge to the Planar Stochastic Hyperbolic Triangulations (PSHT) defined in \cite{PSHT}\footnote{Actually, in \cite{Budzinski_2020}, the authors consider type-I triangulations while the PSHT defined in \cite{PSHT} are type-II triangulations. Type-I refers to triangulations where loops are allowed and type-II to the case where loops are not allowed. The type-I analogue of \cite{PSHT} is defined in \cite{PSHTtypeI}.}. A similar local convergence result has been obtained in the case of arbitrary (even) face degrees in \cite{Budzinski2020LocalLO}. More recently, the paper \cite{budzinski2023distancesisoperimetricinequalitiesrandom}
shows that as soon as $\frac{g}{n} \to \theta \in (0,\frac{1}{2})$, the typical graph distances and the diameter are of logarithmic order with high probability. The present paper again focuses on the same regime $\frac{g}{n} \to \theta \in [0,\frac{1}{2})$ but where now, the triangulations have boundaries with perimeters $p_1^n,\cdots,p_{\ell}^n$ such that $\somme{i=1}{\ell}{p^n_i}\underset{n\to+\infty}{\longrightarrow}+\infty$. The local limit of these triangulations with boundaries is expected to be the half-plane analogue of the PSHT.
\paragraph{Planar/half-planar stochastic hyperbolic triangulations.}
Angel and Ray introduced in \cite{Angel_Ray} a one-parameter family of triangulations of the half-plane. This family splits into two subfamilies, one being \emph{subcritical} and the other \emph{supercritical}. In this paper, we will be interested in the supercritical models. The PSHT introduced in \cite{PSHT} were motivated as the plane analogue of these supercritical half-plane triangulations.\\
Here, we consider the type-I version of the half-plane triangulations of \cite{Angel_Ray}. They form a one-parameter family of triangulations of the half-plane $(\mathbb{H}_{\lambda})_{0 < \lambda \le \lambda_c}$  introduced in \cite{budzinski_geodesics}, where $\lambda_c = (12\sqrt{3})^{-1}$. For $0 < \lambda \le \lambda_c$, let us introduce $h \in (0,\frac{1}{4}]$ such that $\lambda = \frac{h}{(1+8h)^{\frac{3}{2}}}$. The triangulations $\mathbb{H}_{\lambda}$ are characterized by the probabilities of the events of the form $t \subset \mathbb{H}_{\lambda}$ where $t$ are maps of a certain form. More precisely, the maps $t$ considered are those with a finite number of triangles (the grey faces in Figure~\ref{halfplane}), one infinite boundary (the dark grey face in Figure~\ref{halfplane}) and one infinite hole (the white infinite face in Figure~\ref{halfplane}). For such a triangulation $t$, we denote by $|t_{\mathrm{in}}|$ the number of vertices of $t$ not on the infinite boundary (the green vertices in Figure~\ref{halfplane}) and by $|\partial^{*}t|-|\partial t|$ the length of the red segment minus the length of the blue segment in Figure~\ref{halfplane}. Finally, we write $t \subset \mathbb{H}_{\lambda}$ if $\mathbb{H}_{\lambda}$ can be obtained from $t$ by filling the infinite hole with a triangulation of the half-plane. The half-plane triangulations $(\mathbb{H}_{\lambda})_{0 < \lambda \le \lambda_c}$ are characterized by the spatial Markov property: for any such triangulation $t$, we have 
\begin{align*}
\Pf(t \subset \mathbb{H}_{\lambda}) = \bigg(8+\frac{1}{h}\bigg)^{|\partial^{*}t|-|\partial t|}\lambda^{|t_{\mathrm{in}}|}.
\end{align*}

\begin{figure}[H]
\centering
\includegraphics[scale=0.18]{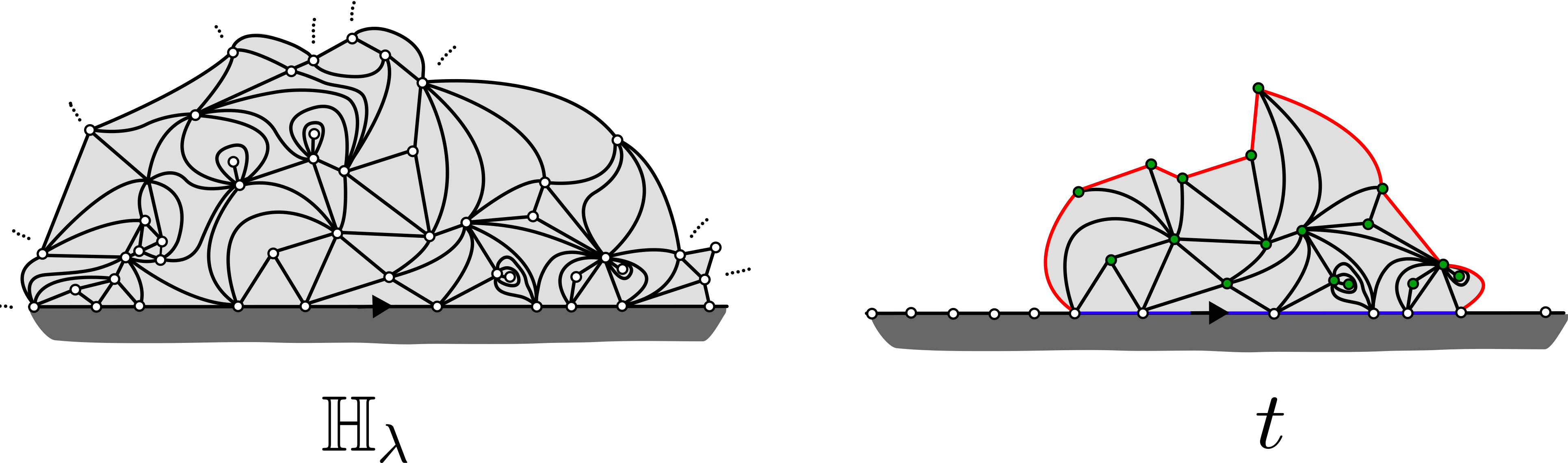}
\caption{On the left: a realisation of $\mathbb{H}_{\lambda}$. On the right, a triangulation $t$ with one infinite boundary (dark grey face) and one infinite hole (white face) such that $t \subset \mathbb{H}_{\lambda}$. On this example, $|t_{\mathrm{in}}| = 17$ denotes the number of green vertices and $|\partial^{*}t|-|\partial t| = 8 - 5 = 3$ denotes the number of red edges minus the number of blue edges.  }
\label{halfplane}
\end{figure}
We also introduce the type-I PSHT $(\mathbb{T}_{\lambda})_{0\le \lambda \le \lambda_c}$ as the full-plane analogue of the half-plane family $(\mathbb{H}_{\lambda})_{0 < \lambda \le \lambda_c}$. The family $(\mathbb{T}_{\lambda})_{0\le \lambda \le \lambda_c}$ was introduced in \cite{PSHTtypeI} and is characterized by the spatial Markov property: for any triangulation $t$ with $|t|$ vertices and one hole of perimeter $p$, we have 
\begin{align*}
\Pf(t \subset \mathbb{T}_{\lambda}) = C_{p}(\lambda)\lambda^{|t|},
\end{align*}
where $C_p(\lambda)$ is explicit (see Section~\ref{triangulation_of_the_plane}). For every $p \ge 1$, we will also need to consider a version $ \mathbb{T}^{(p)}_{\lambda}$ of $ \mathbb{T}_{\lambda}$ with a boundary of perimeter $p$. In that case, the spatial Markov property can be written as follows: for any triangulation $t$ with boundary of perimeter $p$, with $|t_{\mathrm{in}}|$ vertices not on the boundary and with one hole of perimeter $q$, we have 
\begin{align*}
\Pf(t \subset \mathbb{T}^{(p)}_{\lambda}) = \frac{C_{q}(\lambda)}{C_{p}(\lambda)}\lambda^{|t_{\mathrm{in}}|}.
\end{align*}
\paragraph{Uniform triangulations with boundaries.}
We now define the finite models we will be interested in. Fix integers $g,\ell \ge 0$, $n \ge 2g-1$ and $\mathbf{p} = (p_1,\dots,p_{\ell}) \in \mathbb{N}^{\ell}$. We write $|\mathbf{p}|= p_1+\cdots+ p_{\ell}$. We introduce the set of triangulations $\mathcal{T}_{\mathbf{p}}(n,g)$ defined as the set of maps $t$ such that:
\begin{itemize}
	\item[$\bullet$] $t$ has genus $g$ and $n+2-2g$ vertices.
	\item[$\bullet$] $t$ has $\ell$ distinguished faces $\partial_1,\cdots,\partial_{\ell}$ of perimeters $p_1,\cdots,p_{\ell}$ called the \emph{boundaries}, which are simple and vertex-disjoint, and each face $\partial_i$ is equipped with a distinguished oriented edge $e_i$ such that $\partial_i$ lies to the right of $e_i$.
	\item[$\bullet$] The faces of $t$ that are not boundaries have degree $3$.
\end{itemize}  
For $t \in \mathcal{T}_{\mathbf{p}}(n,g)$ and $e$ an oriented edge of $t$, the pair $(t,e)$ is viewed as the triangulation $t$ rooted at $e$. For a sequence $(t_n,e^n)_{n \ge 0}$ of rooted triangulations, we say that the sequence converges locally to a rooted triangulation $(t,e)$ if, for any $r \ge 0$, the ball of radius $r$ around $e^n$ seen as a map is eventually equal to the ball of radius $r$ around $e$ in $t$. See Section~\ref{preliminaries} for more precise definitions. We denote by $T_{n,g,\mathbf{p}}$ an element of $\mathcal{T}_{\mathbf{p}}(n,g)$ chosen uniformly at random. For any $\lambda \in (0,\lambda_c]$, let $h = h(\lambda)$ be the unique $h \in (0,\frac{1}{4}]$ such that $\lambda = \frac{h}{(1+8h)^{\frac{3}{2}}}$ and let
\begin{align*}
d(\lambda) = \frac{h \log\Big(\frac{1+\sqrt{1-4h}}{1-\sqrt{1-4h}}\bigg)}{(1+8h)\sqrt{1-4h}}.
\end{align*} 
It is shown in \cite{Budzinski_2020} that $d(\lambda)$ is increasing with $d(\lambda_c) = \frac{1}{6}$ and $\lim_{\lambda \to 0}d(\lambda) = 0$, and that the quantity $d(\lambda)$ denotes the expectation of the inverse of the root degree in $\mathbb{T}_{\lambda}$. Our main theorems are local convergence results, both around a typical oriented edge and around a typical oriented edge chosen on a boundary. The main new result is the following theorem.
    \begin{theorem}\label{local_limit_boundary}
    Fix $\theta \in [0,\frac{1}{2})$, $\displaystyle \frac{g_n}{n} \underset{n \to +\infty}{\longrightarrow}\theta$ and $\mathbf{p}^{n} = (p^n_{1},\cdots,p^n_{\ell_n})$ such that $|\mathbf{p}^{n}| = o(n)$ and $\displaystyle \frac{|\mathbf{p}^{n}|}{\ell_n}  \underset{n \to +\infty}{\longrightarrow}+\infty$. Denote by $e^n$ a uniformly chosen oriented edge on the union $\partial_1\cup \cdots \cup \partial_{\ell_n}$ of the boundaries of $T_{n,g_n,\mathbf{p}^{n}}$. Then, the following convergence holds for the local topology
\begin{align*}
   (T_{n,g_n,\mathbf{p}^{n}},e^n) \overset{(d)}{\to}\mathbb{H}_{\lambda(\theta)},
\end{align*}
where $\lambda(\theta)$ is the unique solution to the equation 
\begin{align}\label{deflambdatheta}
    d(\lambda(\theta)) = \frac{1-2\theta}{6}.
\end{align}
    \end{theorem}
 The extra hypothesis $\displaystyle \frac{|\mathbf{p}^{n}|}{\ell_n} \underset{n \to +\infty}{\to}+\infty$ ensures that a typical boundary has a size that tends to $+\infty$ (in probability). This theorem answers in the affirmative conjectures in \cite{Angel_Ray}\cite[open question 8.2]{Stflour} stating that the hyperbolic triangulations $\mathbb{H}_{\lambda}$ of the half-plane can be obtained as local limits of large genus triangulations with a boundary length that also tends to infinity. We expect that the result should hold if we choose our root edge to be the distinguished edge $e^n_1$ on the boundary $\partial_1$ and we only assume $p^n_1 \underset{n \to +\infty}{\longrightarrow} +\infty$.

In the regime where $|\mathbf{p}^{n}| \sim \alpha n$ with $\alpha > 0$, we still expect to see a local convergence result. However, in that case, we conjecture that the limit should be a triangulation with infinitely many infinite faces. No canonical such model has yet been studied in the literature. We will investigate this problem in a future work.

Now we give another local convergence result when $e^n$ is chosen uniformly at random among all edges of $T_{n,g_n,\mathbf{p}^n}$. In \cite[Theorem $1.1$]{Budzinski_2020}, the authors treat the case $\mathbf{p}^n = \emptyset$, here we generalize the result to the case $|\mathbf{p}^{n} |=o(n)$. 
\begin{theorem}\label{local_limit_middle}
    Fix $\theta \in [0,\frac{1}{2})$, $\displaystyle \frac{g_n}{n} \to \theta$ and $\mathbf{p}^{n} = (p^n_{1},\cdots,p^n_{\ell_n})$ such that $|\mathbf{p}^{n}|= o(n)$. Let $e^n$ denote an oriented edge chosen uniformly at random in $T_{n,g_n,\mathbf{p}^{n}}$. Then the following convergence holds for the local topology
    \begin{align*}
       &(T_{n,g_n,\mathbf{p}^{n}},e^n) \overset{(d)}{\to}\mathbb{T}_{\lambda(\theta)},
       \end{align*}
 where $\lambda(\theta)$ is the unique solution to the equation 
\begin{align*}
    d(\lambda(\theta)) = \frac{1-2\theta}{6}.
\end{align*}
Moreover, let us also fix $p \ge 1$. Consider $T_{n,g_n,(p,\mathbf{p}^{n})}$ and let $e^n_1$ denote the distinguished edge on the first boundary of $T_{n,g_n,(p,\mathbf{p}^{n})}$. Then the following convergence holds for the local topology
    \begin{align*}
(T_{n,g_n,(p,\mathbf{p}^{n})},e^n_1) \overset{(d)}{\to}\mathbb{T}^{(p)}_{\lambda(\theta)}.
 \end{align*}
    \end{theorem}

\paragraph{Motivation: long-term exploration of uniform triangulations.}
Let us give a potential application of our theorems. First, one can discover $T_{n,g_n}$ triangle by triangle in a Markovian fashion. This yields a growing sequence of triangulations $t_1\subset \cdots t_k\subset\cdots$, called the \emph{peeling exploration} of $T_{n,g_n}$. Depending on the order in which the triangles are discovered, we can keep track of different geometric information. However, the transition probabilities appearing in the exploration are very hard to understand, which makes the process hard to study. Using \cite[Theorem $1.1$]{Budzinski_2020} or equivalently Theorem~\ref{local_limit_middle}, we have $T_{n,g_n} \to \mathbb{T}_{\lambda(\theta)}$. Thus, as long as we restrict to studying a bounded number of peeling steps, the transition probabilities converge to the ones of $\mathbb{T}_{\lambda(\theta)}$ which are explicit. Still, understanding the exploration for an unbounded number of steps is out of reach using only \cite{Budzinski_2020}. Using Theorem~\ref{local_limit_boundary}, one can understand more steps in the peeling exploration. Indeed, for $k \to +\infty$ and $k = o(n)$ assuming that $T_{n,g_n} \backslash t_k$ is connected, then $T_{n,g_n} \backslash t_k \overset{(d)}{=} T_{n-o(n),g_n-o(n),\mathbf{p}^{n}_k}$ for some random $\mathbf{p}^{n}_k$ that represents the perimeters of the holes of $t_k$. Moreover, Theorem~\ref{local_limit_boundary} gives $(T_{n-o(n),g_n-o(n),\mathbf{p}^{n}_k},e^n) \to \mathbb{H}_{\lambda(\theta)}$ when rooted on a typical edge on the boundary, which ensures that we can understand the peeling transitions for an unbounded number of steps, provided $k = o(n)$.

 In a future work, we plan to use this to study the typical distances in $T_{n,g_n,\mathbf{p}^{n}}$ by discovering the balls of radius $r$ with such a Markovian exploration. We hope this will make more precise some results obtained in \cite{budzinski2023distancesisoperimetricinequalitiesrandom} on the logarithmic order of magnitude of these distances.
 
 Finally, note that as in \cite{Budzinski_2020}, our results provide some combinatorial estimates (see Proposition~\ref{combi_estimate_rooted_middle}), which we hope can be of independent interest.

\paragraph{Strategy of the proof.}\label{strategy_of_proof}
The proofs of Theorem~\ref{local_limit_boundary} and Theorem~\ref{local_limit_middle} are quite robust and rely mainly on coarse combinatorial estimates that hold for many classes of large genus random maps. In particular, the proof of Theorem~\ref{local_limit_middle} does not rely on the Goulden-Jackson recursion formula obtained in \cite{GOULDEN2008932}.

In Section~\ref{section_middle}, we prove Theorem~\ref{local_limit_middle}. The main new argument is the proof that the limit is planar. The rest of the proof is standard and follows the same strategy as in \cite{Budzinski_2020}, so we do not describe it here. To prove the planarity, one roughly needs to prove that the probability $\Pf(t \subset (T_{n,g_n,\mathbf{p}^{n}},e^n))$, where $t$ is a finite triangulation of genus $1$, goes to $0$ as $n \to +\infty$. Let us explain how to prove that. One can write this probability in terms of combinatorial estimates and gets a quantity that roughly looks like:
\begin{center}
$\displaystyle \frac{\tau_{\mathbf{p}^{n}}(n-m,g_n-1)}{\tau_{\mathbf{p}^{n}}(n,g_n)}$, for some fixed $m \ge 0$.
\end{center} 
By standard estimates, there is an absolute constant $C > 0$ such that the following crude estimate holds:
\begin{align}\label{crude_estimate}
\exp(-Cn)n^{2g_n} \le \tau_{\mathbf{p}^{n}}(n,g_n) \le \exp(Cn)n^{2g_n}.
\end{align}
 Thus, we find 
\begin{center}
$\displaystyle \frac{\tau_{\mathbf{p}^{n}}(n-m,g_n-1)}{\tau_{\mathbf{p}^{n}}(n,g_n)} \le \exp(2Cn)n^{-2}$.
\end{center}
Unfortunately, the right-hand side does not tend to $0$. We now describe one of the main new ideas of the paper, which solves this issue. Let us fix $\varepsilon > 0 $ and assume that $\Pf(t \subset (T_{n,g_n,\mathbf{p}^{n}},e^n)) \ge \varepsilon > 0$ for $n$ large enough. Since $e^n$ is uniform, there is a probability at least $\varepsilon/2$ that $T_{n,g_n,\mathbf{p}^{n}}$ contains $\eta n$ copies of $t$  called $t^1,\cdots,t^{\eta n}$, with $\eta > 0$ a small constant. Moreover, up to choosing $\eta$ smaller, we may assume that these copies are disjoint. We can rewrite the probability to have the copies $t^1,\cdots,t^{\eta n}$ contained in $T_{n,g_n,\mathbf{p}^n}$ in terms of combinatorics and obtain a quantity that looks like 
\begin{center}
$\displaystyle \frac{\tau_{\mathbf{p}^{n}}(n-\eta nm,g_n-\eta n)}{\tau_{\mathbf{p}^{n}}(n,g_n)} $.
\end{center}
 By the above estimates \eqref{crude_estimate} we have the bound 
\begin{align*}
\displaystyle \frac{\tau_{\mathbf{p}^{n}}(n-\eta nm,g_n-\eta n)}{\tau_{\mathbf{p}^{n}}(n,g_n)} \le \exp(2Cn)n^{-2\eta n}.
\end{align*}
Taking into account all the possible locations for the copies $t^1,\cdots,t^{\eta n}$ multiplies this quantity by a factor $n^{\eta n}$. We conclude using the fact that $\exp(2Cn)n^{-\eta n} \to 0$. This proves $\Pf(t \subset (T_{n,g_n,\mathbf{p}^{n}},e^n)) \le \varepsilon$ and thus the local planarity. In words, taking many copies of $t$ makes the error factor $\exp(Cn)$ negligible compared to the reduction of the genus by $\eta n$ which divides the volume by $n^{2\eta n}$. This argument still works if we replace $(T_{n,g_n,\mathbf{p}^{n}})_{n \ge 1}$ with any model of triangulations $(\tilde{T}_n)_{n \ge 0}$ that is invariant under rerooting and such that $\tilde{T}_n$ is absolutely continuous with respect to $T_{n,g_n,\mathbf{p}^{n}}$ and such that $\displaystyle \frac{\mathrm{d}\tilde{T}_n}{\mathrm{d}T_{n,g_n,\mathbf{p}^{n}}} \le \exp(Cn)$ for some absolute constant $C > 0$. In particular, one would expect this argument to apply to high genus triangulations equipped with a decoration such as an Ising model.\\

In Section~\ref{section_boundary}, we prove Theorem~\ref{local_limit_boundary}. The proof is very technical here, and the difficulties are very different from those encountered in Theorem~\ref{local_limit_middle}. Indeed, the planarity comes for free, since the planarity in Theorem~\ref{local_limit_middle} gives the estimate
\begin{center}
$\displaystyle \frac{\tau_{\mathbf{p}^{n}}(n,g_n-1)}{\tau_{\mathbf{p}^{n}}(n,g_n)}\to 0$.
\end{center} 
 We also easily get the one-endedness, again by using combinatorial estimates that follow from Theorem~\ref{local_limit_middle}. Let us write $\partial_i$ for the boundary such that $e^n \in \partial_i$. The main difficulty here is to exclude the following two pathological cases that are closely related to the geometry of the boundaries (see Figure~\ref{pathogical_cases}):
 \begin{enumerate}
   \item A small neighbourhood of $e^n$ intersects another boundary $\partial_j$ for $i \neq j$.
   \item A small neighbourhood of $e^n$ sees a part of $\partial_i$ that is far from $e^n$ along $\partial_i$, i.e. the boundary $\partial_i$ folds onto itself near $e^n$.  
\end{enumerate}
\begin{figure}[H]
\centering
\includegraphics[scale=0.25]{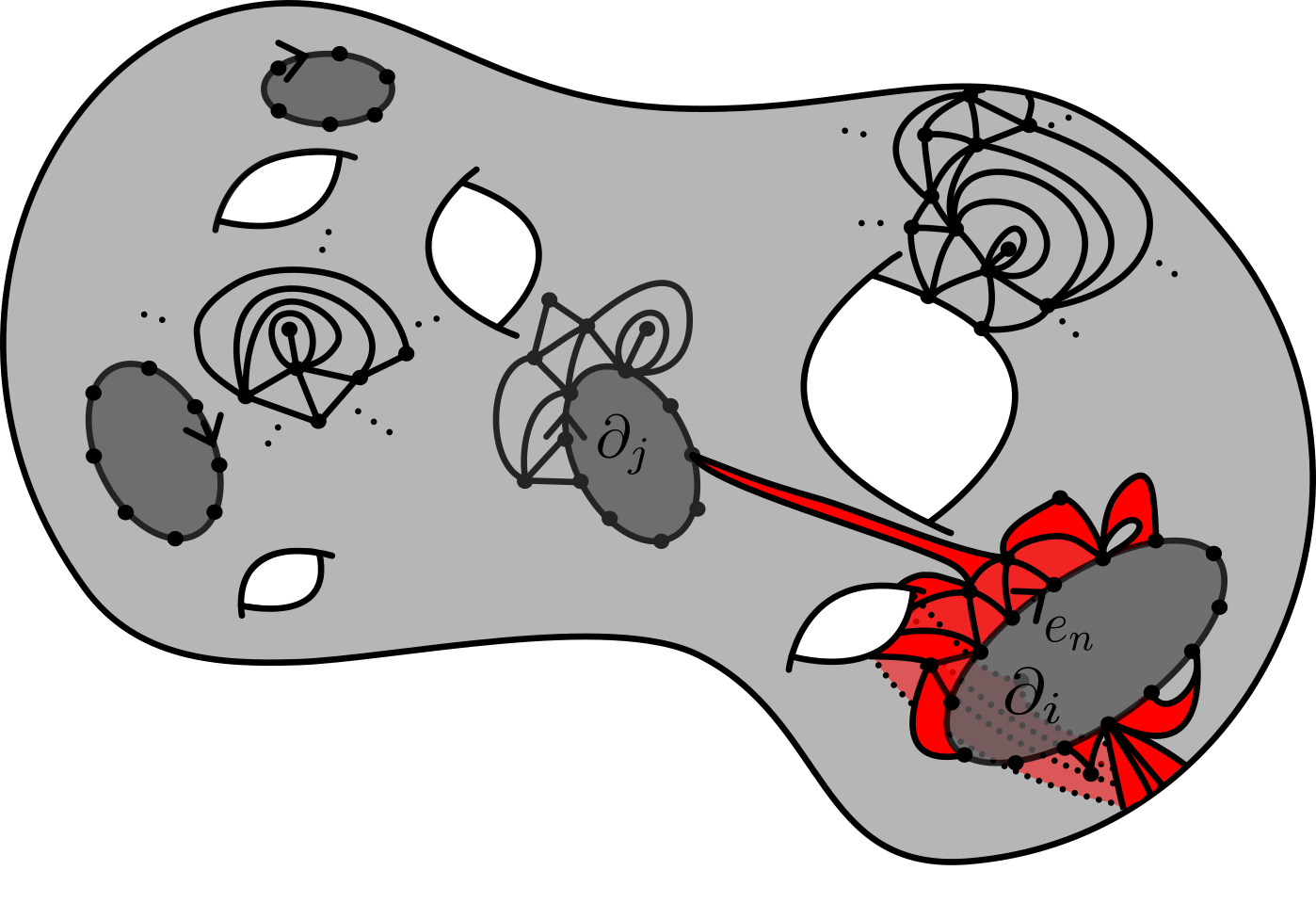}
\caption{The red neighbourhood of the edge $e^n$ touches another boundary and the opposite side of the boundary $\partial_i$. These are the two pathological cases to exclude.}
\label{pathogical_cases}
\end{figure}
 These cases cannot be ruled out using Theorem~\ref{local_limit_middle}, since when $e^n$ is chosen uniformly at random on $T_{n,g_n,\mathbf{p}^{n}}$, it typically lies far from the boundaries. These two cases are excluded in Section~\ref{subsection_pathological} which is the most technical part of this work. Finally, there remains to identify the limit. In \cite{Angel_Ray}, the authors show that the random triangulations of the half-plane that satisfy a nice Markov property form a one-parameter family\footnote{Their work is done in the type $II$ case. In the type-$I$ case the family obtained with their work is much larger. However, with a stronger Markov property that we use here (see Definition~\ref{weak_Markovian_halfplane}) we recover a one-parameter family.} $(H_{\alpha})_{0 < \alpha < 1}$. This family splits into two subfamilies that behave very differently, one being \emph{subcritical} and one \emph{supercritical}:
 \begin{itemize}
 	\item[$\bullet$] The elements of the family $(H_{\alpha})_{\alpha < \frac{1}{\sqrt{3}}}$ are \emph{subcritical}, i.e. the half-plane triangulations look like critical Galton-Watson trees.
 	\item[$\bullet$] The elements of the family $(H_{\alpha})_{\alpha > \frac{1}{\sqrt{3}}}$ are \emph{supercritical}. They share many properties with hyperbolic graphs such as exponential volume growth, positive anchored expansion...
\end{itemize}  
Using structural results on half-plane triangulations satisfying a Markov property, we can show that any subsequential limit in Theorem~\ref{local_limit_boundary} is a mixture of the half-plane triangulations $(H_{\alpha})_{0 < \alpha < 1}$.
Then, using Theorem~\ref{local_limit_middle}, we obtain $$\displaystyle \frac{\tau_{\mathbf{p}^{n}}(n-1,g_n)}{\tau_{\mathbf{p}^{n}}(n,g_n)} \to \lambda(\theta),$$
 and deduce that any subsequential limit is a mixture of $H_{\alpha_1}$ and $H_{\alpha_2}$, where $H_{\alpha_1}$ is subcritical and $H_{\alpha_2} \overset{(d)}{=}\mathbb{H}_{\lambda(\theta)}$ is the expected limit.
 To exclude $H_{\alpha_1}$, we use surgery operations on the \emph{peeling diagram} of $T_{n,g_n,\mathbf{p}^{n}}$ inspired by \cite{Contat2022LastCD}.
\paragraph{Acknowledgements.} 
I am grateful to Thomas Budzinski for his crucial help at various stages of this project. I also thank Grégory Miermont for stimulating discussions. 
\tableofcontents

\section{Preliminaries}\label{preliminaries}

\subsection{Definitions}\label{Definitions}

We begin by recalling the main definitions used throughout the paper.

A (finite or infinite) \emph{map} $m$ is obtained by gluing together a collection of oriented polygons along their edges, with matching orientations, so that the resulting surface is connected. If finitely many polygons are glued, the resulting surface is orientable, and we can define its genus. We say that a map $m$ is \emph{rooted} if it is equipped with a distinguished oriented edge called the \emph{root edge}. The face on the right of the root edge is called the \emph{root face}, and the vertex at its origin is called the \emph{root vertex}. We denote by $m^{*}$ the dual of $m$, defined as the map where the vertices correspond to the faces of $m$ and two vertices of $m^{*}$ are connected by an edge if the corresponding faces are connected by an edge in $m$.

A \emph{triangulation} is a map all of whose faces have degree~$3$. We focus on \emph{type-I triangulations}, that is, triangulations that may contain multiple edges and loops. For $n \ge 1$ and $g \ge 0$, we denote by $\mathcal{T}(n,g)$ the set of rooted triangulations of genus~$g$ with $2n$ faces. By Euler's formula, a triangulation in $\mathcal{T}(n,g)$ has $3n$ edges and $n+2-2g$ vertices. In particular, the set $\mathcal{T}(n,g)$ is non-empty if and only if $n \ge 2g-1$. We denote by $\tau(n,g)$ the cardinality of $\mathcal{T}(n,g)$, and we write $T_{n,g}$ for a uniform random triangulation in $\mathcal{T}(n,g)$.

In this paper, we will consider maps with boundaries. We introduce two notions, the first including the second.

\begin{definition1}\label{triangulation_with_holes}
A triangulation with holes is a finite map $t$ with:
\begin{itemize}
	\item[$\bullet$] A set of distinguished faces, called the boundaries (or external faces) $\partial_1,\dots,\partial_{\ell}$, of degrees $p_1,\dots,p_{\ell} \ge 1$, which are vertex-simple (that is, all vertices of $\partial_i$ are distinct) and share no vertices. Each boundary $\partial_i$ is equipped with a distinguished edge $e_i$ such that $\partial_i$ lies to the right of $e_i$.
	\item[$\bullet$] A set of distinguished faces $h_1,\dots,h_{r}$ of degrees $q_1,\dots,q_r \ge 1$, called the holes, which are edge-simple (each $h_i$ is composed of $q_i$ edges) and share no edges.
	\item[$\bullet$] Faces that are neither boundaries nor holes have degree~$3$.
	\item[$\bullet$] The dual $t^{*}$ remains connected when removing the boundaries.
\end{itemize} 
The perimeter of $t$ is the tuple $\mathbf{p} =(p_1,\dots,p_{\ell})$, and the boundary length or total perimeter is $|\mathbf{p}| = p_1+\cdots+p_{\ell}$. Note that in the above definition, boundaries and holes may share vertices and edges. The internal faces of $t$ are the faces that are not boundaries. In particular, holes are internal faces. Similarly, the internal vertices (resp. edges) are the vertices (resp. edges) that do not lie on a boundary.
\end{definition1}

This definition extends naturally to infinite triangulations and also to the case where $p_i = +\infty$. For a triangulation with holes $t$, we write $\partial^{*}t$ for the union of holes and $\partial t$ for the union of boundaries.

\begin{definition1}\label{triangulation_multipolygon}
For $\ell \ge 0$ and $p_1,\dots,p_{\ell}\ge 1$, a triangulation of the $\mathbf{p}$ multi-polygon (or of the $\mathbf{p}$-gon) is a triangulation with holes that has perimeter $\mathbf{p}$ and has no holes.
\end{definition1}

One can think of a triangulation with holes as a neighborhood of the root in a triangulation of the $\mathbf{p}$-gon (see Figure~\ref{triangulation_multi_polygone}). The holes are faces that can be filled by other triangulations of multi-polygons.

\begin{figure}[H]
    \centering
    \includegraphics[scale = 0.14]{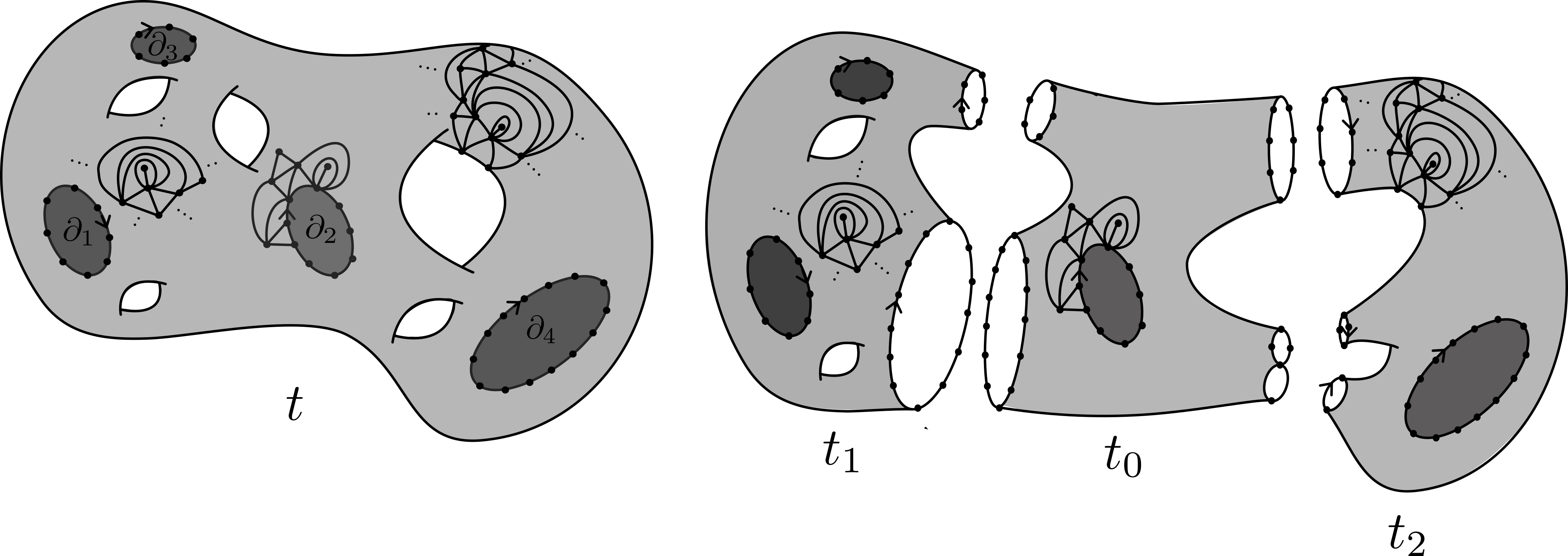}
    \caption{Left: a triangulation of the $(6,8,9,14)$-gon of genus~$5$ with $n$ vertices. 
    Right: a triangulation with holes $t_0 \subset t$, which has $1$ boundary and $5$ holes. Triangulations of multi-polygons $t_1$ and $t_2$ are the connected components of $t \backslash t_0$. The triangulation $t_1$ (resp. $t_2$) is a triangulation of the $(6,8,5,13)$-gon (resp. $(14,2,4,7)$-gon).}
    \label{triangulation_multi_polygone}
\end{figure}

For $\mathbf{p}= (p_1,\dots,p_{\ell}) \in \mathbb{N}^{\ell}$, we denote by $\mathcal{T}_{\mathbf{p}}(n,g)$ the set of triangulations of the $\mathbf{p}$-gon of genus~$g$ with $
2n-\sum_{i=1}^{\ell}(p_i-2)
$ triangles, and by $\tau_{\mathbf{p}}(n,g)$ its cardinality. One can verify that a triangulation of the $\mathbf{p}$-gon has $n+2-2g$ vertices.

\medskip
In the following, we shall introduce a notion of local distance between maps. It is therefore convenient to work with maps having a single root edge, unlike elements of $\mathcal{T}_{\mathbf{p}}(n,g)$. For $t \in \mathcal{T}_{\mathbf{p}}(n,g)$ and $e$ an oriented edge of $t$, we write $(t,e)$ for the map obtained by forgetting the distinguished edges of $t$ and taking $e$ as the root. We denote by $\mathcal{T}^1_{\mathbf{p}}(n,g)$ the set of all such rooted maps, and we define
\[
\mathcal{T}^{1} = \bigsqcup_{\substack{g,n,\ell\ge 0\\p_1,\dots,p_{\ell}\ge 1}}\mathcal{T}^1_{p_1,\dots,p_{\ell}}(n,g),
\]
the disjoint union over all parameters.

If $(t_0,e_0)$ is a rooted triangulation with holes and $(t,e)$ a rooted triangulation of the $\mathbf{p}$-gon, we write $(t_0,e_0)\subset (t,e)$ if $(t,e)$ can be obtained\footnote{ Note that by the last item in Definition~\ref{triangulation_with_holes}, if $(t_0,e_0)\subset (t,e)$, there is a unique way to obtain $(t,e)$ by filling the holes of $(t_0,e_0)$} from $(t_0,e_0)$ by gluing one or several triangulations $t_1,\dots,t_k$ of multi-polygons to the holes of $t_0$ (see Figure~\ref{triangulation_multi_polygone}). We write $(t,e) \backslash (t_0,e_0)$ for the collection of triangulations of multi-polygons that have to be glued to $t_0$ to obtain $(t,e)$. If $(t,e)$ is an infinite rooted triangulation of the $\mathbf{p}$-gon, we say that it is \emph{one-ended} if, for every $(t_0,e_0) \subset (t,e)$ with a finite number of internal faces, only one connected component of $(t,e)\backslash (t_0,e_0)$ is infinite. We say that $(t,e)$ is \emph{planar} if, for every $(t_0,e_0) \subset (t,e)$ with a finite number of internal faces, the triangulation with holes $(t_0,e_0)$ is planar. For two vertices $x,y$ in a triangulation with holes $t$, we write $d_{t}(x,y)$ for the graph distance between $x$ and $y$ in $t$.

\medskip
If $t$ is a triangulation of a multi-polygon, it is more convenient to adopt the convention that the external faces of $t$ do not belong to $t^{*}$. Since the external faces of a triangulation of a multi-polygon are simple and disjoint, the dual map $t^{*}$ is connected. We denote by $d^{*}$ the graph distance on the internal faces of $t$. We extend $d^{*}_t$ to the set of edges by writing for $e,e'$ two distinct edges
\[
d^{*}_t(e,e') = \min_{f,f'} d^{*}_t(f,f')+1,
\]
where $f,f'$ run over all internal faces incident to $e$ and $e'$. We also extend $d^{*}_t$ to the set of vertices by writing, for  $v,v'$ two distinct vertices
\[
d^{*}_t(v,v') = \min_{f,f'} d^{*}_t(f,f') + 1,
\]
where $f,f'$ run over all internal faces incident to $v$ and $v'$.

\subsection{Combinatorics}\label{combi}
In this section, we recall some combinatorial estimates in the planar case.  
For $p \ge 1$ and $n \ge p-2$, a formula is given in \cite{krikun2007explicitenumerationtriangulationsmultiple} for the volumes $\tau_p(n,0)$:
\begin{align*}
\tau_p(n,0) = \frac{p(2p)!}{(p!)^2}\frac{4^{n-p+1}(3n-p+1)!!}{(n-p+2)!(n+p+1)!!} 
\underset{n \to +\infty}{\sim} c(p)\lambda_c^{-n}n^{-\frac{5}{2}},
\end{align*}
where $\lambda_c = \frac{1}{12\sqrt{3}}$, and
\begin{align*}
c(p) = \lambda_c^{2-p}\frac{3^{p-2}p(2p)!}{4\sqrt{2 \pi}(p!)^2} 
\underset{p \to +\infty}{\sim}\frac{1}{36\pi \sqrt{2}}12^{p}\sqrt{p}.
\end{align*}
Note that $\tau_p(n-2+p,0)$ denotes the number of triangulations of the $p$-gon with $n$ internal vertices.
For $p \ge 1$ and $\lambda > 0$, we define 
\[
w_{\lambda}(p) = \sum_{n \ge 0} \tau_p(n-2+p,0)\lambda^{n}.
\]
 This quantity is finite if and only if $\lambda \le \lambda_c$. We also define 
\[
W_{\lambda}(x) = \sum_{p \ge 1} w_{\lambda}(p)x^p.
\]
Then, from Equation~(4) in \cite{krikun2007explicitenumerationtriangulationsmultiple}, we obtain the explicit formula:
\begin{align}\label{formula_W}
W_{\lambda}(x) = \frac{\lambda}{2}\left(\left(1 - \frac{1+8h}{h}x\right)\sqrt{1-4(1+8h)x} - 1 + \frac{x}{\lambda}\right),
\end{align}
where $h \in (0, \frac{1}{4}]$ is given by $\lambda = \frac{h}{(1+8h)^{3/2}}$.  
From this expression, one can derive the formula
\[
w_{\lambda}(1) = \frac{1}{2} - \frac{1+2h}{2\sqrt{1+8h}},
\]
and for $p \ge 2$,
\begin{align}\label{formula_wlambda}
w_{\lambda}(p) = (2+16h)^p \frac{(2p-5)!!}{p!}\frac{((1-4h)p+6h)}{4(1+8h)^{3/2}}.
\end{align}
  
We then define the \emph{Boltzmann triangulation of the $p$-gon} (with parameter $\lambda$) as the random planar triangulation $T^{(p)}_{\lambda}$ that takes the value $t$ with probability $\frac{\lambda^{|t_{\mathrm{in}}|}}{w_{\lambda}(p)}$, where $|t_{\mathrm{in}}|$ denotes the number of vertices of $t$ that do not lie on the boundary face.

\subsection{Local convergence and dual local convergence}\label{subsection_local_convergence}

In this section, we introduce two notions of local distance between maps in $\mathcal{T}^1$.  
The first one is the most important, since Theorems~\ref{local_limit_boundary} and~\ref{local_limit_middle} are stated for this distance.  
However, in Section~\ref{section_middle}, it is sometimes more convenient to use an alternative notion of distance.

\paragraph*{Local distance.}
For $r \ge 0$ and $(t,e) \in \mathcal{T}^1$, we define the ball of radius $r$, denoted $B_r(t,e)$, as the submap consisting of the edges having at least one endpoint at graph distance at most $r-1$ from the origin of $e$ (see Figure~\ref{ball_local}).  
We then introduce the \emph{local distance} on $\mathcal{T}^1$ by
\begin{align}\label{distance_locale}
d_{\mathrm{loc}}((t,e),(t',e')) = (1 + \max \{r \ge 0 : B_r(t,e) = B_r(t',e')\})^{-1}.
\end{align}

The definition used for $B_r(t,e)$ here is not the usual one. Indeed, the most natural definition would be to take all faces incident to a vertex at distance at most $r-1$ from the origin of $e$. However, with this choice, if the root edge $e$ lies on the boundary $\partial_1$, then we would have $\partial_1 \subset B_1(t,e)$. Consequently, as soon as $p^n_1 \underset{n\to+\infty}{\longrightarrow} +\infty$, one cannot expect any local convergence result.

A drawback of this definition is that, in general, we do not have $B_r(t,e) \subset (t,e)$, since the ball $B_r(t,e)$ might not be a triangulation with holes.  
Its completion, denoted $\overline{\mathcal{T}^1}$, is a Polish space that can be viewed as the set of rooted (finite or infinite) triangulations of a multi-polygon (with a finite or infinite number of boundaries) in which all vertices have finite degree.  
Note that $p$ can be infinite, and this space is not compact.

\paragraph*{Dual local distance.}
For $r \ge 0$ and $(t,e) \in \mathcal{T}^1$, we define the dual ball of radius $r$, denoted $B_r^{*}(t,e)$. We define $B_r^{*}(t,e)$ as the triangulation with holes rooted at $e$, obtained as follows (see Figure~\ref{ball_local}):
\begin{itemize}
    \item[$\bullet$] Take all internal faces $f'$ at dual distance at most $r$ from an internal face incident to $e$.
    \item[$\bullet$] Add the boundaries sharing a vertex with one of the faces already included.
\end{itemize}
For any $r \ge 0$, we always have $B^{*}_r(t,e) \subset t$.  
We then define the \emph{local dual distance} on $\mathcal{T}^1$ by
\begin{align}\label{distance_dual_locale}
d^{*}_{\mathrm{loc}}((t,e),(t',e')) = (1 + \max \{r \ge 0 : B^{*}_r(t,e) = B^{*}_r(t',e')\})^{-1}.
\end{align}

Let $\overline{\mathcal{T}^1}^{*}$ denote the completion of $\mathcal{T}^1$ under $d^{*}_{\mathrm{loc}}$.  
The set $\overline{\mathcal{T}^1}^{*}$ is not compact under $d^{*}_{\mathrm{loc}}$, since $B^{*}_r(t,e)$ may touch boundaries and thus may contain arbitrarily large faces.

However, consider a sequence $(T_n,e_n)_{n \ge 0}$ of random rooted triangulations of a multi-polygon, and let $\rho_n$ be the origin of $e_n$.  
Suppose that for any $r \ge 0$, we have $\mathbb{P}(d_{T_n}^{*}(\rho_n,\partial T_n) \le r ) \to 0$, i.e., the root edge typically lies far from the boundaries.  
Then, under the event $\{d_{T_n}(\rho_n,\partial T_n) > r\}$, the dual ball $B^{*}_r(t_n,e_n)$ of radius $r$, has at most $3 \cdot 2^r$ triangles so it can take only finitely many values.  
We deduce that the sequence $(T_n,e_n)$ is tight for $d^{*}_{\mathrm{loc}}$.  
This argument will later be used in Section~\ref{section_middle} to establish tightness easily.  
Indeed, in Section~\ref{section_middle}, we consider $T_n = T_{n,g_n,\mathbf{p}^{n}}$ with $|\mathbf{p}^{n}| = o(n)$ and $e^n$ a uniformly chosen oriented edge on $T_n$.  
Hence, the edge $e^n$ typically lies far from the boundaries.  
In the regime studied in \cite{Budzinski_2020}, the authors consider the case $\partial T_n = \emptyset$, so the condition is always satisfied, and tightness follows automatically under this topology.

\begin{figure}[H]
\centering
\includegraphics[scale=0.2]{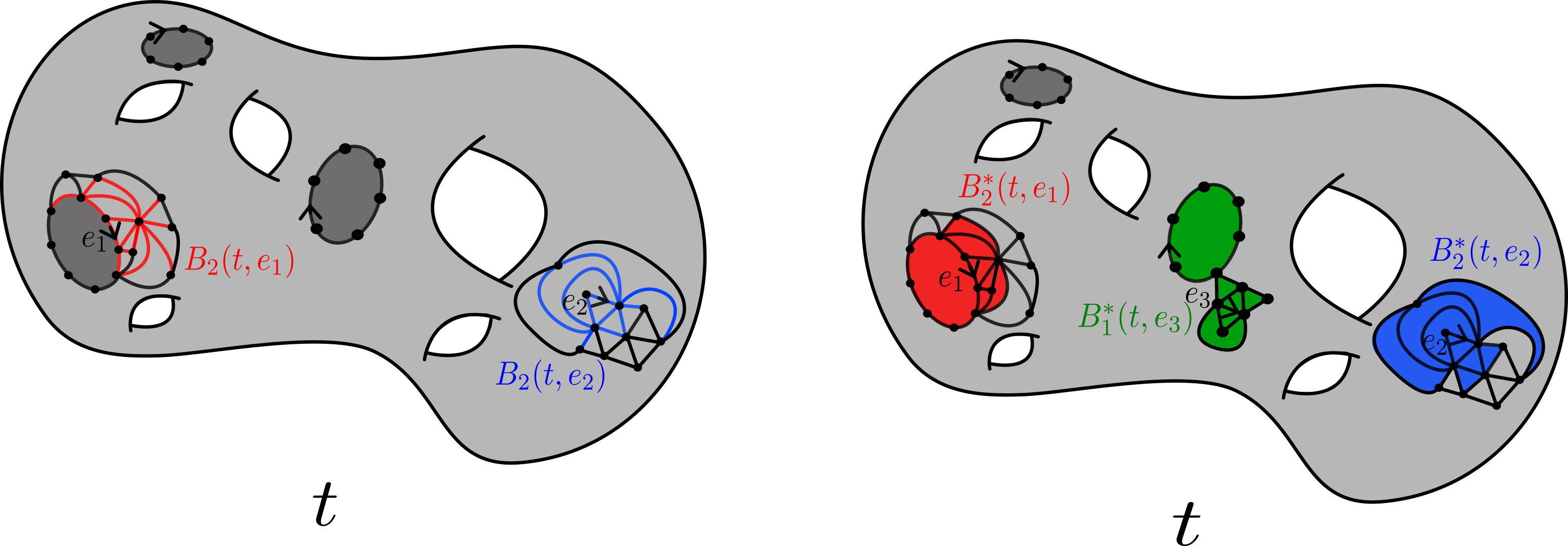}
\caption{We represent a triangulation $t$ of the $\mathbf{p}$-gon. On the left, we show the ball of radius $2$ centered at $e_1$, which lies on a boundary, and at $e_2$, which lies far from the boundaries. On the right, we show the dual local ball of radius $2$ centered at $e_1$, $e_2$, and $e_3$. Note that $B_2^{*}(t,e_1)$ contains the boundary on which $e_1$ lies. Moreover, in $B^{*}_1(t,e_3)$, one of the green internal faces shares a vertex with a boundary face, and thus, this boundary face also belongs to $B^{*}_2(t,e)$.}
\label{ball_local}
\end{figure}

\subsection{Triangulations of the plane}\label{triangulation_of_the_plane}
In this section, we recall the definition of type-I Planar Stochastic Hyperbolic Triangulations (PSHT). They were first defined in \cite{PSHT} for type-II triangulations, and later extended to the type-I setting in \cite{PSHTtypeI}. The interested reader may read \cite[Section 8]{Stflour} for a construction of these random maps in a general setting. The PSHT form a one-parameter family $(\mathbb{T}_{\lambda})_{0 < \lambda \le \lambda_c}$, where $\lambda_c= \frac{1}{12\sqrt{3}}$. They are random triangulations of the plane characterized by a Markov property. For any triangulation $t$ with one hole of perimeter $p$ and $v$ vertices in total, we have
\begin{align*}
    \Pf(t\subset \mathbb{T}_{\lambda}) = C_p(\lambda)\lambda^v.
\end{align*}
Note that $\mathbb{T}_{\lambda_{c}}$ has the law of the Uniform Infinite Planar Triangulation (UIPT). For $\lambda \in (0,\lambda_c)$, let us define $h\in (0,\frac{1}{4}]$ as the unique solution of the equation
\begin{align*}
    \lambda  = \frac{h}{(1+8h)^{\frac{3}{2}}}.
\end{align*}
We have
\begin{align*}
    C_p(\lambda) = \frac{1}{\lambda}\bigg(8+\frac{1}{h}\bigg)^{p-1}\sum_{q=0}^{p-1}\binom{2q}{q}h^q.
\end{align*}
For any $p \ge 0$, we introduce the extension $\mathbb{T}_{\lambda}^{(p)}$ of $\mathbb{T}_{\lambda}$ to the finite boundary case. Indeed, $\mathbb{T}_{\lambda}^{(p)}$ is a random infinite, one-ended, planar triangulation of the $p$-gon characterized by the probabilities
\begin{align*}
    \Pf(t\subset \mathbb{T}_{\lambda}^{(p)}) = \frac{C_{q}(\lambda)}{C_{p}(\lambda)}\lambda^v,
\end{align*}
for any triangulation $t$ with a boundary of perimeter $p$ and one hole of perimeter $q$, where $v$ denotes the total number of vertices of $t$ that do not lie on the boundary. Note that $\mathbb{T}_{\lambda}$ can be identified with $\mathbb{T}^{(1)}_{\lambda}$. Indeed, given an infinite triangulation of the plane $T$, one can split the root edge into a digon and add a loop inside. Then, rerooting on the loop, one obtains a triangulation $T'$ of the plane with boundary of perimeter $1$. This procedure can be reversed, yielding a bijection between triangulations of the plane and those with a boundary of perimeter $1$.\\ 
Note that the triangulations $\mathbb{T}_{\lambda}^{(p)}$ have a nice spatial Markov property since $\mathbb{T}_{\lambda}^{(p)} \backslash t$ has the law of $\mathbb{T}_{\lambda}^{(q)}$. This allows one to explore $\mathbb{T}_{\lambda}^{(p)}$ in a Markovian way and motivates the question of what happens when $p \to +\infty$.

We finally introduce two degenerate infinite planar triangulations. These are deterministic infinite planar triangulations with infinite vertex degrees. The planar triangulation $\mathbb{T}_0$ is obtained by gluing the edges of triangles along the structure of an infinite binary tree. We also define $\mathbb{T}_{\star}$, the planar triangulation obtained by starting with a triangle consisting of three loops and then recursively adding two loops in each loop (see Figure~\ref{degenerate}).
\begin{figure}[H]
\centering
\includegraphics[scale=0.6]{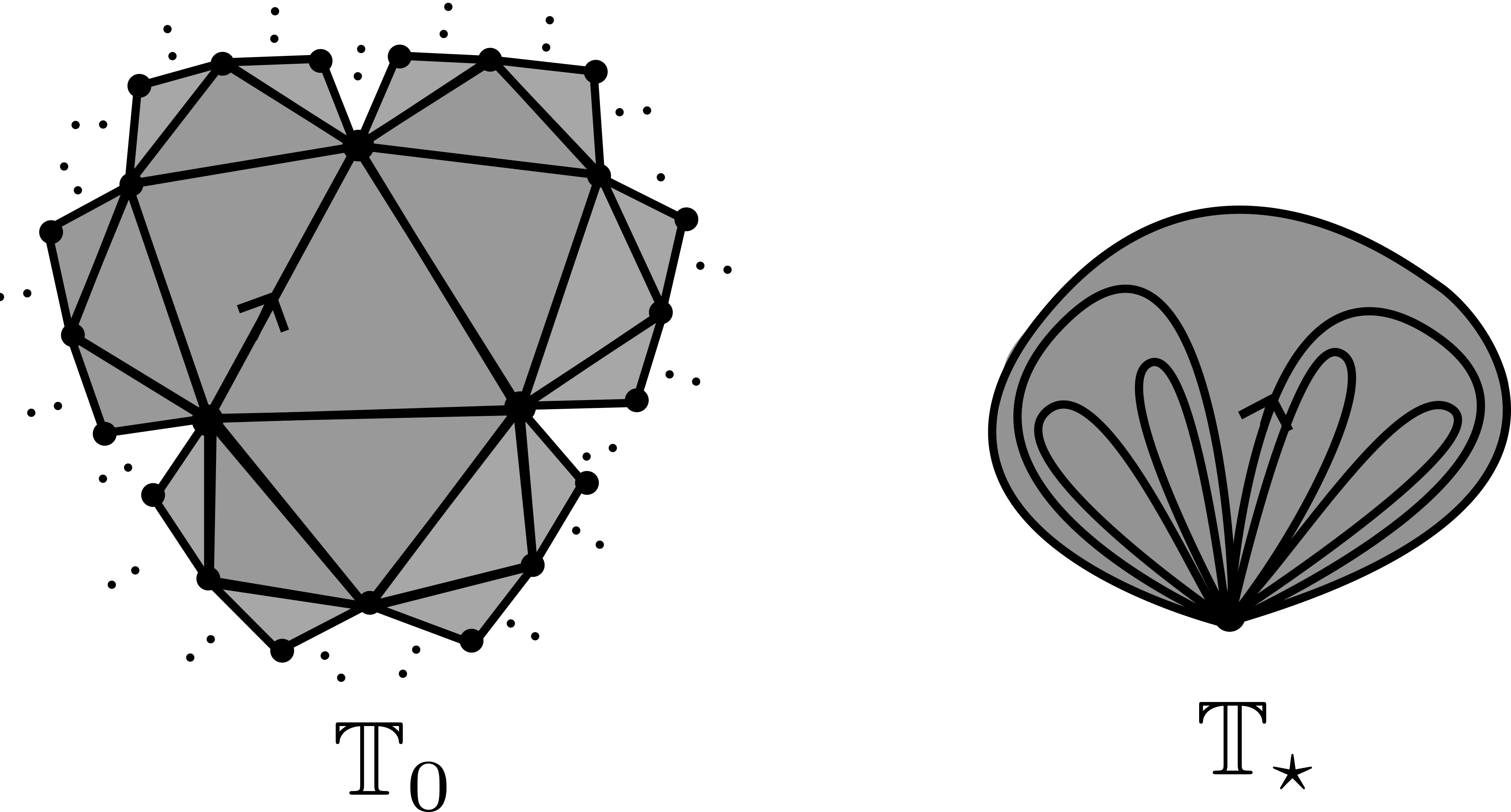}
\caption{On the left, the triangulation $\mathbb{T}_0$. On the right, the triangulation $\mathbb{T}_{\star}$.}
\label{degenerate}
\end{figure}

These two degenerate definitions are motivated by \cite[Theorem~1]{budzinski2022multiendedmarkoviantriangulationsrobust}, which states that any infinite random planar triangulation of the plane satisfying a "reasonable" Markov property is a mixture of $(\mathbb{T}_{\lambda})_{0 < \lambda < \lambda_c}$, $\mathbb{T}_0$, and $\mathbb{T}_{\star}$. 
\begin{definition1}
Let $T$ be a random infinite planar triangulation. We say that $T$ is weakly Markovian if there exist numbers $a_v^{p_1,\dots,p_k}$ for $k \ge 1$ and $p_1,\dots,p_k \ge 1$, $v \ge 0$, such that for any triangulation $t$ with $k$ holes of perimeters $p_1,\dots,p_k$ and $v$ internal vertices,
\begin{align*}
\Pf(t \subset T) = a_v^{p_1,\dots,p_k}.
\end{align*}
\end{definition1}

Then, we state \cite[Theorem $1$]{budzinski2022multiendedmarkoviantriangulationsrobust}
. Note that we do not assume the triangulations to be one-ended here.
\begin{theorem}\label{weakly_markovian_mixture}
Let $T$ be a random infinite planar triangulation that is weakly Markovian. Then, $T$ is of the form $\mathbb{T}_{\Lambda}$ where $\Lambda$ is a random variable taking values in $[0,\lambda_c]\cup\{\star\}$.
\end{theorem}

\subsection{Triangulations of the half-plane}\label{halfplane_triangulation}

A \emph{triangulation of the half-plane} is a planar triangulation of the $\infty$-gon that is one-ended and has vertices of finite degree. For a triangulation with holes $t$ with exactly one infinite boundary $\partial t$ and one infinite hole $\partial^{*}t$ (see Figure~\ref{halfplane}), we use the shorthand notation
\[
|\partial^{*}t| - |\partial t| := |\partial^{*}t \backslash \partial t| - |\partial t \backslash \partial^{*}t|.
\]
We also denote by $|t_{\mathrm{in}}|$ the number of vertices of $t$ that do not lie on $\partial t$.

\begin{definition1}\label{weak_Markovian_halfplane}
A random triangulation of the half-plane $H$ is said to satisfy a weak spatial Markov property if there exists a function $a$ such that, for any $t$, we have
\[
\Pf(t \subset H) = a(|\partial^{*}t| - |\partial t|, |t_{\mathrm{in}}|).
\]
We say that $H$ is strongly Markovian if there exist constants $\beta,\lambda > 0$ such that, for any $t$, we have 
\[
\Pf(t \subset H) = \beta^{|\partial^{*}t| - |\partial t|}\lambda^{|t_{\mathrm{in}}|}.
\]
\end{definition1}

For $0 < \lambda \le \lambda_c$, we introduce the type-I random triangulation of the half-plane, denoted by $\mathbb{H}_{\lambda}$, introduced in \cite{budzinski_geodesics}. These triangulations are characterized by
\begin{align}\label{characterize_halfplane}
\Pf(t \subset \mathbb{H}_{\lambda}) = \bigg(8+\frac{1}{h}\bigg)^{|\partial^{*}t| - |\partial t|}\lambda^{v_{\mathrm{in}}}.
\end{align}
In particular, the half-plane triangulation $\mathbb{H}_{\lambda}$ is strongly Markovian. Note that for any $t$ we have $\mathbb{H}_{\lambda} \backslash t \overset{(d)}{=} \mathbb{H}_{\lambda}$, which corresponds to the Markov property considered in \cite{Angel_Ray}. Moreover, the triangulation of the half-plane $\mathbb{H}_{\lambda}$ is characterized by the local limit
\begin{align}\label{plane_to_halfplane}
\mathbb{T}_{\lambda}^{(p)} \underset{p \to +\infty}{\longrightarrow} \mathbb{H}_{\lambda}.
\end{align}

In \cite{Angel_Ray}, the authors introduced a one-parameter family $(H_{\alpha})_{0 \le \alpha < 1}$ of type-$II$ triangulations of the half-plane that describes all translation-invariant random triangulations of the half-plane satisfying a spatial Markov property. The Markov property considered there states that for any $t$, conditionally on $t \subset H_{\alpha}$, we have $H_{\alpha}\backslash t \overset{(d)}{=} H_{\alpha}$. This result holds for type-II triangulations. Since we work in the type-I setting, there exists a larger class of triangulations satisfying such a Markov property. However, the formula~\eqref{characterize_halfplane} defining $\mathbb{H}_{\lambda}$ above is a stronger assumption. Thus, we recover a one-parameter family of triangulations of the half-plane satisfying this condition (see Proposition~\ref{classification_halfplane} below).

Let us now give a complete description of half-plane triangulations that are weakly or strongly Markovian. This Proposition is the analogue of \cite[Proposition 15]{Budzinski_2020}

\begin{proposition}\label{classification_halfplane}
\begin{enumerate}
\item \label{it1half} For any $0 \le \lambda \le \lambda_c$, there exists a unique random triangulation of the half-plane\footnote{Here $\widetilde{\mathbb{H}}_{\lambda}$ is the analogue of $H_{\alpha}$ for $\alpha \le \frac{2}{3}$, and $\mathbb{H}_{\lambda}$ is the analogue of $H_{\alpha}$ for $\alpha \ge \frac{2}{3}$.} $\widetilde{\mathbb{H}}_{\lambda}$ characterized by the probabilities
\[
\Pf(t \subset \widetilde{\mathbb{H}}_{\lambda}) = \big(32h+4\big)^{|\partial^{*}t| - |\partial t|}\lambda^{v_{\mathrm{in}}}.
\]
In particular, we have $\widetilde{\mathbb{H}}_{\lambda_c} = \mathbb{H}_{\lambda_c}$.
\item \label{it3half} Let $H$ be a strongly Markovian random triangulation of the half-plane. Then $H$ is either of the form $\mathbb{H}_{\lambda}$ for some $0 < \lambda \le \lambda_c$ or $\widetilde{\mathbb{H}}_{\lambda}$ for some $0 \le \lambda \le \lambda_c$.

\item \label{it2half} Let $H$ be a weakly Markovian random triangulation of the half-plane. Then $H$ is a mixture of $\{\mathbb{H}_{\lambda}\}_{0 < \lambda \le \lambda_c} \cup \{\widetilde{\mathbb{H}}_{\lambda}\}_{0 \le \lambda \le \lambda_c}$.

\end{enumerate}
\end{proposition}

\begin{proof}
Let us start by proving Item~\ref{it1half}. The proof is standard and follows the lines of \cite{Angel_Ray}. We define $\widetilde{\mathbb{H}}_{\lambda}$ using a peeling procedure. Indeed, if $\widetilde{\mathbb{H}}_{\lambda}$ exists, let $f_0$ be the triangle incident to the root edge of $\widetilde{\mathbb{H}}_{\lambda}$. If the third vertex of $f_0$ does not lie on the boundary, we say that \emph{Case~I} occurs. For $i \ge 0$, if the third vertex lies $i$ vertices to the left (resp.\ $i+1$ to the right) of the root vertex, we say that \emph{Case~$\text{II}_i$} occurs (resp.\ \emph{Case~$\text{III}_i$}). Note that in Cases~$\text{II}_i$ and~$\text{III}_i$, a hole of perimeter $i+1$ is created, which is filled with an independent Boltzmann triangulation of the $(i+1)$-gon with parameter $\lambda$.

\begin{figure}[H]
\centering
\includegraphics[scale=0.6]{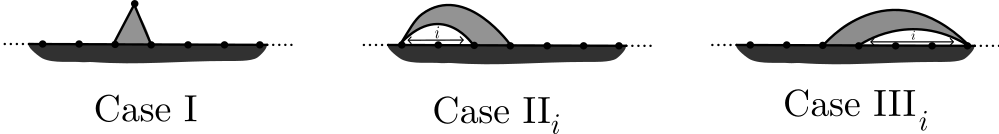}
\caption{The three possible cases.}
\label{three_cases}
\end{figure}

We can write
\[
\Pf(\text{Case I occurs}) = (32h+4)\lambda.
\]
For $i \ge 0$, we also have
\[
\Pf(\text{Case II}_i \text{ occurs}) = \Pf(\text{Case III}_i \text{ occurs}) = (32h+4)^{-i}w_{\lambda}(i+1).
\]
Summing these probabilities gives
\[
(32h+4)\lambda + 2\sum_{i\ge 0}(32h+4)^{-i}w_{\lambda}(i+1) = (32h+4)\lambda + 2(32h+4)W_{\lambda}\big((32h+4)^{-1}\big) = 1.
\]
by~\eqref{formula_W}. Since these probabilities sum to $1$, we can construct $\widetilde{\mathbb{H}}_{\lambda}$ by peeling with these transitions. When a peeling step of type $\text{II}_i$ or $\text{III}_i$ occurs, we fill the newly created hole with a Boltzmann triangulation of the $(i+1)$-gon with parameter $\lambda$. See \cite{Angel_Ray,PSHT} for similar constructions.\\

We now prove Item~\ref{it2half}. Fix a weakly Markovian triangulation of the half-plane $H$. Let $X = \{(p,v)\in \mathbb{Z}^2 : v \ge \max(0,p)\}$. For $(p,v)\in X$, let $t_{p,v}$ be a triangulation with one infinite boundary, one infinite hole such that $|\partial^{*}t_{p,v}| - |\partial t_{p,v}| = p$ and exactly $v$ vertices not lying on the boundary. Since $H$ is weakly Markovian, define 
\[
a_{p,v} = \Pf(t_{p,v}\subset H).
\]
We also introduce 
\[
b_{p,v} = \Pf(t_{p,v}\subset \mathbb{H}_{\lambda_c}) = 12^{-p}\bigg(\frac{1}{12\sqrt{3}}\bigg)^{v},
\]
and define
\[
h(p,v) = \frac{a_{p,v}}{b_{p,v}}.
\]

Let $(T_k)_{k\ge 0}$ be a peeling exploration of $\mathbb{H}_{\lambda_c}$ for an arbitrary peeling algorithm $\mathcal{A}$ (see Section~\ref{peeling_def}). Let $(P_k,V_k)_{k \ge 0}$ denote the process where $P_k := |\partial^{*}T_k| - |\partial T_k|$ and $V_k$ is the number of vertices of $T_k$ not on the boundary. Since we consider the evolution of the volume and perimeter of $(T_k)_{k\ge 0}$, the distribution of this process is independent of the choice of $\mathcal{A}$.  Since $\mathbb{H}_{\lambda_c}\backslash T_k$ is distributed as $\mathbb{H}_{\lambda_c}$, the process $(P_k,V_k)_{k \ge 0}$ is a random walk starting from $(0,0)$ with i.i.d.\ increments in $\mathbb{Z}\times \mathbb{N}$. The transition probabilities are given by 
\[
\begin{cases}
Q\big((p,v),(p+1,v+1)\big) = \frac{1}{\sqrt{3}},\\[4pt]
Q\big((p,v),(p-k,v+m)\big) = 2\cdot 12^{k}\bigg(\frac{1}{12\sqrt{3}}\bigg)^{m}, \quad \forall (k,m)\in \mathbb{N}\times \mathbb{N}.
\end{cases}
\]
Thus, one can verify that $(P_k,V_k)$ takes values in $X$. We claim that $h$ is $Q$-harmonic. Indeed, conditionally on $t_{p,v}\subset H$, we can fix an edge on the hole of $t_{p,v}$ and reveal the triangle incident to this edge. Filling any finite hole created yields
\begin{align}\label{peeling_equa1}
a_{p,v} = a_{p+1,v+1} + 2\sum_{m,k \ge 0} a_{p-k,v+m}\tau_{k+1}(k+m-1,0).
\end{align}
Dividing by $b_{p,v}$ gives
\[
h(p,v) = h(p+1,v+1) Q\big((p,v),(p+1,v+1)\big) + \sum_{m,k \ge 0} Q\big((p,v),(p-k,v+m)\big) h(p-k,v+m),
\]
showing that $h$ is $Q$-harmonic.

Our main tool is the classification of $Q$-harmonic functions obtained in Corollary~\ref{harmonic}. Let $h_0$ be the function on $X$ defined by $h_0(p,v) = \mathbf{1}_{\{p=v\}} Q((0,0),(1,1))^{-v}$. By Corollary~\ref{harmonic}, there exists a measure $\Lambda_0$ on $(\mathbb{R}_{+})^2$ and $\gamma \ge 0$ such that for all $(p,v) \in X$,
\[
h(p,v) = \int_{(\mathbb{R}_{+})^2} \eta^{p}\alpha^{v}\, \Lambda_0(d\eta,d\alpha) + \gamma\, h_0(p,v).
\]
Multiplying by $b_{p,v}$ gives
\[
a_{p,v} = \int_{(\mathbb{R}_{+})^2} \bigg(\frac{\eta}{12}\bigg)^{p}\bigg(\frac{\alpha}{12\sqrt{3}}\bigg)^{v} \Lambda_0(d\eta,d\alpha) + \gamma\, b_{p,v} h_0(p,v).
\]
Letting $\Lambda$ be the pushforward of $\Lambda_0$ by $(\eta,\alpha)\mapsto(\frac{\eta}{12},\frac{\alpha}{12\sqrt{3}})$, we have
\[
a_{p,v} = \int_{(\mathbb{R}_{+})^2} \beta^{p}\lambda^{v}\, \Lambda(d\beta,d\lambda) + \gamma\, b_{p,v} h_0(p,v).
\]

We now show that $\gamma = 0$. First note that
\[
b_{p,v} h_0(p,v) = \mathbf{1}_{\{p=v\}} Q(1,1)^{-v} Q(1,1)^v = \mathbf{1}_{\{p=v\}}.
\]
Choosing $t_{v,v}$ a triangulation with a hole in which the root vertex has degree $2+v$, we deduce from the fact that the root edge of $H$ almost surely has finite degree that
\[
0 \le \gamma \le \Pf(t_{v,v} \subset H) \xrightarrow[v\to\infty]{} 0,
\]
and hence $\gamma=0$.

Rewriting \eqref{peeling_equa1} with $p=v=0$, we find that $\Lambda$ is a probability measure satisfying
\[
1 = a_{0,0} = \int_{(\mathbb{R}_+)^2} \bigg(\beta \lambda + 2 \sum_{m,k \ge 0}\lambda^{m}\beta^{-k}\tau_{k+1}(k+m-1,0)\bigg)\Lambda(d\beta,d\lambda).
\]
Using \eqref{formula_W}, the term in parentheses can be rewritten as $\beta \lambda + 2 \beta W_{\lambda}(\beta^{-1})$. The function $W_{\lambda}(\beta^{-1})$ is finite if and only if $\lambda \le \lambda_c$ and $\beta \in [32h+4,8+\frac{1}{h}]$. For $0 \le \lambda \le \lambda_c$ and $\beta \in [32h+4,8+\frac{1}{h}]$ (if $\lambda = 0$, then $\beta \in [4,+\infty)$), we have
\begin{align}\label{relation}
\beta \lambda + 2 \beta W_{\lambda}(\beta^{-1}) = 1 + \lambda \beta \bigg(1-\frac{1+8h}{h}\beta^{-1}\bigg)\sqrt{1-4(1+8h)\beta^{-1}}.
\end{align}
Since the right-hand side is negative, we deduce that $\Lambda$ is supported on 
\[
\{(\lambda,\beta) : 0 < \lambda \le \lambda_c,\, \beta \in \{32h+4, (1+8h)/h\} \} \cup \{(0,4)\}.
\]
This concludes the proof of Item~\ref{it2half}.

Item~\ref{it3half} follows from Item~\ref{it2half}. Indeed, fix a strongly Markovian triangulation $H$ of the half-plane. Then, there exist $\lambda_1,\beta_1 \ge 0$ such that for any $(p,v)\in X$,
\[
\Pf(t_{p,v} \subset H) = \lambda_1^v \beta_1^p.
\]
Assuming $\lambda_1 > 0$ (the case $\lambda_1=0$ being analogous), applying \eqref{peeling_equa1} with $p=v=0$ gives
\[
1 = \beta_1 \lambda_1 + 2 \beta_1 W_{\lambda_1}(\beta_1^{-1}).
\]
The right-hand side is finite if and only if $\lambda_1 \le \lambda_c$ and $\beta_1 \in [32h+4, (1+8h)/h]$. In that case,
\[
1 = 1 + \lambda_1 \beta_1 \bigg(1 - \frac{1+8h}{h}\beta_1^{-1}\bigg)\sqrt{1 - 4(1+8h)\beta_1^{-1}},
\]
which implies $\beta_1 \in \{32h+4, (1+8h)/h\}$. This concludes the proof.
\end{proof}

\subsection{Peeling explorations and peeling diagrams}\label{peeling_def}
In this section, we define \emph{peeling explorations}. This method allows discovering a triangulation, triangle by triangle, in a Markovian fashion. We also introduce the \emph{peeling diagram} associated with such an exploration. It provides a convenient graphical representation of the peeling process. This notion is strongly inspired by \cite{Contat2022LastCD}, which we extend to a general genus setting.

\begin{definition1}\label{peeling_algorithm}
A peeling algorithm $\mathcal{A}$ is a function that, given a triangulation with holes $t$, associates one edge $\mathcal{A}(t)$ on one of the holes of $t$.
\end{definition1}

Fix $n,g \ge 0$ and $\mathbf{p} = (p_1,\dots,p_{\ell}) \in \mathbb{N}^{\ell}$.  
Also fix $t \in \mathcal{T}_{(p_1,\dots,p_{\ell})}(n,g)$ and a peeling algorithm $\mathcal{A}$. We define by induction a sequence $(\mathcal{E}_k^{\mathcal{A}}(t))_{k \ge 0}$ of triangulations with holes and an increasing sequence $(\mathcal{D}_k^{\mathcal{A}}(t))_{k \ge 0}$ of diagrams. We define $\mathcal{E}_0^{\mathcal{A}}(t)$ as the map consisting of a single $p_1$-gon, and $\mathcal{D}_0^{\mathcal{A}}(t)$ as the graph with one vertex labelled $p_1$.  

For $k \ge 0$, suppose we have already constructed the triangulation with holes $\mathcal{E}_k^{\mathcal{A}}(t)$ and the decorated graph $\mathcal{D}_k^{\mathcal{A}}(t)$.  
If $\mathcal{E}_k^{\mathcal{A}}(t) = t$, we set $\mathcal{E}_{k+1}^{\mathcal{A}}(t) = \mathcal{E}_k^{\mathcal{A}}(t)$ and $\mathcal{D}_{k+1}^{\mathcal{A}}(t) = \mathcal{D}_k^{\mathcal{A}}(t)$.  

Let $h_1,\dots,h_d$ be the holes of $\mathcal{E}_k^{\mathcal{A}}(t)$, and let $\ell_1,\dots,\ell_d$ be their lengths.  
We assume that the holes $h_1,\dots,h_d$ are represented in $\mathcal{D}_k^{\mathcal{A}}(t)$ by vertices $x_1,\dots,x_d$ labelled $\ell_1,\dots,\ell_d$.  
We denote by $h_i$ the hole containing the edge $\mathcal{A}(\mathcal{E}_k^{\mathcal{A}}(t))$.  
We then define $\mathcal{E}_{k+1}^{\mathcal{A}}(t)$ and $\mathcal{D}_{k+1}^{\mathcal{A}}(t)$ according to the type of triangle revealed incident to this edge.

\begin{minipage}[c]{0.4\textwidth}
\vspace{0.5em}
\textbf{Type I.} The hole $h_i$ has length $2$ ($\ell_i = 2$) and is filled with the empty map. Then $\mathcal{E}_{k+1}^{\mathcal{A}}(t)$ is obtained by identifying the two edges of $h_i$. The decorated graph $\mathcal{D}_{k+1}^{\mathcal{A}}(t)$ is obtained by adding a vertex labelled $0$ and an oriented edge from $x_i$ to this new vertex.
\vspace{0.5em}
\end{minipage}%
\hfill
\begin{minipage}[c]{0.55\textwidth}
\centering
\includegraphics[width =\linewidth]{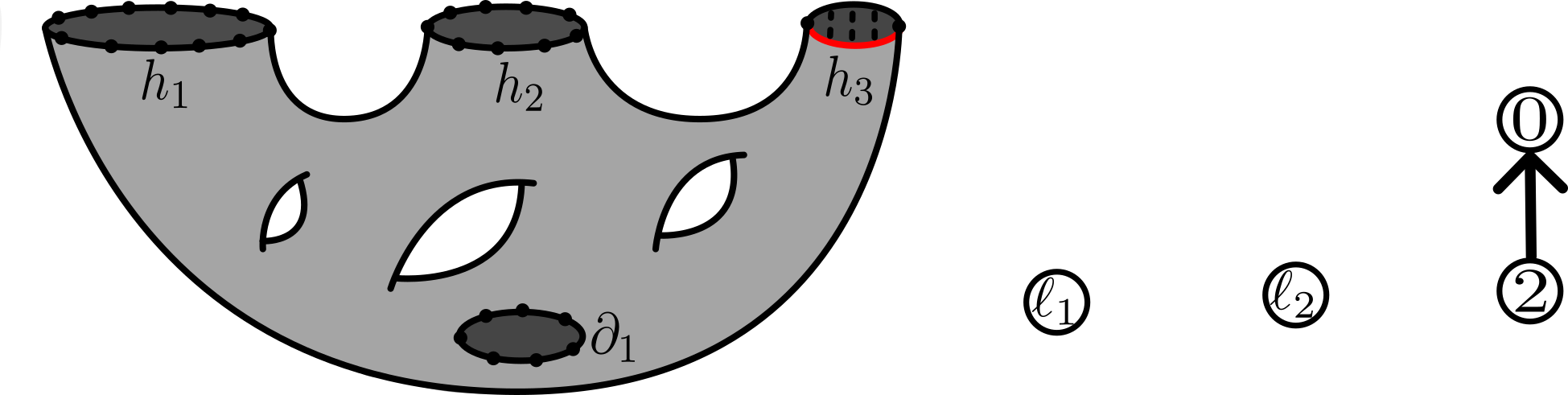}
\end{minipage}

\begin{minipage}[c]{0.4\textwidth}
\vspace{0.5em}
\textbf{Type II.} The discovered triangle has its third vertex not lying on any hole nor on any boundary $\partial_j$ of $t$. Then $\mathcal{E}_{k+1}^{\mathcal{A}}(t)$ is obtained by gluing this triangle along $\mathcal{A}(\mathcal{E}_k^{\mathcal{A}}(t))$. The decorated graph $\mathcal{D}_{k+1}^{\mathcal{A}}(t)$ is obtained by adding a vertex labelled $\ell_i + 1$ and an oriented edge from $x_i$ to this new vertex.
\vspace{0.5em}
\end{minipage}%
\hfill
\begin{minipage}[c]{0.55\textwidth}
\centering
\includegraphics[width =\linewidth]{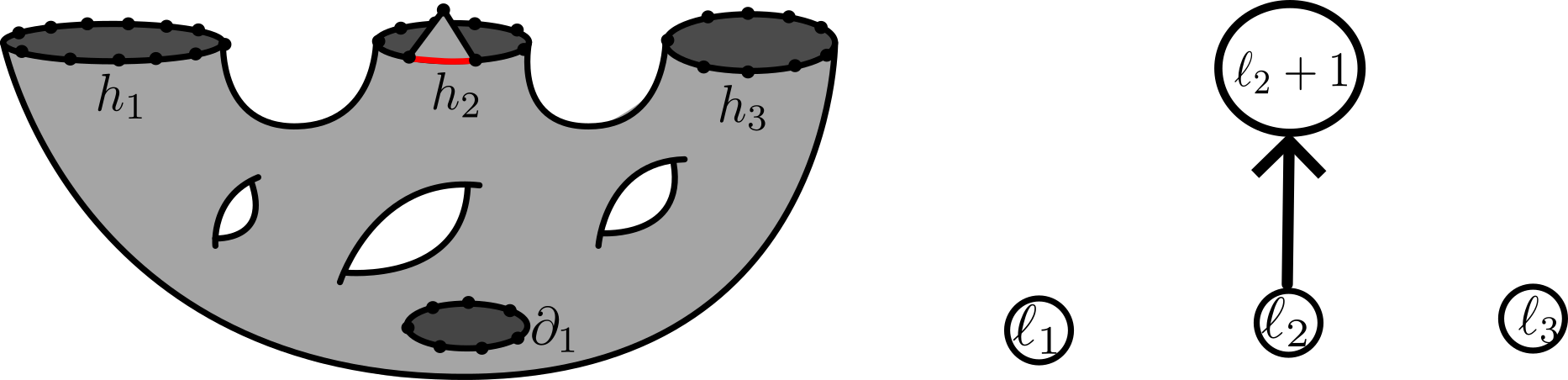}
\end{minipage}

\begin{minipage}[c]{0.4\textwidth}
\vspace{0.5em}
\textbf{Type III.} The discovered triangle has its third vertex $v$ lying on a boundary $\partial_j$ of $t$ not yet revealed in $\mathcal{E}_k^{\mathcal{A}}(t)$. Then $\mathcal{E}_{k+1}^{\mathcal{A}}(t)$ is obtained by gluing this triangle to the boundary $\partial_j$ along $\mathcal{A}(\mathcal{E}_k^{\mathcal{A}}(t))$. The decorated graph $\mathcal{D}_{k+1}^{\mathcal{A}}(t)$ is obtained by adding one vertex labelled $\ell_i + p_j + 1$ and an oriented edge from $x_i$ to this new vertex. We label the added edge $(j,b)$, where $b$ is such that $v$ is the $b^{\mathrm{th}}$ vertex to the right of the root of $\partial_j$.
\vspace{0.5em}
\end{minipage}%
\hfill
\begin{minipage}[c]{0.55\textwidth}
\centering
\includegraphics[width =\linewidth]{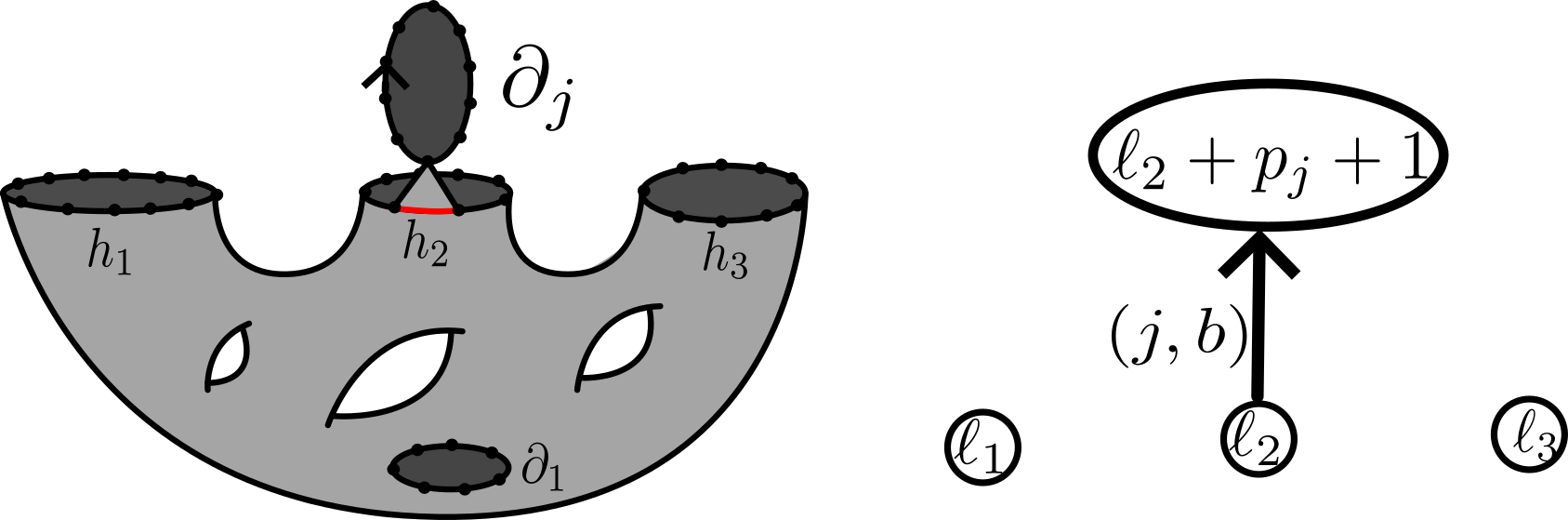}
\end{minipage}

\begin{minipage}[c]{0.4\textwidth}
\vspace{0.5em}
\textbf{Type IV.} The discovered triangle has its third vertex lying on $h_i$, splitting it into two holes. The left one has length $a$ and the right one has length $b$, with $a + b = \ell_i + 1$. Then $\mathcal{E}_{k+1}^{\mathcal{A}}(t)$ is obtained by gluing this trianglealong $\mathcal{A}(\mathcal{E}_k^{\mathcal{A}}(t))$. The decorated graph $\mathcal{D}_{k+1}^{\mathcal{A}}(t)$ is obtained by adding two vertices: one labelled $a$ and one labelled $b$. We add an oriented edge labelled $\leftarrow$ (resp. $\rightarrow$) from $x_i$ to the vertex labelled $a$ (resp. $b$).
\vspace{0.5em}
\end{minipage}%
\hfill
\begin{minipage}[c]{0.55\textwidth}
\centering
\includegraphics[width =\linewidth]{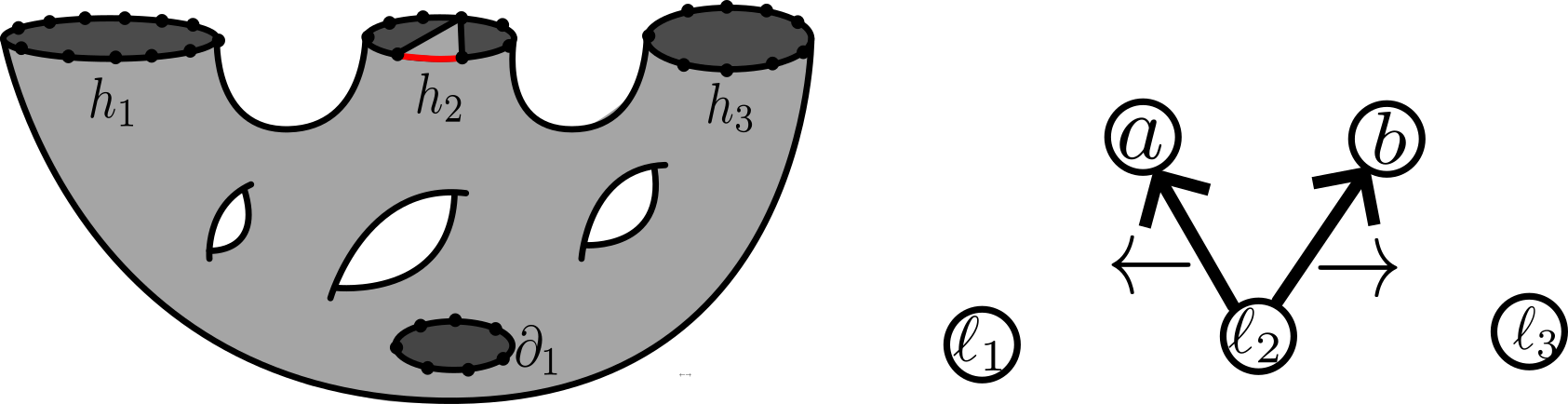}
\end{minipage}

\begin{minipage}[c]{0.4\textwidth}
\vspace{0.5em}
\textbf{Type V.} The discovered triangle has its third vertex $v$ lying on another hole $h_j$ with $j \neq i$. Thus it merges the two holes $h_i$ and $h_j$. Then $\mathcal{E}_{k+1}^{\mathcal{A}}(t)$ is obtained by gluing this triangle along $\mathcal{A}(\mathcal{E}_k^{\mathcal{A}}(t))$. The decorated graph $\mathcal{D}_{k+1}^{\mathcal{A}}(t)$ is obtained by adding a vertex $x$ labelled $\ell_i + \ell_j + 1$. We add two oriented edges: one going from $x_i$ to $x$ and one going from $x_j$ to $x$. The edge from $x_i$ to $x$ is labelled by the integer $b$ such that $v$ is the $b^{\mathrm{th}}$ vertex to the right of the leftmost vertex of $h_j$ at minimal distance from the root $\rho$.
\vspace{0.5em}
\end{minipage}%
\hfill
\begin{minipage}[c]{0.55\textwidth}
\centering
\includegraphics[width =\linewidth]{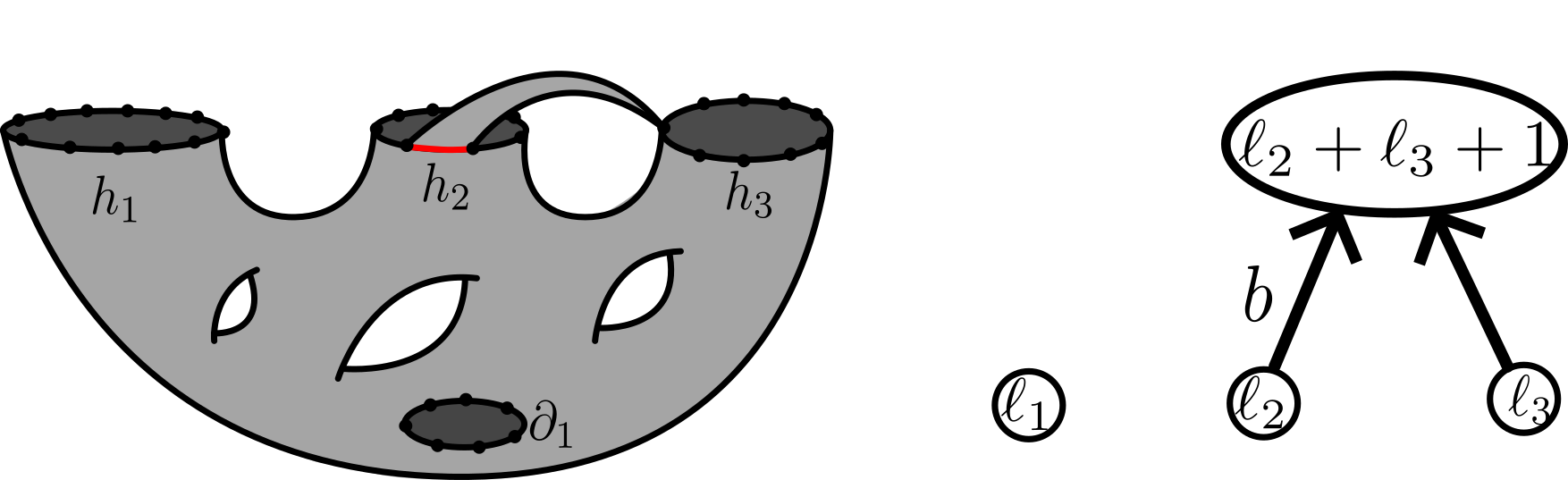}
\end{minipage}

This defines two eventually constant families $(\mathcal{E}_k^{\mathcal{A}}(t))_{k \ge 0}$ and $(\mathcal{D}_k^{\mathcal{A}}(t))_{k \ge 0}$.  
Let $\mathcal{D}^{\mathcal{A}}(t)$ denote the limit value of $(\mathcal{D}_k^{\mathcal{A}}(t))_{k \ge 0}$.  
We call $\mathcal{D}^{\mathcal{A}}(t)$ the \emph{peeling diagram} of $t$ associated with the algorithm $\mathcal{A}$.  
It is a decorated graph that encodes the complete genealogy of the peeling exploration.  
In other words, given $\mathcal{A}$ and $\mathcal{D}^{\mathcal{A}}(t)$, one can reconstruct the entire peeling exploration $(\mathcal{E}_k^{\mathcal{A}}(t))_{k \ge 0}$ by following the transitions encoded in the diagram.

We denote by $\mathcal{D}^{\mathcal{A}}(n,g,\mathbf{p}) := \mathcal{D}^{\mathcal{A}}(\mathcal{T}_{\mathbf{p}}(n,g))$ the set of all possible peeling diagrams obtained with algorithm $\mathcal{A}$.  
It is important to note that this set depends on the chosen algorithm $\mathcal{A}$.

The vertex labels encode the sizes of the holes during the exploration.  
The edge labels $(j,b)$ in \emph{Type III} steps indicate which boundary of $T$ is being revealed and which vertex is being touched by the triangle.  
The arrows $\leftarrow$ and $\rightarrow$ in \emph{Type IV} steps specify which hole was created on the left and which one on the right.  
Finally, in \emph{Type V} steps, the label $b$ on an edge indicates where the third vertex of the discovered triangle lies on $h_j$.

Note that in the planar case ($g = 0$), peeling diagrams are trees, since \emph{Type V} steps cannot occur.  
Each \emph{Type V} step increases the genus of $\mathcal{E}_k^{\mathcal{A}}(T)$ by exactly one, so there are precisely $g$ such steps.  
In the planar case ($g=0$), the set of all possible peeling diagrams does not depend on $\mathcal{A}$.  
Moreover, each \emph{Type II} step corresponds to the discovery of one internal vertex, so the number of internal vertices $n + 2 - 2g - \sum_{i=1}^{\ell} p_i$ equals the number of Type II steps.  
In \cite{Contat2022LastCD}, a similar construction is used for planar triangulations, though the definitions differ slightly since the peeling procedures are not identical: our boundaries are simple, whereas in \cite{Contat2022LastCD} they are not. Nevertheless, the two constructions remain closely related.

\begin{figure}[H]
\includegraphics[scale=0.15]{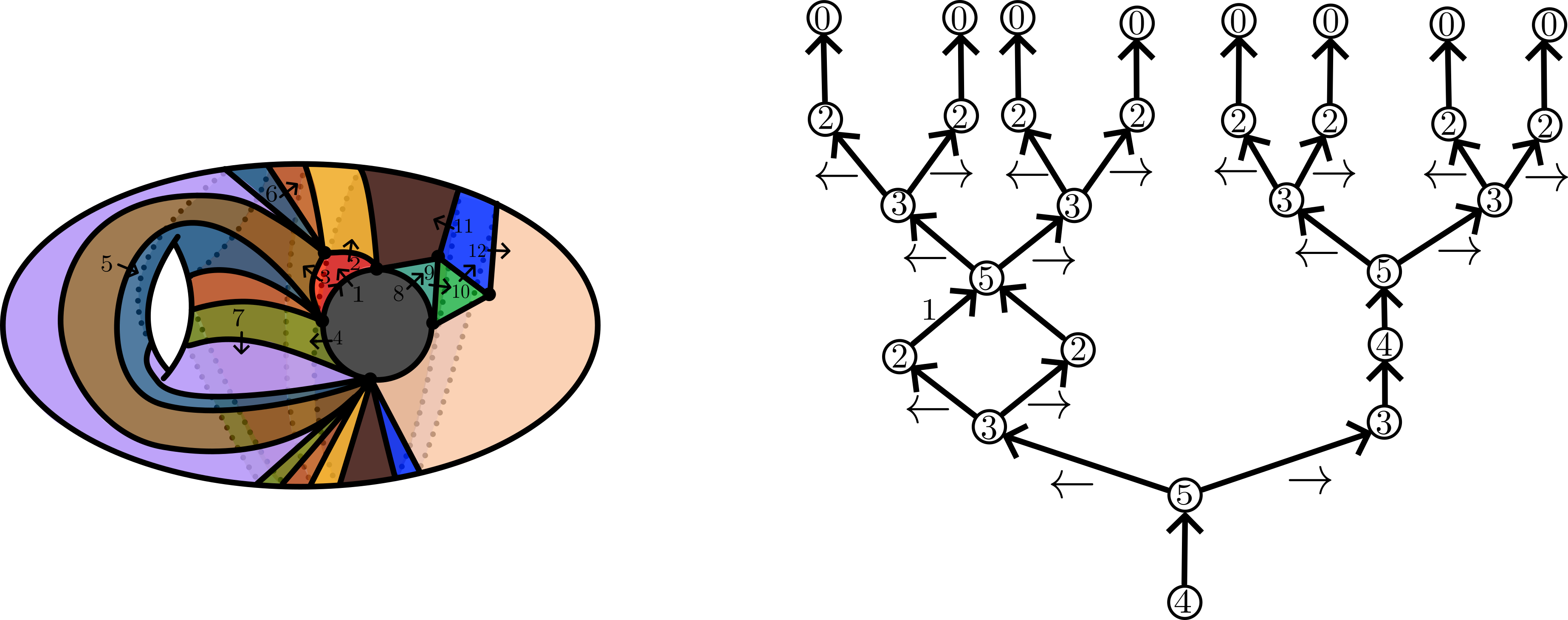}
\label{peeling_diagram}
\caption{Left: a triangulation $t \in \mathcal{T}_{4}(7,1)$. Right: its \emph{peeling diagram} associated with an algorithm $\mathcal{A}$. The peeling order is indicated by the numbers. Starting from a polygon of length $4$, one can reconstruct the entire peeling exploration of the triangulation on the left using the peeling diagram on the right.}
\end{figure}

\section{Combinatorial estimates}

This section collects the coarse combinatorial estimates used throughout the paper. 
We begin with a crude bound on the volume $\tau(n,g)$. 
A sharper asymptotic, valid in the regime $\frac{g}{n}\to \theta$, can be found in \cite[Theorem~3]{Budzinski_2020}. 
Here we only require the following weaker estimate, which also holds for various other models of random maps of large genus.

\begin{lemma}\label{estimate_tau}
There exists a constant $C > 0$ such that for all $n,g \ge 0$ with $n \ge 2g-1$,
\begin{align*}
    e^{-Cn} n^{2g} \le \tau(n,g) \le e^{Cn} n^{2g}.
\end{align*}
\end{lemma}

\begin{proof}
Let $\widetilde{\mathcal{T}}(n,g)$ denote the set of triangulations $t \in \mathcal{T}(n,g)$ endowed with a spanning tree $S$. 
Since $t$ has $n + 2 - 2g$ vertices, the tree $S$ has $n + 1 - 2g$ edges. 
Write $\tilde{\tau}(n,g)$ for the cardinality of this set. 
For each $t \in \mathcal{T}(n,g)$, since $t$ has exactly $3n$ edges, the number of spanning trees of $t$ is at most $\binom{3n}{n + 1 - 2g} \le 8^n$. 
Hence
\begin{align*}
    8^{-n} \tilde{\tau}(n,g) \le \tau(n,g) \le \tilde{\tau}(n,g).
\end{align*}

We now estimate $\tilde{\tau}(n,g)$. 
Given $t$ and a spanning tree $S$ of $t$, consider $U(t,S)$, the unicellular map of genus $g$ obtained by cutting the dual map $t^*$ along all edges dual to those of $S$. 
The resulting map is precubic (each vertex has degree $1$ or $3$). 
This cutting operation defines a non-crossing matching $\mathcal{M}$ of the leaves on the boundary face of $U(t,S)$. 
The map $U(t,S)$ has $4n + 1 - 2g$ edges. 
For such a unicellular map, the number of possible matchings equals the Catalan number $Cat(n + 1 - 2g) \le 4^n$. 
Setting $\varphi(t,S) = (U(t,S), \mathcal{M})$, we obtain an injective map from $\widetilde{\mathcal{T}}(n,g)$ to the set of precubic unicellular maps of genus $g$ with $4n + 1 - 2g$ edges and a matching. 
In fact, this mapping is a bijection.

The number of such precubic unicellular maps is given explicitly in \cite[Corollary~7]{CHAPUY2011874} as
\begin{align*}
    \frac{(8n - 4g + 2)(4n - 2g)!}{12^g g!\,(2n - 4g + 2)!(2n - g)!}.
\end{align*}
Consequently,
\begin{align}\label{cardinal_triangulation_decorated_with_spanning_tree}
    \tilde{\tau}(n,g) = Cat(n + 1 - 2g) \cdot 
    \frac{(8n - 4g + 2)(4n - 2g)!}{12^g g!\,(2n - 4g + 2)!(2n - g)!}.
\end{align}

A direct estimate yields
\begin{align*}
    12^{-n} (4n - 4g + 2)^{2g - 2} 
    \le \tilde{\tau}(n,g) 
    \le 4^{n + 1} (8n + 2) (4n)^{2g - 2}.
\end{align*}
Using $n \ge 2g - 1$, we obtain
\begin{align*}
    n^{-2} 12^{-n} n^{2g} 
    \le \tilde{\tau}(n,g) 
    \le (8n + 2) n^{-2} 64^{n + 1} n^{2g}.
\end{align*}
This proves the claim by choosing $C > 0$ large enough so that 
$(8n + 2) n^{-2} 64^{n + 1} \le e^{Cn}$ 
and $e^{-Cn} \le n^{-2} 12^{-n}$.
\end{proof}

\begin{lemma}\label{bound_separating}
For all $n,g \ge 0$ such that $n \ge 2g-1$, we have
\begin{align*}
    \sum_{\substack{n_1 + n_2 = n \\ g_1 + g_2 = g}}
    \tau(n_1,g_1)\tau(n_2,g_2)
    \le \tau(n + 1, g).
\end{align*}
\end{lemma}

\begin{proof}
For each $(t_1, t_2) \in 
\displaystyle \bigsqcup_{\substack{n_1 + n_2 = n \\ g_1 + g_2 = g}} 
\mathcal{T}(n_1,g_1) \times \mathcal{T}(n_2,g_2)$,
let $\varphi(t_1,t_2)$ be the map obtained as follows: 
in each $t_i$, split the root edge $e_i$ into a digon, 
fill the digon with a loop based at the origin of $e_i$, 
and glue the two resulting maps along the loops, rooting the resulting map on this loop. 
This produces $\varphi(t_1,t_2) \in \mathcal{T}(n + 1, g)$. 
The construction is clearly injective, completing the proof.
\end{proof}

We now state a version of the bounded-ratio lemma \cite[Lemma~4]{Budzinski_2020} adapted to our setting.\footnote{%
There is a minor inaccuracy in the proof of \cite[Lemma~4]{Budzinski_2020}: 
the gluing operation described there does not necessarily yield a map in $\mathcal{T}_p(n-1,g)$ when one of the two \textbf{GOOD} vertices lies on a boundary. 
The conclusion remains correct since the authors only consider the regime $|\mathbf{p}| = \mathcal{O}(1)$.}
Since in our work we consider boundaries of perimeters $\mathbf{p}^n=(p^n_1,\dots,p^n_{\ell_n})$ satisfying $|\mathbf{p}^n| = o(n)$, 
we also require an estimate showing that changing the total perimeter by a constant only modifies the volume by a constant factor.

\begin{lemma}\label{bounded_ratio_vertices}
Let $\varepsilon > 0$. There exists a constant $C_{\varepsilon} > 0$ such that for all $g \ge 0$ and $n \ge 2g-1$ with $\frac{g}{n} \le \frac{1}{2} - 4\varepsilon$ and for all $p_1,\dots,p_{\ell} \ge 1$ with $|\mathbf{p}| \le \varepsilon n$,
\begin{align*}
    1 \le \frac{\tau_{\mathbf{p}}(n+1,g)}{\tau_{\mathbf{p}}(n,g)} \le C_{\varepsilon},
\end{align*}
and
\begin{align*}
    C_\varepsilon^{-1} \le 
    \frac{\tau_{(p_1 + 1, p_2, \dots, p_{\ell})}(n,g)}{\tau_{(p_1, p_2, \dots, p_{\ell})}(n,g)} 
    \le 1.
\end{align*}
\end{lemma}

\begin{proof}
For the first inequality, we follow the argument of \cite[Lemma~4]{Budzinski_2020}. 
For $t \in \mathcal{T}_{\mathbf{p}}(n,g)$, we call a pair of vertices $(v,v')$ \textbf{GOOD} if 
$\deg(v) + \deg(v') \le \frac{6}{\varepsilon}$, 
$d^{*}(v,v') \le \frac{12}{\varepsilon}$, 
and neither vertex lies on a boundary\footnote{Our definition slightly differs from that in \cite[Lemma~5]{Budzinski_2020}, 
where boundary vertices were allowed.}.

As in \cite[Lemma~5]{Budzinski_2020}, one shows that there are at least $\frac{\varepsilon}{12}n$ such pairs in $t$. 
Indeed, the average degree $m$ satisfies
\[
m= \frac{2(3n + \sum_{i=1}^{\ell} (3 - p_i))}{n + 2 - 2g} \le \frac{3}{2\varepsilon},
\]
and since $n + 2 - 2g \ge 8\varepsilon n$, there are at least $\varepsilon n$ vertices of degree less than $\frac{3}{\varepsilon}$ (excluding boundary vertices).

Following the proof of \cite[Lemma~4]{Budzinski_2020}, to each $(t,v,v')$ with $t \in \mathcal{T}_{\mathbf{p}}(n+1,g)$ and $(v,v')$ a \textbf{GOOD} pair, 
associate a triangulation $\varphi(t,v,v')$ in $\mathcal{T}_{\mathbf{p}}(n,g)$ with a marked vertex $v$ of degree at most $\frac{36}{\varepsilon}$ and two marked edges incident to $v$. 
The number of possible outputs is bounded by 
$n (36/\varepsilon)^{24/\varepsilon + 2} \tau_{\mathbf{p}}(n,g)$. 
Since $\varphi$ is injective, we deduce
\begin{align*}
    \frac{\varepsilon}{12} n \tau_{\mathbf{p}}(n+1,g)
    \le \Big(\frac{36}{\varepsilon}\Big)^{24/\varepsilon + 2} n \tau_{\mathbf{p}}(n,g).
\end{align*}

For the second inequality, the right-hand side follows from the injection 
$\mathcal{T}_{(p_1 + 1, p_2, \dots, p_{\ell})}(n,g) \hookrightarrow \mathcal{T}_{(p_1, p_2, \dots, p_{\ell})}(n,g)$ 
obtained by adding an edge inside the first external boundary between the root edge and the second vertex to its right. 
The left-hand side follows from the injection 
$\mathcal{T}_{(p_1, p_2, \dots, p_{\ell})}(n,g) \hookrightarrow \mathcal{T}_{(p_1 + 1, p_2, \dots, p_{\ell})}(n+1,g)$ 
constructed by gluing a triangle along the distinguished edge $e_1$ inside the first boundary, 
and applying the first inequality.
\end{proof}

\begin{lemma}\label{add_new_boundary}
Let $\varepsilon > 0$. 
There exists $C_{\varepsilon} > 0$ such that for all $g \ge 0$ and $n \ge 2g-1$ with $\frac{g}{n} \le \frac{1}{2} - 4\varepsilon$ and all $p_1,\dots,p_{\ell} \ge 1$ with $|\mathbf{p}| \le \varepsilon n$,
\begin{align*}
    \tau_{(p_1,\dots,p_{\ell})}(n,g) \le (6n)^{\ell - 1} \tau(n,g),
\end{align*}
and
\begin{align*}
    (C_{\varepsilon})^{-1} n 
    \le \frac{\tau_{(1,p_1,\dots,p_{\ell})}(n,g)}{\tau_{(p_1,\dots,p_{\ell})}(n,g)}.
\end{align*}
\end{lemma}

\begin{proof}
The first inequality follows from \cite[Lemma~1]{Budzinski_2020}. 
For the second, fix $t \in \mathcal{T}_{(p_1,\dots,p_{\ell})}(n,g)$ and an oriented edge $e$ with endpoints $x$ and $y$ (possibly $x=y$). 
Perform the following operations (see Figure~\ref{split_edge}):
\begin{enumerate}
    \item Split $e$ into a digon.
    \item Add a loop based at $x$ inside the digon.
    \item Add a digon inside that loop and denote by $z$ the new vertex.
    \item Add a loop inside the new digon starting at $z$.
\end{enumerate}
\begin{figure}[H]
\centering
\includegraphics[scale=0.7]{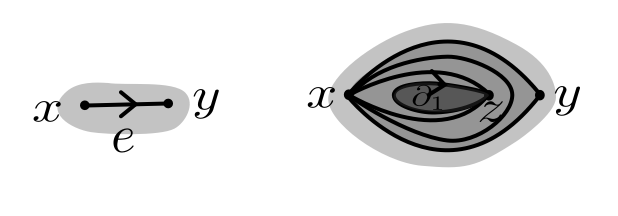}
\caption{Left: the distinguished edge. Right: the resulting map after the sequence of operations.}
\label{split_edge}
\end{figure}
This operation preserves the boundary lengths and adds one vertex and one boundary (the last loop) of length $1$. Moreover, the operation is injective. We deduce
\begin{align*}
    2\Big(3n + \sum_{i=1}^{\ell} (3 - p_i)\Big) 
    \tau_{(p_1,\dots,p_{\ell})}(n,g) 
    \le \tau_{(1,p_1,\dots,p_{\ell})}(n+1,g),
\end{align*}
where the factor on the left-hand side counts the number of oriented edges. 
Combining this with Lemma~\ref{bounded_ratio_vertices} gives
\begin{align*}
    \frac{\tau_{(1,p_1,\dots,p_{\ell})}(n,g)}{\tau_{(p_1,\dots,p_{\ell})}(n,g)}
    \ge (C_{\varepsilon})^{-1} (6 - 2\varepsilon) n.
\end{align*}
\end{proof}

\begin{proposition}\label{planarity_estimate}
Let $\varepsilon > 0$. 
There exists $C_{\varepsilon} > 0$ such that for all $g \ge 0$ and $n \ge 2g-1$ with $\frac{g}{n} \le \frac{1}{2} - \varepsilon$, 
any $k \ge 0$, and sequences $\mathbf{p}_1,\dots,\mathbf{p}_k$ with 
$\mathbf{p}_i = (p_{1,i},\dots,p_{\ell_i,i})$, we have
\begin{align*}
    \sum_{\substack{n_1 + \cdots + n_k = n \\ h_1 + \cdots + h_k = g}}
    \prod_{i=1}^k \tau_{\mathbf{p}_i}(n_i,h_i)
    \le \exp(C_{\varepsilon}(n+k)) n^{2g + \sum_{i=1}^{k} (\ell_i - 1)}.
\end{align*}
\end{proposition}

\begin{proof}
Using Lemma~\ref{add_new_boundary} and Lemma~\ref{bound_separating} and the fact that $n_i \le n$ for all $i$, we obtain
\begin{align*}
    \sum_{\substack{n_1+\cdots+n_k = n \\ h_1+\cdots+h_k = g}}
    \prod_{i=1}^{k} \tau_{\mathbf{p}_i}(n_i,h_i)
    &\le (6n)^{\sum_{i=1}^{k} (\ell_i - 1)}
    \sum_{\substack{n_1+\cdots+n_k = n \\ h_1+\cdots+h_k = g}}
    \prod_{i=1}^{k} \tau(n_i,h_i) \\
    &\le (6n)^{\sum_{i=1}^{k} (\ell_i - 1)} \tau(n + k - 1, g).
\end{align*}
Since $\frac{g}{n+k-1} \le \frac{g}{n} \le \frac{1}{2} - \varepsilon$, using Lemma~\ref{bounded_ratio_vertices} and
Lemma~\ref{estimate_tau} yields
\[
    \tau(n + k - 1, g)
    \le C_{\varepsilon}^k\exp(C_{\varepsilon}n ) n^{2g},
\]
and the result follows.
\end{proof}
We conclude this section with an estimate on planar triangulations that will be useful later.
\begin{lemma}\label{estimate_big_hole}
For any $\alpha,\varepsilon > 0$, there exists a constant $C_{\alpha,\varepsilon} > 0$ such that for any $n,p \ge 1$ such that $n +2 \ge p$ we have 
   \begin{align*}
    \sum_{\substack{n_1+n_2 = n-1\\p_1+p_2 = p+1 \\  \alpha p \le p_1 \le \frac{p}{2}}} \frac{\tau_{p_1}(n_1,0) \tau_{p_2}(n_2,0)}{\tau_{p}(n,0)} \le C_{\alpha,\varepsilon}p^{-\frac{1}{4}+\varepsilon}.
    \end{align*}
\end{lemma}
\begin{remark1}
This last estimate is expected to be non-optimal. In the regime where $\frac{p}{n} \to (0,1)$ the decay should be of order $p^{-\frac{1}{2}}$ while in the regime $p = \mathcal{O}(\sqrt{n})$ it should be of order $p^{-\frac{3}{2}}$. Here, we obtain the bound $p^{-\frac{1}{4}+\varepsilon}$ in all regimes. We did not attempt to optimize, as it is sufficient for the rest of the paper.
\end{remark1}
\begin{proof}
Let us fix $\alpha,\varepsilon > 0$.
In this proof, we use the notation $f(n,p) = \Theta(g(n,p))$ when $cg(n,p) \le f(n,p) \le Cg(n,p)$ for some absolute constants  $c,C > 0$.

For $p \ge 1$ and $n \ge p-2$, we recall the formula obtained in \cite{krikun2007explicitenumerationtriangulationsmultiple} for the volumes $\tau_p(n,0)$:
\begin{align*}
\tau_p(n,0) = \frac{p(2p)!}{(p!)^2}\frac{4^{n-p+1}(3n-p+1)!!}{(n-p+2)!(n+p+1)!!}.
\end{align*}
Using Stirling's formula, we obtain the following uniform estimate:
    \begin{align*}
    \tau_p(n,0) &= \Theta\left(\frac{\sqrt{p}\sqrt{3n-p+1}}{\sqrt{n+3-p}(n+p+1)^{\frac{5}{2}}}\frac{4^{n}(3n-p+1)^{\frac{3n-p+1}{2}}}{(n+2-p)^{n+2-p}(n+p+1)^{\frac{n+p-3}{2}}}\right)\\
    &= \Theta\left(\frac{\sqrt{p}\sqrt{3n-p+1}}{\sqrt{n+3-p}(n+p+1)^{\frac{5}{2}}}(12 \sqrt{3})^{n}3^{-\frac{p}{2}}\varphi(n,p)\right),
    \end{align*}   
   where
   \begin{align*}
    \varphi(n,p)= \exp\left(n h\left(\frac{p}{n+2}\right)\right),
   \end{align*}
    where for any $u \in [0,1]$,
    \begin{align*}
    h(u) = \frac{3-u}{2}\log\left(1-\frac{u}{3}\right) - (1-u)\log(1-u) - \frac{1+u}{2}\log(1+u).
    \end{align*}
 Consequently, for any $n_1+n_2 = n-1$ and $p_1+p_2 = p+1$ such that $\alpha p \le p_1 \le \frac{p}{2}$, we can write
    \begin{align}\label{ineq_ratio}
    \frac{\tau_{p_1}(n_1,0) \tau_{p_2}(n_2,0)}{\tau_{p}(n,0)} = \mathcal{O}\bigg( &\frac{\sqrt{p_1 p_2}}{\sqrt{p}} \frac{\sqrt{3n_1-p_1+1}\sqrt{3n_2-p_2+1}}{\sqrt{3n-p+1}} \nonumber \\
    &\times \frac{\sqrt{n+3-p}}{\sqrt{n_1+3-p_1}\sqrt{n_2+3-p_2}} \nonumber \\
    &\times \frac{(n+p+1)^{\frac{5}{2}}}{(n_1+p_1+1)^{\frac{5}{2}}(n_2+p_2+1)^{\frac{5}{2}}} \nonumber \\
    &\times \frac{\varphi(n_1,p_1)\varphi(n_2,p_2)}{\varphi(n,p)}\bigg).
    \end{align}
    
     Now, we observe that
    \begin{align}\label{bound_on_phi2}
    \frac{\varphi(n_1,p_1)\varphi(n_2,p_2)}{\varphi(n,p)}  = \mathcal{O}\bigg( \exp\left(-n\left(h\left(\frac{p}{n+2}\right)-\frac{n_1}{n}h\left(\frac{p_1}{n_1+2}\right)- \frac{n_2}{n}h\left(\frac{p_2}{n_2+2}\right)\right)\right) \bigg).
    \end{align} 
    It can be verified that $h$ is strictly concave with $h''(0) < 0$. Therefore, there exists a constant $c_{\alpha} > 0$ such that 
    \begin{align}\label{exp_bound}
    \frac{\varphi(n_1,p_1)\varphi(n_2,p_2)}{\varphi(n,p)} \le \exp\left( -c_{\alpha}n \left|\frac{n_1}{n}-\frac{p_1}{p}\right|^2\right).
    \end{align}
     Then, for $p_1+p_2 = p+1$ fixed, we write 
    \begin{align}\label{boundbyAandB}
    \somme{n_1+n_2=n-1}{}{\frac{\tau_{p_1}(n_1,0) \tau_{p_2}(n_2,0)}{\tau_{p}(n,0)}} = \mathcal{O} \bigg( \mathcal{A}_n(p,p_1,p_2) + \mathcal{B}_n(p,p_1,p_2)\bigg) ,
    \end{align}
    where $\mathcal{A}_n(p,p_1,p_2)$ contains the terms in the sum such that $\bigg|n_1-p_1 \frac{n}{p}\bigg| > n^{\frac{1}{2}+\varepsilon}$ and $\mathcal{B}_n(p,p_1,p_2)$ contains the terms in the sum such that $\bigg|n_1-p_1 \frac{n}{p}\bigg| \le n^{\frac{1}{2}+\varepsilon}$. Using \eqref{ineq_ratio} and \eqref{exp_bound}, it follows that 
    \begin{align}\label{boundonA}
    \mathcal{A}_n(p,p_1,p_2) = \mathcal{O}\bigg( n \times n^{5}  \exp(-c_{\alpha}n^{\varepsilon}) \bigg).
    \end{align}
The factor $n^{5}$ is a crude bound for the factors in \eqref{ineq_ratio} (except the last factor). The factor $n$ bounds the number of terms appearing in the sum $\mathcal{A}_n(p,p_1,p_2)$.
   
    Now let us bound $\mathcal{B}_n(p,p_1,p_2)$. Since $(n_1+3-p_1) + (n_2+3-p_2) = n+6-p > n + 3 -p$ we deduce that 
   \begin{align*}
   \min_{i=1,2} \frac{\sqrt{n+3-p}}{\sqrt{n_i+3-p_i}} \le \sqrt{2}.
   \end{align*}
Moreover we also have
\begin{align*}
   \max_{i=1,2} \frac{\sqrt{3n_i-p_i+1}}{\sqrt{3n-p+1}} \le \frac{\sqrt{3}}{\sqrt{2}}.
   \end{align*} 
For $n_1$ such that $\bigg|n_1-p_1 \frac{n}{p}\bigg| \le n^{\frac{1}{2}+\varepsilon}$, we have $n_1 \ge \alpha n - n^{\frac{1}{2}+ \varepsilon}$ and $n_2 = n - 1 - n_1 \ge \frac{n}{2} - 1 - n^{\frac{1}{2}+ \varepsilon}$. We deduce the following bounds 
\begin{align*}
\frac{\tau_{p_1}(n_1,0) \tau_{p_2}(n_2,0)}{\tau_{p}(n,0)}  & = \mathcal{O}_{\alpha}\bigg(  \sqrt{p}n^{-\frac{5}{2}}\max_{i=1,2}\frac{\sqrt{3n_i-p_i+1}}{\sqrt{n_i+3-p_i}} \bigg) \\
& =  \mathcal{O}_{\alpha}\bigg( \sqrt{p}n^{-2}\bigg(\frac{1}{\sqrt{n_1+3-p_1}}+\frac{1}{\sqrt{n_2+3-p_2}}\bigg)\bigg).
\end{align*}
Finally summing over $n_1$, we obtain the bound 
\begin{align}\label{boundonB}
\mathcal{B}_n(p,p_1,p_2) =  \mathcal{O}_{\alpha}\bigg(\sqrt{p}n^{-2} \somme{k=1}{2n^{\frac{1}{2}+\varepsilon}+1}{\frac{1}{\sqrt{k}}} \bigg)  =\mathcal{O}_{\alpha}\bigg( n^{-\frac{5}{4}+2\varepsilon} \bigg),
\end{align}
where the first bound follows from the fact there are at most $2n^{\frac{1}{2}+\varepsilon}+1$ terms appearing in $\mathcal{B}_n(p,p_1,p_2)$ and the sum of the terms $\frac{1}{\sqrt{n_1+3-p_1}}+\frac{1}{\sqrt{n_2+3-p_2}}$ is at most\footnote{This inequality is very crude. This is why we obtain $p^{-\frac{5}{4}+\varepsilon}$ and not $p^{-\frac{3}{2}}$.} $2\somme{k=1}{2n^{\frac{1}{2}+\varepsilon}+1}{\frac{1}{\sqrt{k}}}$.

Putting together \eqref{boundonA} and \eqref{boundonB} in \eqref{boundbyAandB}, we obtain 
      \begin{align*}
    \somme{n_1+n_2=n-1}{}{\frac{\tau_{p_1}(n_1,0) \tau_{p_2}(n_2,0)}{\tau_{p}(n,0)}} = \mathcal{O}_{\alpha,\varepsilon}\bigg( n^{-\frac{5}{4}+\frac{\varepsilon}{2}}\bigg).
    \end{align*}
    Then summing over $p_1 \in \{cp,\cdots,\frac{p}{2}\}$ and using the fact that there are at most $n$ terms gives
    \begin{align*}
     \sum_{\substack{n_1+n_2 = n-1\\p_1+p_2 = p+1 \\  \alpha p \le p_1 \le \frac{p}{2}}} \frac{\tau_{p_1}(n_1,0) \tau_{p_2}(n_2,0)}{\tau_{p}(n,0)} = \mathcal{O}_{\alpha,\varepsilon}\bigg( p^{-\frac{1}{4}+\frac{\varepsilon}{2}} \bigg).
    \end{align*}
    This concludes the proof.
\end{proof}

\section{Local limit from a uniform point in $T_{n,g_n,\mathbf{p}^{n}}$}\label{section_middle}

In this section, let $\theta \in [0,\frac{1}{2})$, assume that $\displaystyle \frac{g_n}{n} \to \theta$, and let $\mathbf{p}^{n} = (p^n_{1},\dots,p^n_{\ell_n})$ be such that $|\mathbf{p}^{n}|= o(n)$. 
Let $e^n$ denote an oriented edge chosen uniformly at random in $T_{n,g_n,\mathbf{p}^{n}}$ and $\rho^n$ its starting point. 
The goal of this section is to show that the local limit of $(T_{n,g_n,\mathbf{p}^{n}},e^n)$ with respect to the topology $d_{\mathrm{loc}}$ is $\mathbb{T}_{\lambda(\theta)}$, where $\lambda(\theta)$ is defined in~\eqref{deflambdatheta}.

Following the approach of~\cite{Budzinski_2020}, we first obtain estimates ensuring that any subsequential limit is planar and one-ended. 
We then conclude that every subsequential limit must be $\mathbb{T}_{\lambda(\theta)}$.

The main new idea of this section lies in the proof of Lemma~\ref{T_is_almost_surely_planar_rooted_middle}, which establishes the planarity of any subsequential limit. 
We refer the reader to Section~\ref{strategy_of_proof} for an overview of the proof strategy.

\subsection{Planarity when rooted in the middle}\label{planarity_middle}

In this section, we focus on planarity.  
We begin by showing that a uniformly chosen oriented edge is, with high probability, at graph distance greater than $r$ from the boundaries. 

\begin{lemma}\label{T_do_not_see_boundaries__rooted_middle}
Recall that $e^n$ denotes a uniformly chosen oriented edge in $T_{n,g_n,\mathbf{p}^{n}}$ and $\rho^n$ its starting point. For any $r >0$, we have
\[
\Pf\big(d(\rho^n,\partial T_{n,g_n,\mathbf{p}^{n}}) \le r\big)\underset{n\to +\infty}{\to}0.
\]
\end{lemma}

\begin{proof}
The argument is similar to the proof of the lower bound in~\cite[Theorem~1]{budzinski2023distancesisoperimetricinequalitiesrandom}.

Fix $(t,e)$ where $t \in \mathcal{T}_{\mathbf{p}^{n}}(n,g_n)$ and $e$ is an oriented edge on $t$ starting from a vertex $\rho$. Suppose that there exists a vertex $x \in \bigcup_{i=1}^{\ell_n}\partial_i t$ such that $d(x,\rho) \le r$, where $\rho$ denotes the starting point of $e$. 
We distinguish three cases.

\smallskip
\noindent
\textbf{Case 1.}  
Assume first that $e \in \partial t$.  
The number of such pairs $(t,e)$ equals $2|\mathbf{p}^{n}|\tau_{\mathbf{p}^{n}}(n,g_n)$.

\smallskip
\noindent
\textbf{Case 2.}  
Assume now that $\rho \in \partial t$ but $e \notin \partial t$. 
Fix $1 \le i \le \ell_n$ such that $\rho \in \partial_i$. 
Consider the triangulation obtained by splitting $e$ into a digon and then splitting $\rho$ into two vertices $x$ and $y$ joined by a boundary edge. 
This produces a triangulation $t' \in \mathcal{T}_{(p^n_1,\dots,p^n_i+1,\dots,p^n_{\ell_n})}(n,g_n)$. 
Let $e'$ denote the oriented edge corresponding to the new edge created on the $i^{\mathrm{th}}$ boundary, oriented so that $\partial_i$ lies to its right (see Figure~\ref{add_triangle}). 
We define $\varphi(t,e) = (t',e')$. 
This map is injective, since $(t,e)$ can be recovered by identifying the endpoints of $e'$, removing $e'$, and merging the two edges of the resulting digon. 
This procedure applies even when $e$ is a loop. 
Hence, the number of such pairs $(t,e)$ is bounded by
\begin{align}\label{first_bound}
    \sum_{i=1}^{\ell_n}\tau_{(p^n_1,\dots,p^n_i+1,\dots,p^n_{\ell_n})}(n,g_n)(p^n_i+1) 
    \le C_{\theta}|\mathbf{p}^{n}|\tau_{\mathbf{p}^{n}}(n,g_n),
\end{align}
where the inequality follows from Lemma~\ref{bounded_ratio_vertices}.

\begin{figure}[H]
    \centering
    \includegraphics[scale =0.8]{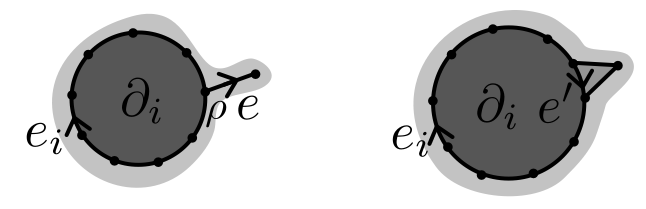}
    \caption{Left: the distinguished edge $e$ with starting point $\rho$ on $\partial_i$. 
    Right: the map obtained after the operation.}
    \label{add_triangle}
\end{figure}

\smallskip
\noindent
\textbf{Case 3.}  
Assume finally that $\rho \notin \bigcup_{i=1}^{\ell_n}\partial_i$. 
Choose a simple path $\gamma = (\gamma_0,\dots,\gamma_s)$ such that the first edge $(\gamma_0,\gamma_1)$ is $e$ (or its reverse $-e$) and the endpoint $\gamma_s$ lies on some $\partial_i$ with $1\le i\le \ell_n$. 
Assume moreover that $\gamma_j \notin \partial t$ for all $j < s$ and that $s \le r+1$. 
Cutting along $\gamma$ produces a triangulation $t' \in \mathcal{T}_{p^n_1,\dots,p^n_{i}+2s,\dots,p^n_{\ell_n}}(n+s,g_n)$. 
We equip the $i^{\mathrm{th}}$ boundary of $t'$ with a distinguished oriented edge $e'$ corresponding to $e$ after the cutting, oriented so that $\partial_i$ lies to its right. 
We also distinguish the segment $\gamma'$ on $\partial_i$ arising from the cut. 
This defines $\varphi(t,e) = (t',e',\gamma')$. 
The map is injective since $(t,e)$ can be recovered from $(t',e',\gamma')$ by gluing the edges along $\gamma'$. 
Hence, the number of inputs is bounded by 
\begin{align}\label{second_bound}
    \sum_{i=1}^{\ell_n}\sum_{s=0}^{r+1} 2(p^n_i+2s)\tau_{p^n_1,\dots,p^n_{i}+2s,\dots,p^n_{\ell_n}}(n+s,g_n)
    \le C_{\theta,r} |\mathbf{p}^{n}|\tau_{\mathbf{p}^{n}}(n,g_n),
\end{align}
where the inequality again follows from Lemma~\ref{bounded_ratio_vertices}. 
Here, the extra factor $p_i^n+2s$ accounts for the number of choices for $e'$ on $\partial_i$, and the factor $2$ bounds the number of possible segments $\gamma'$.

\begin{figure}[H]
    \centering
    \includegraphics[scale =0.8]{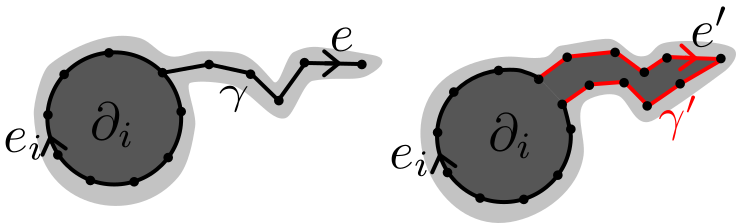}
    \caption{Left: the distinguished edge $e$ and a path $\gamma$ from $e$ to $\partial_i$. 
    Right: the resulting map after cutting along $\gamma$.}
    \label{path_to_boundary}
\end{figure}

Combining \eqref{first_bound} and \eqref{second_bound}, we obtain
\[
\Pf\big(d(\rho^n,\partial T_{n,g_n,\mathbf{p}^{n}}) \le r\big)
\le C_{\theta,r}\frac{|\mathbf{p}^{n}|}{6n - 2\sum_{i=1}^{\ell_n}(3-p^n_i)}.
\]
Since $|\mathbf{p}^{n}| = o(n)$, the result follows.
\end{proof}

\begin{corollary}\label{tightness}
The sequence $(T_{n,g_n,\mathbf{p}^{n}},e^n)$ is tight with respect to the topology $d^{*}_{\mathrm{loc}}$, and any subsequential limit is a weakly Markovian triangulation.
\end{corollary}

\begin{proof}
The tightness, and the fact that any subsequential limit is a triangulation, follow from the discussion in Section~\ref{subsection_local_convergence} together with Proposition~\ref{T_do_not_see_boundaries__rooted_middle}. Fix $T$ a subsequential limit. Fix a finite rooted triangulation with holes $t$ having no boundary. 
Writing $\Pf(t \subset (T_{n,g_n,\mathbf{p}^{n}},e^n))$ as a ratio of two combinatorial quantities, there are functions $\varphi_n$ such that
\[
\Pf(t \subset (T_{n,g_n,\mathbf{p}^{n}},e^n)) = \varphi_n(v,p_1,\dots,p_k),
\]
where $v$ denotes the number of vertices of $t$ and $p_1,\dots,p_k$ are the perimeters of its holes. 
Hence $\Pf(t\subset T )= \lim_n \varphi_n(v,p_1,\dots,p_k) $, which shows that $T$ is weakly Markovian.
\end{proof}

For the rest of this section, we introduce a family $(e^n_i)_{i\ge 0}$ of independent uniformly chosen oriented edges on $T_{n,g_n,\mathbf{p}^{n}}$. For $i \ge 0$, we write $\rho^n_i$ for the starting point of $e^n_i$.
For any $r \ge 0$ and $i,j \in \mathbb{N}$ such that $i \neq j$, we define the events
\begin{align}\label{events_C_and_D}
    &\mathcal{A}_{n}(i,j,r) = \big\{d(\rho^n_i,\rho^n_j) > 2r \big\},\\
    &\mathcal{B}_n(i,r) = \big\{d(\rho^n_i,\partial T_{n,g_n,\mathbf{p}^{n}}) > 3r\big\}.
\end{align}
Using Proposition~\ref{T_do_not_see_boundaries__rooted_middle}, we obtain that for any $i \ge 0$,
\[
\Pf(\mathcal{B}_n(i,r)) \xrightarrow[n\to\infty]{} 1.
\]

 The next proposition shows that, conditionally on $\mathcal{B}_n(i,r)$, the event $\mathcal{A}_n(i,j,r)$ holds with high probability.

\begin{remark1}
The probability of the event $\mathcal{A}_n(i,j,r)$ does not depend on the choice of $i$ and $j$.
\end{remark1}

\begin{proposition}\label{two_uniform_edges_stay_far}
For any $r \ge 0$, there exists a constant $C_{r,\theta} > 0$ such that for $n$ large enough
\[
\Pf\big(\mathcal{B}_n(1,r)\cap \mathcal{A}_n(1,2,r)^{c}\big) \le \frac{C_{r,\theta}}{n}.
\]
\end{proposition}

\begin{proof}
The proof follows the same general strategy as that of Proposition~\ref{T_do_not_see_boundaries__rooted_middle}. 
Fix $r \ge 0$, and consider $t \in \mathcal{T}_{\mathbf{p}^{n}}(n,g_n)$ together with two oriented edges $e$ and $e'$. We introduce the starting points $\rho$ and $\rho'$ of $e$ and $e'$. Suppose that $d(\rho,\rho') \le 2r$ and $d(\rho,\partial t) > 3r$. 
Choose a vertex-injective path $\gamma = (\gamma_0,\dots,\gamma_s)$ of length at most $2r+2$, whose first edge is either $e$ or its reverse $-e$, and whose last edge is either $e'$ or $-e'$. 
Since $d(\rho,\partial t) > 3r$, this path is vertex-disjoint from all boundaries of $t$.

Cutting along $\gamma$ produces a triangulation $t' \in \mathcal{T}_{(2s,p_1^n,\dots,p^n_{\ell_n})}(n+s-1,g_n)$ and an additional oriented edge $e^{*}$ lying on $\partial_1$. 
Intuitively, the distinguished boundary of length $2s$ encodes the path $\gamma$, the root edge of this boundary encodes $e$ and $e^{*}$ encodes $e'$. 
This construction defines an injective mapping. Hence,
\begin{align*}
    \Pf\big(\mathcal{B}_n(1,r)\cap \mathcal{A}_n(1,2,r)^{c}\big) 
    &\le 
    \frac{\sum_{s=1}^{2r+2} 2s\, \tau_{(2s,\mathbf{p}^{n})}(n,g_n)}
    {(6n-2\sum_{i=1}^{\ell_n}(3-p^n_i))^2 \, \tau_{\mathbf{p}^{n}}(n,g_n)}\\
    &\le 
    \frac{C_{r,\theta}n}{(6n-2\sum_{i=1}^{\ell_n}(3-p^n_i))^2}\\
    &\le 
    \frac{C_{r,\theta}}{n},
\end{align*}
where the second inequality follows from Lemma~\ref{bounded_ratio_vertices} and Lemma~\ref{add_new_boundary}, and the last one from the fact that $|\mathbf{p}^{n}| = o(n)$.
\end{proof}

\medskip
In the next lemma, we use the main new idea of this section (see the strategy of proof in the Section~\ref{introduction}). We now show that any subsequential limit is planar.

\begin{lemma}\label{T_is_almost_surely_planar_rooted_middle}
Every subsequential limit $(T,e)$ of the sequence $(T_{n,g_n,\mathbf{p}^{n}},e^n)_{n \ge 0}$ with respect to the topology $d^{*}_{\mathrm{loc}}$ is almost surely planar.
\end{lemma}

\begin{proof}
Fix a subsequential limit $(T,e)$. We reason by contradiction. Assume that there exist $\varepsilon > 0$ and a finite rooted triangulation with holes $t$ of genus $g' \ge 1$ such that 
\[
\Pf(t \subset (T,e)) \ge 3\varepsilon > 0.
\]
Then, for $n$ large enough
\[
\Pf(t \subset (T_{n,g_n,\mathbf{p}^{n}},e^n)) > 2\varepsilon > 0.
\]
Let $r \ge 0$ be such that $t$ has diameter at most $r$. Recall the definitions of $\mathcal{A}_n(i,j,r)$ and $\mathcal{B}_n(i,r)$ from~\eqref{events_C_and_D}. 
Let $c := c_{\varepsilon,r,\theta} \le C_{r,\theta}^{-1}\displaystyle \frac{\varepsilon}{2}$, where $C_{r,\theta}$ is the constant from Proposition~\ref{two_uniform_edges_stay_far}. 
For $1\le i \le cn$, we define
\begin{align}\label{new_events}
    \mathcal{C}_{n}(i,t) 
    = 
    \{t \subset (T_{n,g_n,\mathbf{p}^{n}},e^n_i)\}
    \cap 
    \mathcal{B}_n(i,r)
    \cap
    \bigcap_{\substack{j=1\\j\neq i}}^{cn}\mathcal{A}_n(i,j,r).
\end{align}

Since the family $\big(T_{n,g_n,\mathbf{p}^{n}},e^n_i\big)_{i\ge 0}$ is exchangeable, the probabilities $\Pf(\mathcal{C}_n(i,t))$ are identical for all $1 \le i \le cn$. 
By Lemma~\ref{T_do_not_see_boundaries__rooted_middle} and Proposition~\ref{two_uniform_edges_stay_far}, we have $\Pf(\mathcal{C}_n(i,t))\ge \varepsilon$ for $n$ large enough. 
Hence, 
\[
\sum_{i=1}^{cn}\mathbf{1}_{\mathcal{C}_n(i,t)}
\]
is a random variable taking values in $\{1,\dots,cn\}$ with expectation at least $cn\varepsilon$. Therefore,
\begin{align}\label{lower_bound_proba_copies}
    \Pf\!\left(
        \bigcup_{\substack{I\subset \{1,\dots,cn\}\\|I|=\frac{\varepsilon}{2}cn}}
        \bigcap_{i\in I}\mathcal{C}_n(i,t)
    \right)
    =
    \Pf\!\left(
        \sum_{i=1}^{cn}\mathbf{1}_{\mathcal{C}_n(i,t)} 
        \ge \frac{\varepsilon}{2}cn
    \right)
    \ge 
    \frac{\varepsilon}{2}.
\end{align}

On the other hand,
\begin{align}\label{bound_proba_see_copies_t}
    \Pf\!\left(
        \bigcup_{\substack{I\subset \{1,\dots,cn\}\\|I|=\frac{\varepsilon}{2}cn}}
        \bigcap_{i\in I}\mathcal{C}_n(i,t)
    \right)
    \le 
    2^{cn}\,
    \Pf\!\left(\bigcap_{i=1}^{\frac{\varepsilon}{2}cn}\mathcal{C}_n(i,t)\right),
\end{align}
where $2^{cn}$ crudely bounds the number of subsets $I$, and all intersections $\bigcap_{i\in I}\mathcal{C}_n(i,t)$ have the same probability. 
We now estimate $\Pf\!\left(\bigcap_{i=1}^{\frac{\varepsilon}{2}cn}\mathcal{C}_n(i,t)\right)$. 

We denote by $h_1,\dots,h_s$ the holes of $t$ and by $q_1,\dots,q_s$ their respective perimeters. Under the event $\bigcap_{i=1}^{\frac{\varepsilon}{2}cn}\mathcal{C}_n(i,t)$, for each $1 \le i \le \frac{\varepsilon}{2}cn$, let $t_i$ be the copy of $t$ around $e_i^n$ and write $h^i_1,\dots,h^i_s$ for its holes. Note that the copies $t_1,\cdots,t_{\frac{\varepsilon}{2}cn}$ are vertex-disjoint. Let $T_1,\dots,T_k$ be the connected components of $T_{n,g_n,\mathbf{p}^{n}}\setminus \{t_1,\dots,t_{\frac{\varepsilon}{2}cn}\}$. 
This induces a partition $\pi = \{X_1,\dots,X_k\}$ of the set $\bigsqcup_{i=1}^{\frac{\varepsilon}{2}cn}\{h^i_1,\dots,h^i_s\}$, where each $X_j$ consists of the holes glued to the same component $T_j$. 
Similarly, we obtain a partition $\pi'$ of the boundary components $\{\partial_1,\dots,\partial_{\ell_n}\}$ (see Figure~\ref{genus_reducing}).

Let $\alpha_j$ denote the number of holes glued to $T_j$, and $\beta_j$ the number of boundaries $\partial_i$ contained in $T_j$. 
Write $(\mathbf{q}_j,\mathbf{p}_j)$ for the boundary sizes of $T_j$, where $\mathbf{q}_j$ corresponds to the glued holes and $\mathbf{p}_j$ to the original boundaries. 
Let $n_j, b_j \ge 0$ be such that $T_j \in \mathcal{T}_{(\mathbf{q}_j,\mathbf{p}_j)}(n_j,b_j)$. 
Then, a direct computation yields
\begin{align*}
    &\sum_{j=1}^k b_j = g_n - \frac{\varepsilon}{2}cng' -\bigg(\frac{\varepsilon}{2}cn(s-1)-k+1\bigg),\\
    &\sum_{j=1}^k n_j = n + \frac{\varepsilon}{2}cnm,
\end{align*}
where 
\[
m = \frac{1}{2}\left(-|F(t)|+\sum_{i=1}^{r}(q_i - 2)\right).
\]
\begin{figure}[H]
    \centering
    \includegraphics[scale =0.13]{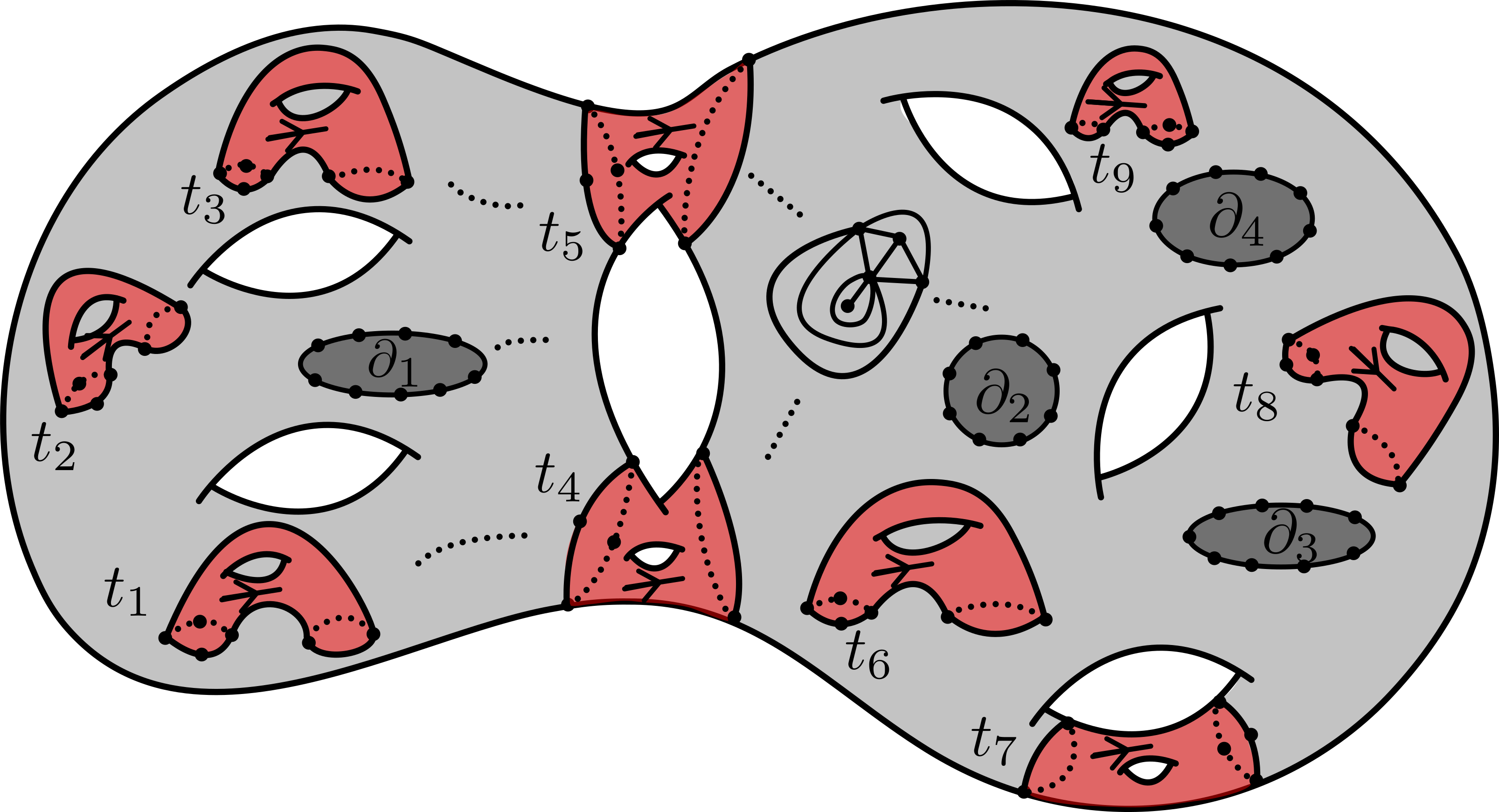}
    \caption{A triangulation $T \in \mathcal{T}_{(9,8,9,9)}(n,21)$ and nine copies of $t$, where $t$ is a triangulation with holes of genus $1$ with two holes of sizes $4$ and $2$. The copies $t_1,\dots,t_9$ split $T$ into two components $T_1,T_2$. The corresponding partitions are $\pi = \{\{h^1_1,h^{1}_2,h^2_1,h^2_2,h^3_1,h^3_2,h^4_1,h^5_1\},\{h^4_2,h^5_2,h^6_1,h^6_2,h^7_1,h^7_2,h^8_1,h^8_2,h^9_1,h^9_2\}\}$ and $\pi' = \{\{\partial_1\},\{\partial_2,\partial_3,\partial_4\}\}$.}
    \label{genus_reducing}
\end{figure}

This yields the bound
\begin{align}\label{sum_to_bound1}
    \Pf\!\left(\bigcap_{i=1}^{\frac{\varepsilon}{2}cn}\mathcal{C}_n(i,t)\right) 
    \le 
    \sum_{\pi,\pi'} 
    \sum_{\substack{n_1+\cdots+n_k = n+\frac{\varepsilon}{2}cn m\\
                   b_1+\cdots+b_k = g_n -\frac{\varepsilon}{2}cng'-(\frac{\varepsilon}{2}cn(s-1)-k+1)}}
    \frac{\prod_{j=1}^{k}\tau_{(\mathbf{q}_j,\mathbf{p}_j)}(n_j,b_j)}
    {(6n-2\sum_{i=1}^{\ell_n}(3-p^n_i))^{\frac{\varepsilon}{2}cn}\tau_{\mathbf{p}^{n}}(n,g_n)}.
\end{align}
The factor $(6n-2\sum_{i=1}^{\ell_n}(3-p^n_i))^{\frac{\varepsilon}{2}cn}$ in the denominator counts the number of choices for $e^n_1,\dots,e^n_{\frac{\varepsilon}{2}cn}$.

Fix $\pi,\pi'$ in the first sum. 
Choosing $c > 0$ small enough, we have for $n$ large enough
\[
\frac{g_n -\frac{\varepsilon}{2}cng'-\bigg(\frac{\varepsilon}{2}cns-k\bigg)}{n+\frac{\varepsilon}{2}cnm} 
\le \frac{g_n}{n}\Big(1+\frac{\varepsilon}{2}cm\Big)^{-1} 
\le \frac{\theta}{2}+\frac{1}{4} < \frac{1}{2}.
\]
Thus, by Proposition~\ref{planarity_estimate}, the contribution of $(\pi,\pi')$ to~\eqref{sum_to_bound1} is bounded by
\[
(6n-2\sum_{i=1}^{\ell_n}(3-p^n_i))^{-\frac{\varepsilon}{2}cn}\tau_{\mathbf{p}^{n}}(n,g_n)^{-1}\exp(C_{\theta} n)\,
n^{2(g_n -\frac{\varepsilon}{2}cng'-(\frac{\varepsilon}{2}cn(s-1)-k+1))+\sum_{j=1}^{k}(\alpha_j+\beta_j -1)}.
\]
By repeatedly applying Lemma~\ref{bounded_ratio_vertices},  Lemma~\ref{add_new_boundary} and Lemma~\ref{estimate_tau}, we obtain the lower bound
\[
\tau_{\mathbf{p}^{n}}(n,g_n) \ge \exp(-c_{\theta}n)\, n^{2g_n+(\ell_n-1)}.
\]
Combining these inequalities, noting that $\sum_{j=1}^{k}\beta_j = \ell_n$, $\sum_{j=1}^{k}\alpha_j = \frac{\varepsilon}{2}cns$ and $\ell_n = o(n)$, we deduce that the contribution of $(\pi,\pi')$ in~\eqref{sum_to_bound1} is bounded by
\[
n^{-\varepsilon cn g' -\frac{\varepsilon}{2}cn(s-1) + k}\exp(C'_{\theta} n).
\]
Now we bound the number of such partitions. 
The number of partitions of a set of $p$ elements into $q$ subsets is at most $\frac{q^p}{q!}$. 
Here $\pi$ partitions the $\frac{\varepsilon}{2}cns$ holes of $\bigsqcup_{i=1}^{\frac{\varepsilon}{2}cn}\{h^i_1,\dots,h^i_s\}$. 
Summing the last bound over all $\pi$ (the number of subsets), we obtain that the contribution of $\pi'$ to \eqref{sum_to_bound1} is bounded by 
\[
n^{-\varepsilon cn g'-\frac{\varepsilon}{2}cn(s-1)} \exp(C'_{\theta} n) \frac{k^{\frac{\varepsilon}{2}csn}}{k!}n^{k}.
\]
Using $k^{\frac{\varepsilon}{2}csn} \le n^{\frac{\varepsilon}{2}csn}$ and 
$
\frac{n^k}{k!} \le \exp(n)$, we finally get the bound
\[
n^{-\varepsilon cn g'+ \frac{\varepsilon}{2}cn} \exp((C'_{\theta}+1)n).
\]
Finally, the number of partitions $\pi'$ is at most $\ell_n^{\ell_n} \le |\mathbf{p}^{n}|^{|\mathbf{p}^{n}|}$. 
Therefore, the right-hand side of~\eqref{bound_proba_see_copies_t} is bounded by
\[
2^{cn}|\mathbf{p}^{n}|^{|\mathbf{p}^{n}|}\, n^{-\varepsilon cn g'+\frac{\varepsilon}{2}cn}\exp((C'_{\theta}+1)n)
\xrightarrow[n\to\infty]{} 0,
\]
since $|\mathbf{p}^{n}| = o(n)$ and $g' \ge 1$.
This contradicts~\eqref{lower_bound_proba_copies} and completes the proof.
\end{proof}

\subsection{One-endedness and finite degrees when rooted in the middle}

By Corollary~\ref{tightness} and Lemma~\ref{T_is_almost_surely_planar_rooted_middle}, any subsequential limit of $(T_{n,g_n,\mathbf{p}^{n}},e^n)$ for $d^{*}_{\mathrm{loc}}$ is a weakly Markovian infinite planar triangulation. 
By Theorem~\ref{weakly_markovian_mixture}, it is therefore of the form $\mathbb{T}_{\Lambda}$, where $\Lambda$ is a random variable taking values in $[0,\lambda_c]\cup \{\star\}$.

\begin{proposition}\label{one_endedness_rooted_middle}
Let $T$ be any subsequential limit of $(T_{n,g_n,\mathbf{p}^{n}},e^n)$. Writing $T = \mathbb{T}_{\Lambda}$, we have 
\[
\Pf(\Lambda\in \{0,\star\}) = 0.
\]
In particular, $T$ is almost surely one-ended and has finite vertex degrees. 
\end{proposition}

\begin{proof}
First, note that proving $\Pf(\Lambda\in \{0,\star\}) = 0$ implies that $T$ has almost surely finite vertex degrees and is one-ended. 
Indeed, for every $\lambda \in (0,\lambda_c]$, the triangulation $\mathbb{T}_{\lambda}$ almost surely has finite degrees and is one-ended. Thus, we only prove the first part of the proposition.

The argument used here follows the classical tightness argument of \cite{Angel_UITP}. 
For any $r \ge 0$, write $\mathbb{T}_{0,r} = B^{*}_r(\mathbb{T}_0)$ for the dual ball of radius $r$ in $\mathbb{T}_0$. 
It follows that, for $n$ large enough,
\[
\Pf\big(B^{*}_r(T_{n,g_n,\mathbf{p}^{n}},e^n) = \mathbb{T}_{0,r}\big) 
\ge \frac{\Pf(\Lambda = 0)}{2} > 0.
\]

Recall the definition of the peeling exploration from Section~\ref{peeling_def}. 
Fix any peeling algorithm $\mathcal{A}$ and consider the peeling process $(\mathcal{E}^{\mathcal{A}}_n(k))_{k\ge 0}$ applied to $(T_{n,g_n,\mathbf{p}^{n}},e^n)$. 
Under the event $B^{*}_r(T_{n,g_n,\mathbf{p}^{n}},e^n) = \mathbb{T}_{0,r}$, for any $k \in \{0,\dots,r-1\}$, the triangulation with a hole $\mathcal{E}^{\mathcal{A}}_n(k)$ is planar with a single boundary of perimeter $3+k$. 
Moreover, at each step $\mathcal{E}^{\mathcal{A}}_n(k) \subset \mathcal{E}^{\mathcal{A}}_n(k+1)$, the third vertex of the discovered triangle does not belong to the current boundary of $\mathcal{E}^{\mathcal{A}}_n(k)$.

Let $\mathcal{L}_{n}(k)$ denote the event that the triangle discovered behind the peeled edge $\mathcal{A}(\mathcal{E}^{\mathcal{A}}_n(k))$ creates a loop at the left endpoint of that edge, and that this loop is immediately filled by an edge. 
The probability of this event satisfies
\[
\Pf\big(\mathcal{L}_{n}(k)\,\big|\, \mathcal{E}^{\mathcal{A}}_n(k) \subset \mathbb{T}_{0,r}\big)
= 
\frac{\tau_{(3+k,p^n_1,\dots,p^n_{\ell_n})}(n-1,g_n)}
     {\tau_{(3+k,p^n_1,\dots,p^n_{\ell_n})}(n,g_n)}
\ge c_{\theta} > 0,
\]
where the inequality follows from Lemma~\ref{bounded_ratio_vertices}. 
Consequently,
\[
\Pf\big(\mathcal{E}^{\mathcal{A}}_n(r) \subset B^{*}(\mathbb{T}_0,r)\big) 
\le (1-c_{\theta})^r.
\]
We deduce that for any $r \ge 0$, we have  $(1-c_{\theta})^r \ge \frac{\Pf(\Lambda = 0)}{2}$. Thus $\Pf(\Lambda = 0) = 0$.

\medskip
The proof that $\Pf(\Lambda = \star) = 0$ proceeds analogously, replacing the event $\mathcal{L}_{n}(k)$ by $\mathcal{C}_n(k)$, which denotes the event that the triangle revealed behind the peeled edge has its third vertex lying outside $\mathcal{E}^{\mathcal{A}}_n(k)$.
\end{proof}

\subsection{$\mathbb{T}_{\lambda}$ as the limit of $(T_{n,g_n,\mathbf{p}^{n}},e^n)$}
Let $T$ be a subsequential limit of $(T_{n,g_n,\mathbf{p}^{n}},e^n)$ with respect to $d^{*}_{\mathrm{loc}}$, where $e^n$ is an oriented edge chosen uniformly at random in $T_{n,g_n,\mathbf{p}^{n}}$. By Proposition~\ref{one_endedness_rooted_middle}, we can write $T = \mathbb{T}_{\Lambda}$ where $\Lambda$ is a random variable taking values in $(0,\lambda_c]$. Since the vertices of $\mathbb{T}_{\Lambda}$ have almost surely finite degrees, it follows from \cite[Lemma~3]{Budzinski_2020} that the convergence in distribution also holds for $d_{\mathrm{loc}}$. 
The next proposition shows that $\Lambda$ is deterministic, thereby proving the first convergence stated in Theorem~\ref{local_limit_middle}. We recall that $\lambda(\theta)$ is defined in \eqref{deflambdatheta}.

\begin{proposition}\label{deterministic_Lambda}
Almost surely, $\Lambda = \lambda(\theta)$.
\end{proposition}

\begin{proof}
We do not provide all the details of the proof since they are the same as in \cite{Budzinski_2020}. 
Let $e^n_1$ and $e^n_2$ be two independent uniformly chosen oriented edges in $T_{n,g_n,\mathbf{p}^{n}}$. 
Since both triangulations $(T_{n,g_n,\mathbf{p}^{n}},e^n_1)$ and $(T_{n,g_n,\mathbf{p}^{n}},e^n_2)$ converge in distribution with respect to $d_{\mathrm{loc}}$, their joint distribution 
\[
((T_{n,g_n,\mathbf{p}^{n}},e^n_1),(T_{n,g_n,\mathbf{p}^{n}},e^n_2))
\]
is tight. 
Fix a subsequential limit $(\mathbb{T}_{\Lambda^1},\mathbb{T}_{\Lambda^2})$. 
By the Skorokhod representation theorem, we may assume that the convergence holds almost surely. 
Then, using Lemma~\ref{T_do_not_see_boundaries__rooted_middle} and Proposition~\ref{two_uniform_edges_stay_far}, the same proof as in \cite[Proposition~18]{Budzinski_2020} applies, yielding $\Lambda^1 = \Lambda^2$ almost surely. Moreover, by \cite[Lemma~20]{Budzinski_2020}, there exists a measurable function 
\[
\ell : \bigcup_{n\ge 1}\mathcal{T}_{\mathbf{p}^{n}}(n,g_n) \longrightarrow (0,\lambda_c]
\]
such that 
\[
\big((T_{n,g_n,\mathbf{p}^{n}},e^n),\ell(T_{n,g_n,\mathbf{p}^{n}},e^n)\big)
\;\xrightarrow[n\to\infty]{(d)}\;
(\mathbb{T}_{\Lambda},\Lambda).
\]
In words, this means that asymptotically the value of $\Lambda$ can be read directly from $(T_{n,g_n,\mathbf{p}^{n}},e^n)$. It follows that 
\[
\big(\ell(T_{n,g_n,\mathbf{p}^{n}},e^n_1),\ell(T_{n,g_n,\mathbf{p}^{n}},e^n_2)\big)
\xrightarrow[n\to\infty]{(d)} (\Lambda,\Lambda).
\]
Let $f$ be any bounded continuous function on $(0,\lambda_c]$. 
Then,
\[
f\!\left(\ell(T_{n,g_n,\mathbf{p}^{n}},e^n_1)\right) 
- f\!\left(\ell(T_{n,g_n,\mathbf{p}^{n}},e^n_2)\right)
\xrightarrow[n\to\infty]{L^1} 0.
\]
Let $ d(\lambda) := \mathbb{E}\!\left[\frac{1}{\deg_{\mathbb{T}_{\lambda}}(\rho)}\right]$ denote the mean inverse degree of the root in $\mathbb{T}_{\lambda}$. 
From \cite[Proposition~20]{Budzinski_2020}, the function $d$ is strictly increasing on $[0,\lambda_c]$. We can then write
\[
\mathbb{E}\!\left[\frac{1}{\deg_{T_{n,g_n,\mathbf{p}^{n}}}(\rho^n_1)}
    f\!\left(\ell(T_{n,g_n,\mathbf{p}^{n}},e^n_1)\right)\right]
=
\mathbb{E}\!\left[\frac{1}{\deg_{T_{n,g_n,\mathbf{p}^{n}}}(\rho^n_1)}
    f\!\left(\ell(T_{n,g_n,\mathbf{p}^{n}},e^n_2)\right)\right] + o(1).
\]
The left-hand side converges to $\mathbb{E}[d(\Lambda)f(\Lambda)]$. 
Conditionally on $(T_{n,g_n,\mathbf{p}^{n}},e^2_n)$, the expectation of $\frac{1}{\deg_{T_{n,g_n,\mathbf{p}^{n}}}(\rho^n_1)}$ equals
\[
\frac{\#V(T_{n,g_n,\mathbf{p}^{n}})}{2\,\#E(T_{n,g_n,\mathbf{p}^{n}})} 
\;\longrightarrow\; \frac{1-2\theta}{6}.
\]
Hence,
\[
\mathbb{E}[d(\Lambda)f(\Lambda)] 
= 
\frac{1-2\theta}{6}\,\mathbb{E}[f(\Lambda)].
\]
We therefore conclude that $d(\Lambda) = \frac{1-2\theta}{6}$ almost surely, and by injectivity of $d$, this implies that $\Lambda$ is almost surely constant, equal to $\lambda(\theta)$.
\end{proof}
We can now complete the proof of Theorem~\ref{local_limit_middle}.
\begin{proof}
The convergence $(T_{n,g_n,\mathbf{p}^{n}},e^n) \to \mathbb{T}_{\lambda}$ follows directly from Proposition~\ref{deterministic_Lambda}. 
We briefly explain the proof of
\[
(T_{n,g_n,(p,\mathbf{p}^{n})},e^n_1) \to \mathbb{T}_{\lambda}^{(p)}.
\]
For a fixed $p \ge 1$, let $t$ be a rooted planar triangulation with one hole of perimeter $p$ and $m$ vertices. 
The value of $m$ is arbitrary but chosen so that such a triangulation $t$ exists. 
The first part of the theorem gives the convergence 
\[
(T_{n+m-p,g_n,\mathbf{p}^{n}},e^n) \to \mathbb{T}_{\lambda}.
\]
On the other hand, conditionally on $t \subset (T_{n+m-p,g_n,\mathbf{p}^{n}},e^n) $, we have
\[
(T_{n+m-p,g_n,\mathbf{p}^{n}},e^n) \setminus t \;\overset{(d)}{=}\; T_{n,g_n,(p,\mathbf{p}^{n})}.
\]
Combining these two facts, we deduce that 
\[
T_{n,g_n,(p,\mathbf{p}^{n})} \to T,
\]
where $T$ has the same distribution as $\mathbb{T}_{\lambda} \setminus t$ conditionally on the event $\{t \subset \mathbb{T}_{\lambda}\}$. 
This is precisely the law of $\mathbb{T}_{\lambda}^{(p)}$. This completes the proof.
\end{proof}

\subsection{Combinatorial consequences}

The fact that $T = \mathbb{T}_{\Lambda}$ with $\Lambda \in (0,\lambda_c]$ is non-trivial and relies on the results of \cite{budzinski2022multiendedmarkoviantriangulationsrobust}. In particular, the one-endedness of $T$ follows directly from Proposition~\ref{one_endedness_rooted_middle}. This property can be rephrased combinatorially, leading to non-trivial combinatorial estimates that will be useful throughout the rest of the paper. In particular, for a triangulation $t$ with two holes $h_1$ and $h_2$, under the event $t \subset (T_{n,g_n,\mathbf{p}^{n}},e^n)$, it is very unlikely that $(T_{n,g_n,\mathbf{p}^{n}},e^n) \setminus t$ remains connected. 
Moreover, if $(T_{n,g_n,\mathbf{p}^{n}},e^n) \setminus t$ has two connected components, then with high probability one of them contains all the genus, all the boundaries $\partial_1,\dots,\partial_{\ell_n}$, and almost all the vertices (see Proposition~\ref{combi_estimate_rooted_middle}). 
We begin with a simpler estimate that captures part of this phenomenon.

\begin{proposition}\label{estimate_genus_reducing}
 Fix $\theta \in [0,\frac{1}{2})$, $\displaystyle \frac{g_n}{n} \to \theta$ and $\mathbf{p}^{n} = (p^n_{1},\cdots,p^n_{\ell_n})$ such that $|\mathbf{p}^{n}|= o(n)$. We have the convergence
\[
\frac{\tau_{\mathbf{p}^{n}}(n,g_n-1)}{\tau_{\mathbf{p}^{n}}(n,g_n)} \longrightarrow 0.
\]
\end{proposition}

\begin{proof}
Fix a triangulation $t$ of genus $1$ with one hole of perimeter $1$ and $7$ vertices in total 
(the choice of $7$ is arbitrary and ensures that such a triangulation exists). By Theorem~\ref{local_limit_middle}, we have 
\[
\Pf(t \subset (T_{n,g_n,\mathbf{p}^{n}},e^n)) \longrightarrow 0.
\]
On the other hand, combining Lemmas~\ref{bounded_ratio_vertices} and~\ref{add_new_boundary}, there exists a constant $C_{\theta}>0$ such that, for $n$ large enough,
\[
\Pf(t \subset (T_{n,g_n,\mathbf{p}^{n}},e^n))
= 
\frac{\tau_{1,\mathbf{p}^n}(n-9,g_n-1)}{(6n - 2\sum_{i=1}^{\ell_n}(3-p^n_i))\,\tau_{\mathbf{p}^{n}}(n,g_n)} \ge C_{\theta}\frac{\tau_{\mathbf{p}^{n}}(n,g_n-1)}{\tau_{\mathbf{p}^{n}}(n,g_n)} \longrightarrow 0.
\]
This concludes the proof.
\end{proof}

\medskip
For $a>0$, define $\mathcal{M}(n,a)$ as the set of all pairs $((n_1,g_1,I),(n_2,g_2,J))$ satisfying
\begin{align}\label{definition_M}
\begin{cases}
n_1 + n_2 = n,\\
g_1 + g_2 = g_n,\\
I \sqcup J = \{1,\dots,\ell_n\},\\
n_1 + 2 - 2g_1 \ge \dfrac{n + 2 - 2g_n}{2},\\
n_2 \ge a \ \text{or}\ g_2 \ge 1 \ \text{or}\ J \neq \emptyset.
\end{cases}
\end{align}

\begin{proposition}\label{combi_estimate_rooted_middle}
 Fix $\theta \in [0,\frac{1}{2})$, $\displaystyle \frac{g_n}{n} \to \theta$ and $\mathbf{p}^{n} = (p^n_{1},\cdots,p^n_{\ell_n})$ such that $|\mathbf{p}^{n}|= o(n)$. For any $r_1,r_2 \ge 0$ and any $\varepsilon > 0$, there exists $a > 0$ such that, for all sufficiently large $n$,
\[
(n\tau_{\mathbf{p}^{n}}(n,g_n))^{-1}
\sum_{((n_1,g_1,I),(n_2,g_2,J)) \in \mathcal{M}(n,a)}
\tau_{(r_1,\mathbf{p}_{I}^{n})}(n_1,g_1)\,
\tau_{(r_2,\mathbf{p}_{J}^{n})}(n_2,g_2)
\le \varepsilon.
\]
Moreover,
\[
\frac{\tau_{r_1,r_2,\mathbf{p}^{n}}(n,g_n-1)}{n\tau_{\mathbf{p}^{n}}(n,g_n)} \longrightarrow 0.
\]
In particular,
\[
\frac{\tau_{\mathbf{p}^{n}}(n,g_n-1)}{\tau_{\mathbf{p}^{n}}(n,g_n)} = o(n^{-1}).
\]
\end{proposition}

\begin{proof}
The last statement follows from the second one by setting $r_1 = r_2 = 1$ and using Lemma~\ref{add_new_boundary}.

Fix $r_1, r_2$ and $\varepsilon > 0$. 
We start by proving the first inequality. 
Consider a triangulation $t$ with two holes $h_1$ and $h_2$ of perimeters $r_1$ and $r_2$, respectively. 
Let 
\[
m = \frac{1}{2}\big(-|F(t)| + r_1 - 2 + r_2 - 2\big).
\]
By Theorem~\ref{local_limit_middle}, the rooted triangulation $(T_{n-m,g_n,\mathbf{p}^{n}},e^n)$ converges in distribution to $\mathbb{T}_{\lambda}$. 
Under the event $t \subset \mathbb{T}_{\Lambda}$, let $\mathfrak{t}_1$ and $\mathfrak{t}_2$ denote the triangulations filling the holes $h_1$ and $h_2$. 
Since $\mathbb{T}_{\Lambda}$ is almost surely one-ended, we have
\begin{align}\label{limit_one_ended_combinatoric}
\Pf\!\left(t \subset \mathbb{T}_{\lambda},\,
\#V(\mathfrak{t}_1) = +\infty,\,
\mathfrak{t}_2 \notin \bigcup_{k=0}^{a}\mathcal{T}_{r_2}(k,0)
\right) 
\underset{a\to\infty}{\longrightarrow} 0.
\end{align}

Now, define the event $\mathcal{D}_n$ on which $t \subset (T_{n-m,g_n,\mathbf{p}^{n}},e^n)$ and $T_{n-m,g_n,\mathbf{p}^{n}} \setminus t$ is disconnected. 
Under $\mathcal{D}_n$, write $T_1$ and $T_2$ for the triangulations filling $h_1$ and $h_2$, respectively, and denote
\[
(T_1,T_2) \in 
\mathcal{T}_{(r_1,\mathbf{p}^{n}_{I})}(n_1,g_1) \times 
\mathcal{T}_{(r_2,\mathbf{p}^{n}_{J})}(n_2,g_2),
\]
with $n_1 + n_2 = n$, $g_1 + g_2 = g_n$, and $I \sqcup J = \{1,\dots,\ell_n\}$. 
Under $\mathcal{D}_n$, the pair $((n_1,g_1,I),(n_2,g_2,J))$ is random. 
For any $a \ge 0$, the event 
\[
\#V(T_1) \ge \frac{n + 2 - 2g_n}{2} 
\quad\text{and}\quad
T_2 \notin \bigcup_{k=0}^{a}\mathcal{T}_{r_2}(k,0)
\]
is equivalent to $((n_1,g_1,I),(n_2,g_2,J)) \notin \mathcal{M}(n,a)$. 
Moreover, by Theorem~\ref{local_limit_middle},
\begin{align*}
\limsup_{n}\Pf\!\left(\mathcal{D}_n,\,
\#V(T_1) \ge \frac{n + 2 - g_n}{2},\,
T_2 \notin \bigcup_{k=0}^{a}\mathcal{T}_{r_2}(k,0)
\right)
\le 
\Pf\!\left(t \subset \mathbb{T}_{\lambda},\,
\#\mathfrak{t}_1 = +\infty,\,
\mathfrak{t}_2 \notin \bigcup_{k=0}^{a}\mathcal{T}_{r_2}(k,0)
\right).
\end{align*}
Then, by~\eqref{limit_one_ended_combinatoric}, there exists $a > 0$ such that, for $n$ large enough
\[
\Pf\!\left(\mathcal{D}_n,\,
\#V(T_1) \ge \frac{n + 2 - g_n}{2},\,
T_2 \notin \bigcup_{k=0}^{a}\mathcal{T}_{r_2}(k,0)
\right) \le \varepsilon.
\]
Rewriting the left-hand side gives the bound
\[
\big((6n - 2\sum_{i=1}^{\ell_n}(3-p^n_i))\,
\tau_{\mathbf{p}^{n}}(n-m,g_n)\big)^{-1}
\sum_{\substack{(n_1,g_1,I),(n_2,g_2,J)\in \mathcal{M}(n,a)}}
\tau_{(r_1,\mathbf{p}_{I}^n)}(n_1,g_1)\,
\tau_{(r_2,\mathbf{p}_{J}^n)}(n_2,g_2)
\le \varepsilon.
\]
Finally, using Lemma~\ref{bounded_ratio_vertices} and the fact that $|\mathbf{p}^n| = o(n)$, there exists $C_{\theta,r_1,r_2} > 0$ such that, for $n$ large enough,
\[
(n\tau_{\mathbf{p}^{n}}(n,g_n))^{-1}
\sum_{\substack{(n_1,g_1,I),(n_2,g_2,J)\in \mathcal{M}(n,a)}}
\tau_{(r_1,\mathbf{p}_{I}^n)}(n_1,g_1)\,
\tau_{(r_2,\mathbf{p}_{J}^n)}(n_2,g_2)
\le C_{\theta,r_1,r_2}\,\varepsilon.
\]
Since this holds for any $\varepsilon > 0$, the first inequality follows. 
The proof of the second estimate is similar, considering instead the event $\mathcal{C}_n$ on which $t \subset T_{n,g_n}$ and $T_{n,g_n} \setminus t$ is connected.
\end{proof}

In the case $\mathbf{p}^n = \emptyset$, Proposition~\ref{combi_estimate_rooted_middle} yields
\[
\frac{\tau(n,g_n-1)}{\tau(n,g_n)} = o(n^{-1}).
\]
This is weaker than the estimate obtained from the Goulden–Jackson recursion formula~\cite{GOULDEN2008932}, which gives $\frac{\tau(n,g_n-1)}{\tau(n,g_n)} = \mathcal{O}(n^{-2})$. 
However, our combinatorial estimate is quantitative and derived solely from coarse combinatorial bounds and probabilistic arguments. 
We thus expect that similar bounds can be obtained for related models using the same approach. 
Moreover, our result holds in the more general regime $|\mathbf{p}^n| = o(n)$.

\medskip
Theorem~\ref{local_limit_middle} also allows us to compute the asymptotic ratio $\frac{\tau_{\mathbf{p}^{n}}(n-1,g_n)}{\tau_{\mathbf{p}^{n}}(n,g_n)}$.

\begin{proposition}\label{application_cv_ratio}
 Fix $\theta \in [0,\frac{1}{2})$, $\displaystyle \frac{g_n}{n} \to \theta$ and $\mathbf{p}^{n} = (p^n_{1},\cdots,p^n_{\ell_n})$ such that $|\mathbf{p}^{n}|= o(n)$. We have the limits
\[
\frac{\tau_{(1,\mathbf{p}^{n})}(n,g_n)}{6n\,\tau_{\mathbf{p}^{n}}(n,g_n)} \underset{n\to +\infty}{\to} 1,
\qquad
\frac{\tau_{\mathbf{p}^{n}}(n-1,g_n)}{\tau_{\mathbf{p}^{n}}(n,g_n)} \underset{n\to +\infty}{\to} \lambda(\theta).
\]
\end{proposition}

\begin{proof}
By Lemma~\ref{T_do_not_see_boundaries__rooted_middle},
\begin{align}\label{esti}
\Pf\!\big(d(\rho^n,\partial T_{n,g_n,\mathbf{p}^n}) \le 1\big) \underset{n\to +\infty}{\to}0.
\end{align}
Given $t \in \mathcal{T}_{\mathbf{p}^{n}}(n,g_n)$ and an oriented edge $e$ whose starting vertex $x$ does not lie on a boundary, define $\varphi(t,e) \in \mathcal{T}_{(1,\mathbf{p}^{n})}(n,g_n)$ as the triangulation obtained by splitting $e$ into a digon filled with a loop at $x$. 
It is straightforward to verify that $\varphi$ defines a bijection. 
Combining~\eqref{esti} with this observation yields
\[
\frac{\tau_{(1,\mathbf{p}^{n})}(n,g_n)}{(6n + 2\sum_{i=1}^{\ell_n}(3-p^n_i))\,\tau_{\mathbf{p}^{n}}(n,g_n)} \underset{n\to +\infty}{\to}1.
\]
The first convergence follows from the fact that 
$6n + 2\sum_{i=1}^{\ell_n}(3-p^n_i) = 6n + o(n)$ since $|\mathbf{p}^n| = o(n)$.

For the second estimate, fix a triangulation $t$ with one hole of perimeter $1$ and exactly two vertices. 
By Theorem~\ref{local_limit_middle},
\[
\frac{\tau_{(1,\mathbf{p}^{n})}(n-1,g_n)}{6n\,\tau_{\mathbf{p}^{n}}(n,g_n)}
= 
\Pf(t \subset \mathbb{T}_{n,g_n,\mathbf{p}^{n}}) 
\underset{n\to +\infty}{\to}
\Pf(t \subset \mathbb{T}_{\lambda}) 
= \lambda(\theta).
\]
Using the first part of the proposition, we then have
\[
\frac{\tau_{(1,\mathbf{p}^{n})}(n-1,g_n)}{6n\,\tau_{\mathbf{p}^{n}}(n-1,g_n)} \underset{n\to +\infty}{\to}1,
\]
which completes the proof.
\end{proof}

\section{Local limit from the boundaries of $T_{n,g_n,\mathbf{p}^{n}}$}\label{section_boundary}

In this section, let $\theta \in [0,\frac{1}{2})$ and let $n,g_n \ge 0$ be such that $n + 2 - 2g_n > 0$ and $\displaystyle \frac{g_n}{n} \to \theta$. 
We also introduce $\mathbf{p}^{n} = (p^n_{1},\dots,p^n_{\ell_n})$ satisfying $|\mathbf{p}^{n}| = o(n)$ and $\frac{\ell_n}{|\mathbf{p}^n|} \to 0$.

\medskip
Throughout the rest of the paper, the notation $e^n$ refers to a uniformly chosen oriented edge on $\partial T_{n,g_n,\mathbf{p}^{n}}$, oriented so that a boundary face lies to the right of $e^n$. 
We denote by $\rho^n$ the starting vertex of $e^n$. For each $1 \le i \le \ell_n$, we also denote by $e^n_i$ the distinguished oriented edge on the $i^{\mathrm{th}}$ boundary and by $\rho^n_i$ its starting vertex.

\medskip
The goal of this section is to prove Theorem~\ref{local_limit_boundary}. 
Namely, we aim to show that, for the topology $d_{\mathrm{loc}}$, the local limit of $(T_{n,g_n,\mathbf{p}^{n}},e^n)$ is $\mathbb{H}_{\lambda(\theta)}$, where $\lambda(\theta)$ is defined in~\eqref{deflambdatheta}.

The condition $|\mathbf{p}^{n}| = o(n)$ and $\frac{\ell_n}{|\mathbf{p}^{n}|} \to 0$ can be interpreted as follows: for any fixed $A > 0$, the probability that $e^n$ lies on a boundary of perimeter less than $A$ tends to $0$ as $n \to \infty$. 
More precisely, we introduce the event
\begin{align}\label{event_boundary_perimeter_typically_large}
\mathcal{P}_n := \Big\{ e^n \text{ lies on a boundary with perimeter at least } \Big(\frac{|\mathbf{p}^n|}{\ell_n}\Big)^{1/2} \Big\}.
\end{align}

\begin{proposition}\label{perimeter_typically_large}
We have 
\[
\Pf(\mathcal{P}_n) \xrightarrow[n\to\infty]{} 1.
\]
\end{proposition}

\begin{proof}
We can write
\begin{align*}
\Pf(\mathcal{P}_n^{c})
&= \Pf\Big(\exists\, i \in \{1,\dots,\ell_n\} : e^n \in \partial_i \text{ and } p^n_i < \Big(\frac{|\mathbf{p}^n|}{\ell_n}\Big)^{1/2}\Big) \\
&= |\mathbf{p}^n|^{-1} \sum_{\substack{i=1\\ p^n_i \le (\frac{|\mathbf{p}^n|}{\ell_n})^{1/2}}}^{\ell_n} p^n_i \\
&\le |\mathbf{p}^n|^{-1}\, \ell_n \Big(\frac{|\mathbf{p}^n|}{\ell_n}\Big)^{1/2}
= \Big(\frac{\ell_n}{|\mathbf{p}^n|}\Big)^{1/2} \xrightarrow[n\to\infty]{} 0.
\end{align*}
\end{proof}

\medskip
We now outline the strategy used to prove the local limit result.  
In Sections~\ref{small_holes}, \ref{subsection_pathological} and \ref{conv_subsequences}, we establish the tightness of the sequence $(T_{n,g_n,\mathbf{p}^{n}},e^n)$ for the topology $d_{\mathrm{loc}}$ (see Proposition~\ref{tightness_d_loc}). We also prove that any subsequential limit is a triangulation of the half-plane.
This is the most technical part of the paper. 
Then, in Section~\ref{last_section}, we will show that the only possible limit is $\mathbb{H}_{\lambda(\theta)}$. This last step will be relatively easy because we already have Proposition~\ref{application_cv_ratio}.

\paragraph*{Peeling cases to avoid.}
To prove tightness, we use an appropriate peeling exploration to reveal the local neighbourhood of $e^n$. 
However, we must ensure that the triangles revealed during the exploration do not fall into certain pathological cases, which we now describe. 
Let $t$ be the triangle adjacent to $e^n$, and let $z$ be its third vertex. 
Assume that $e^n \in \partial_i$ for some $1 \le i \le \ell_n$. 
The pathological situations are the following:

\begin{enumerate}
  \item \label{list_pathological_cases1}
  The vertex $z$ is the $k^{\mathrm{th}}$ vertex to the left (resp. right) of $\rho^n$ on $\partial_i$ and $T_{n,g_n,\mathbf{p}^{n}} \setminus t$ is connected.  
  This case is handled in Proposition~\ref{Finite_holes_filled_by_finite_maps}.

  \item \label{list_pathological_cases2}
  The vertex $z$ is the $k^{\mathrm{th}}$ vertex to the left (resp. right) of $\rho^n$ on $\partial_i$, and $T_{n,g_n,\mathbf{p}^{n}} \setminus t$ is disconnected, but the triangulation filling the left (resp. right) created hole of perimeter $k+1$ is either non-planar, contains other boundaries $\{\partial_1,\dots,\partial_{\ell_n}\}\setminus \{\partial_i\}$, or has a number of vertices much larger than~$1$.  
  This is also treated in Proposition~\ref{Finite_holes_filled_by_finite_maps}.

  \item \label{list_pathological_cases3}
  The discovered triangle $t$ splits $\partial_i$ into two holes of perimeter much larger than~$1$.  
  This case is treated in Proposition~\ref{no_event_H_or_O}.

  \item \label{list_pathological_cases4}
  The vertex $z$ lies on another boundary $\partial_j$ with $j \neq i$.  
  This is handled in Proposition~\ref{no_event_H_or_O}.
\end{enumerate}

\medskip
Excluding these cases ensures that a \emph{typical} local neighbourhood of $(T_{n,g_n,\mathbf{p}^{n}},e^n)$ resembles that of a triangulation of the half-plane. Items~\ref{list_pathological_cases3} and~\ref{list_pathological_cases4} are the most delicate to rule out.  
To deal with them, we crucially use the fact that $e^n$ is chosen uniformly among all boundary edges.  
If $e^n$ were conditioned to lie on a specific boundary, say $\partial_1$, we could not complete the proof.

\medskip
Fix $n,g \ge 0$ such that $n + 2 - 2g > 0$, and let $\mathbf{p} \in \mathbb{N}^{\ell}$, $t \in \mathcal{T}_{\mathbf{p}}(n,g)$, and $e \in \partial_i$ for some $1 \le i \le \ell$. 
Let $\rho$ denote the starting vertex of $e$ and $z$ the third vertex of the triangle $t_0$ adjacent to $e$. 
For $k \in \{0,\dots,p_i\}$:

\begin{itemize}
    \item[$\bullet$]
    We say that $(t,e)$ satisfies $\mathcal{L}_k$ (resp. $\mathcal{R}_k$) if $z$ is the $k^{\mathrm{th}}$ vertex to the left (resp. to the right) of $\rho$ on $\partial_i$.  
    This operation creates two holes: one of perimeter $k+1$ on the left (resp. right) of $e$, and another of perimeter $p_i - k$ on the right (resp. the left) of $e$.  
    If $t \setminus t_0$ is connected, we say that $(t,e)$ satisfies $\mathcal{L}_k^{\mathrm{nsep}}$ (resp. $\mathcal{R}_k^{\mathrm{nsep}}$), and in this case 
    \[
    t \setminus t_0 \in \mathcal{T}_{(p_1,\dots,k+1,p_i-k,\dots,p_{\ell})}(n-1,g-1).
    \]
    Otherwise, we say that $(t,e)$ satisfies $\mathcal{L}_k^{\mathrm{sep}}$ (resp. $\mathcal{R}_k^{\mathrm{sep}}$).  
    In that case, writing $t \setminus t_0 = t_1 \cup t_2$, we have
    \[
    (t_1,t_2) \in \bigsqcup_{\substack{n_1+n_2=n-1\\ g_1+g_2=g \\ I\sqcup J = \{1,\dots,\ell\}\setminus\{i\}}}
    \mathcal{T}_{(k+1,\mathbf{p}_{I})}(n_1,g_1) \times \mathcal{T}_{(p_i-k,\mathbf{p}_{J})}(n_2,g_2).
    \]

    \item[$\bullet$]
    We say that $(t,e)$ satisfies $(\mathcal{H}_k)$ if $z \in \partial_i$ and both resulting holes have perimeter at least $k$.

    \item[$\bullet$]
    We say that $(t,e)$ satisfies $(\mathcal{O})$ if $z \in \partial_j$ for some $j \neq i$.
\end{itemize}

Note that the events $(\mathcal{H}_k)$ and $(\mathcal{O})$ are disjoint, since different boundaries are vertex-disjoint.

\medskip
For any $n,k \ge 0$ and any random oriented edge $e_n$ on $\partial T_{n,g_n,\mathbf{p}^{n}}$, we denote by $\mathcal{L}_{n,k}(e_n)$ the event that $(T_{n,g_n,\mathbf{p}^{n}},e_n)$ satisfies $\mathcal{L}_k$. 
We define analogously the events $\mathcal{L}^{\mathrm{sep}}_{n,k}(e_n)$, $\mathcal{L}^{\mathrm{nsep}}_{n,k}(e_n)$, $\mathcal{R}_{n,k}(e_n)$, $\mathcal{R}^{\mathrm{sep}}_{n,k}(e_n)$, $\mathcal{R}^{\mathrm{nsep}}_{n,k}(e_n)$, $\mathcal{H}_k(e_n)$, and $\mathcal{O}_n(e_n)$.
\subsection{Small holes are filled by small planar triangulations}\label{small_holes}

This section is devoted to ruling out the first two pathological cases~\eqref{list_pathological_cases1} and~\eqref{list_pathological_cases2}.  
Under the event $\mathcal{L}_{n,k}(e^n)$, we denote by $T^1(e^n)$ the triangulation filling the left hole of perimeter $k+1$ (see Figure~\ref{small_holes_not_too_big}). Note that denoting by $t$ the triangle incident to $e^n$, under the event $\mathcal{L}_{n,k}(e^n)$, we might have $T_{n,g_n,\mathbf{p}^n }\backslash t$ connected. In that case, we have $T^1(e^n) = T_{n,g_n,\mathbf{p}^n }\backslash t$. We will show that, conditionally on $\mathcal{L}_{n,k}(e^n)$, the event $\mathcal{L}_{n,k}^{\mathrm{sep}}(e^n)$ occurs with high probability and that $T^1(e^n)$ is a planar triangulation of the $(k+1)$-gon.  
Fortunately, this is almost a corollary of Proposition~\ref{combi_estimate_rooted_middle}.

\medskip
We first prove that it is unlikely that $\mathcal{L}_{n,k}^{\mathrm{sep}}(e^n)$ occurs and that the triangulation $T^{1}(e^n)$ contains a positive proportion of all faces.

\begin{proposition}\label{small_hole_has_small_proportion_faces}
For any $k,\delta > 0$, we have 
\[
\Pf\big(\mathcal{L}_{n,k}^{\mathrm{sep}}(e^n),\, \#F(T^1(e^n)) \ge \delta n\big) \xrightarrow[n\to\infty]{} 0.
\]
\end{proposition}

\begin{proof}
Fix $k,\delta > 0$ and $\varepsilon > 0$.  
By Proposition~\ref{perimeter_typically_large}, it suffices to show that for all $n$ large enough and for any $1 \le i \le \ell_n$ with $p^n_i \ge \frac{12(k+1)}{\delta \varepsilon}$, we have 
\begin{align}\label{boundi}
\Pf\big(\mathcal{L}_{n,k}^{\mathrm{sep}}(e^n),\, \#F(T^1(e^n)) \ge \delta n \,\big|\, e^n \in \partial_i\big) \le \frac{\varepsilon}{2}.
\end{align}
The left-hand side of~\eqref{boundi} can be rewritten as
\[
\Pf\big(\mathcal{L}_{n,k}^{\mathrm{sep}}(e^n_i),\, \#F(T^1(e^n_i)) \ge \delta n\big).
\]
Since the order of the boundaries plays no role, we may assume $p^n_1 \ge \frac{12(k+1)}{\delta \varepsilon}$ and prove
\[
\Pf\big(\mathcal{L}_{n,k}^{\mathrm{sep}}(e^n_1),\, \#F(T^1(e^n_1)) \ge \delta n\big) \le \frac{\varepsilon}{2}.
\]

For $0 \le i < p^n_1$, let $u_i$ denote the $i^{\mathrm{th}}$ edge to the right of $e^n_1$.  
Define the event $\mathcal{B}_{n,i}$ where $\mathcal{L}_{n,k}^{\mathrm{sep}}(u_i)$ occurs and $\#F(T^1(u_i)) \ge \delta n$.  
We claim that
\begin{align}\label{bound_expectation}
\#\{\, i \in \{0,\dots,p^n_1-1\} : \mathcal{B}_{n,i} \text{ occurs} \,\} \le \frac{6(k+1)}{\delta}.
\end{align}

Indeed, suppose that
\[
\#\{\, i \in \{0,\dots,p^n_1-1\} : \mathcal{B}_{n,i} \text{ occurs} \,\} > \frac{6(k+1)}{\delta}.
\]
Then there exist distinct edges $u_{i_1},\dots,u_{i_{\frac{3}{\delta}}}$, each two of them separated by at least $k+1$ edges along $\partial_1$, such that each $\mathcal{B}_{n,i_j}$ occurs.  
Consequently, for $1 \le j < j' \le \frac{3}{\delta}$, the corresponding triangulations are disjoint: 
$F(T^1(u_{i_j})) \cap F(T^1(u_{i_{j'}})) = \emptyset$.  
Hence,
\[
2n - \sum_{i=1}^{\ell_n}(p^n_i - 2)
= \#F(T_{n,g_n,\mathbf{p}^n})
\ge \tfrac{3}{\delta}\, \delta n = 3n.
\]
That is a contradiction for $n$ large enough (see Figure~\ref{small_holes_not_too_big}).
\begin{figure}[H]
\centering
\includegraphics[scale=0.3]{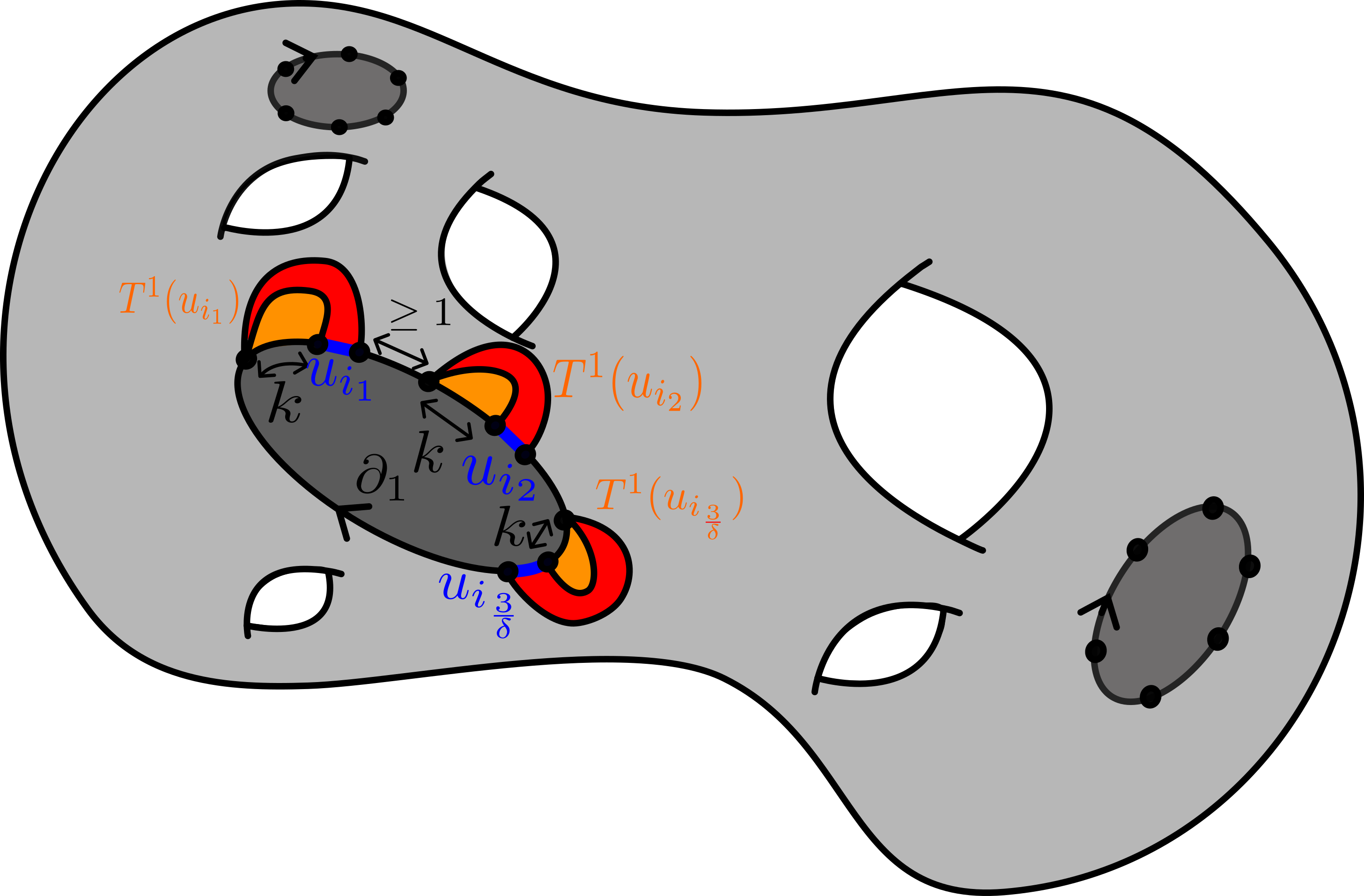}
\caption{We represent in blue the edges $u_{i_1},\dots,u_{i_{\frac{3}{\delta}}}$. The triangles behind the edges are represented in red. The components $T^1(u_{i_1}),\dots,T^1(u_{i_{\frac{3}{\delta}}})$ are represented in orange. Each orange component has at least $\delta n $ triangles, thus $T_{n,g_n,\mathbf{p}^n}$ has at least $3n$ internal faces.}
\label{small_holes_not_too_big}
\end{figure}

 Using the translation invariance of $(T_{n,g_n,\mathbf{p}^n},e^n_1)$ along $\partial_1$, we have
\[
\mathbb{E}\Big[\#\{\, i \in \{0,\dots,p^n_1-1\}: \mathcal{B}_{n,i} \text{ occurs} \,\}\Big]
= p^n_1\, \Pf\big(\mathcal{L}_{n,k}^{\mathrm{sep}}(e^n_1),\, \#F(T^1(e^n_1)) \ge \delta n\big).
\]
Combining this with~\eqref{bound_expectation} yields
\[
\Pf\big(\mathcal{L}_{n,k}^{\mathrm{sep}}(e^n_1),\, \#F(T^1(e^n_1)) \ge \delta n\big)
\le \frac{6(k+1)}{\delta p^n_1}
\le \frac{\varepsilon}{2}.
\]
This concludes the proof.

\end{proof}

\medskip
We now prove that for fixed $k \ge 0$, if $\mathcal{L}_{n,k}(e^n)$ occurs, then with high probability $T^1(e^n)$ is a small planar triangulation of the $(k+1)$-gon.

\begin{proposition}\label{Finite_holes_filled_by_finite_maps}
For any $k,\varepsilon > 0$, there exists $a > 0$ such that for $n$ large enough,
\[
\Pf\big(\mathcal{L}_{n,k}(e^n),\, T^1(e^n) \notin \bigsqcup_{j=0}^{a}\mathcal{T}_{k+1}(j,0)\big) \le \varepsilon.
\]
\end{proposition}

\begin{proof}
Fix $\varepsilon > 0$.  
Let $\delta > 0$ be small enough so that, for $n$ sufficiently large,
\[
\frac{g_n}{n} \le (1 - 4\delta)\Big(\frac{\theta}{2} + \frac{1}{4}\Big).
\]
By Proposition~\ref{small_hole_has_small_proportion_faces}, it suffices to prove that there exists $a > 0$ such that for $n$ large enough,
\[
\Pf\!\left(
\mathcal{L}_{n,k}(e^n),
T^1(e^n) \notin \bigsqcup_{j=0}^{a}\mathcal{T}_{k+1}(j,0),
\big(\mathcal{L}^{\mathrm{sep}}_{n,k} \cap \{\#F(T^{1}(e^n)) \ge \delta n\}\big)^{c}
\right) \le \varepsilon.
\]
As in the proof of Proposition~\ref{small_hole_has_small_proportion_faces},  
we may assume $p^n_1 \ge k$ and show that there exists $a > 0$ such that, for $n$ large enough,
\[
\Pf\!\left(
\mathcal{L}_{n,k}(e^n_1),
T^1(e^n_1) \notin \bigsqcup_{j=0}^{a}\mathcal{T}_{k+1}(j,0),
\big(\mathcal{L}^{\mathrm{sep}}_{n,k} \cap \{\#F(T^{1}(e^n_1)) \ge \delta n\}\big)^{c}
\right) \le \varepsilon.
\]

We can express this probability as
\begin{align}\label{to_bound}
\tau_{\mathbf{p}^{n}}(n,g_n)^{-1}\!\Big(&
\tau_{k+1,p^n_1-k,p^n_2,\dots,p^n_{\ell_n}}(n,g_n-1) \\
&\nonumber \quad+ \sum_{(n_1,g_1,I),(n_2,g_2,J)}
\tau_{(k+1,\mathbf{p}^{n}_{I})}(n_1,g_1)
\tau_{(p^{n}_1-k,\mathbf{p}^{n}_{J})}(n_2,g_2)
\Big),
\end{align}
where the sum is taken over indices satisfying
\[
\left\{
\begin{array}{l}
n_1 + n_2 = n-1,\\
g_1 + g_2 = g_n,\\
I \sqcup J = \{2,\dots,\ell_n\},\\
2n_1 - \sum_{i\in I}(p^n_i - 2) \le \delta n,\\
n_1 \ge a \text{ or } g_1 \ge 1 \text{ or } I \neq \emptyset.
\end{array}
\right.
\]
The condition $2n_1 - \sum_{i\in I}(p^n_i - 2) \le \delta n$ follows from $\#F(T^1(e^n_1)) \le \delta n$, and the last condition expresses that $T^1(e^n_1) \notin \bigsqcup_{j=0}^{a}\mathcal{T}_{k+1}(j,0)$.  
We have $n_1 \le \delta n + |\mathbf{p}^n| \le 2\delta n$ for $n$ large enough, hence $n_2 \ge n - 4\delta n$.  
In particular,
\[
\frac{g_2}{n_2} \le \frac{g_n}{(1 - 4\delta)n} \le \frac{\theta}{2} + \frac{1}{4} \in [0,\tfrac{1}{2})
\]
by the choice of $\delta$.  
Combining this with $|\mathbf{p}^{n}_{J}| + p^n_1 - k \le |\mathbf{p}^{n}| = o(n)$ and using Lemmas~\ref{bounded_ratio_vertices} and~\ref{add_new_boundary}, we obtain for $n$ large enough,
\[
\tau_{p^n_{1}-k,\mathbf{p}^{n}_{J}}(n_2,g_2)
\le C_{\theta,k}\, n^{-1}\, \tau_{1,p^n_1,\mathbf{p}^n_J}(n_2,g_2).
\]
Applying both parts of Proposition~\ref{combi_estimate_rooted_middle} with $r_1 = k+1$ and $r_2 = 1$, and using that the indices in~\eqref{to_bound} satisfy $((n_1,g_1,I),(n_2,g_2,J)) \in \mathcal{M}(n,a)$, we deduce that there exists $a > 0$ such that, for $n$ large enough, the quantity in~\eqref{to_bound} is bounded by $C_{\theta,k}\varepsilon$.  
This concludes the proof, since this holds for any $\varepsilon > 0$.
\end{proof}

\subsection{Boundaries do not touch each other or fold on themselves}\label{subsection_pathological}

In this section, we rule out the last two pathological cases, namely~\eqref{list_pathological_cases3} and~\eqref{list_pathological_cases4}.  
Pathological case~\eqref{list_pathological_cases3} occurs when the boundary containing $e^n$ folds back onto itself at $e^n$, whereas pathological case~\eqref{list_pathological_cases4} arises when the boundary containing $e^n$ comes too close to another boundary at $e^n$ (see Figure~\ref{boundary_touch_itself_illustration}).  
\begin{figure}[H]
\centering
\includegraphics[scale=0.3]{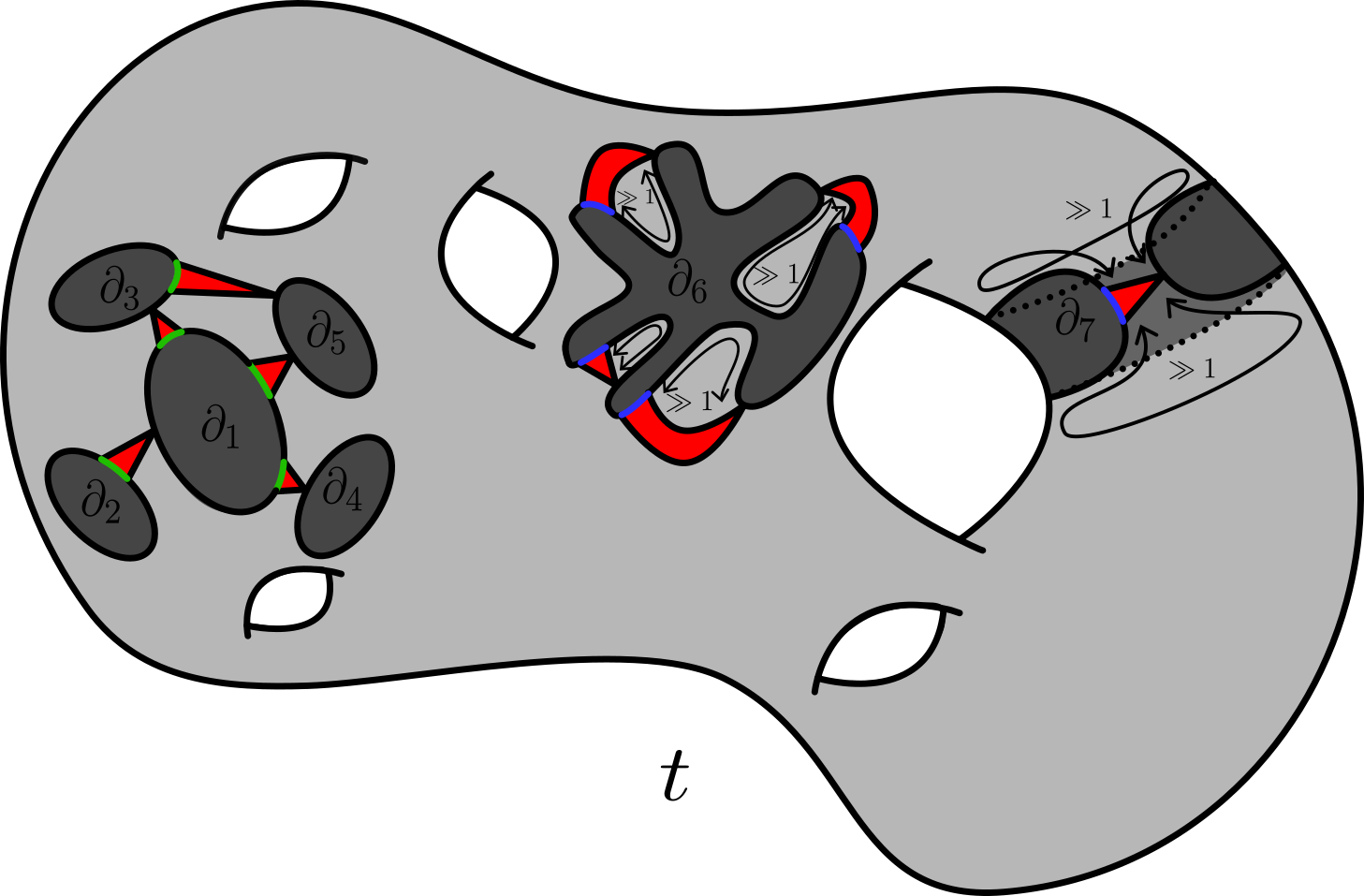}
\caption{We represent a triangulation of the multi-polygon with many edges $e \in \partial t$ such that $(t,e)$ satisfies $(\mathcal{H}_k) \text{ or } (\mathcal{O})$ with $k \gg 1$. The green edges $e$ are such that $(t,e)$ satisfies $(\mathcal{O})$. The blue edges $e$ are such that $(t,e)$ satisfies $(\mathcal{H}_k)$.}
\label{boundary_touch_itself_illustration}
\end{figure}

The main result of this section is the following.

\begin{proposition}\label{no_event_H_or_O}
For any $\varepsilon > 0$, there exists $k \ge 0$ such that for all sufficiently large $n$,
\[
\Pf\big( (T_{n,g_n,\mathbf{p}^{n}},e^n) \text{ satisfies } (\mathcal{H}_k) \text{ or } (\mathcal{O}) \big) \le \varepsilon.
\]
\end{proposition}

\paragraph{Heuristic idea.}
We first give some intuition for why Proposition~\ref{no_event_H_or_O} holds.  
Suppose, for the sake of contradiction, that there exists $\varepsilon > 0$ and arbitrarily large $k,n$ such that
\[
\Pf\big( (T_{n,g_n,\mathbf{p}^{n}},e^n) \text{ satisfies } (\mathcal{H}_k) \text{ or } (\mathcal{O}) \big) \ge \varepsilon.
\]
Then, at least a proportion $\varepsilon/2$ of triangulations $t \in \mathcal{T}_{\mathbf{p}^{n}}(n,g_n)$ would contain at least $e_1,\dots,e_{\frac{\varepsilon}{2}|\mathbf{p}^{n}|}$ distinct boundary edges such that each pair $(t,e_j)$ satisfies $(\mathcal{H}_k)$ or $(\mathcal{O})$. Let $f_j$ denote the triangle incident to $e_j$. We claim that only two possible scenarios (see Proposition~\ref{many_events_C}) can occur:
\begin{itemize}
    \item[$\bullet$] The triangles $f_1,\dots,f_{\frac{\varepsilon}{2}|\mathbf{p}^{n}|}$ cross each other a lot, which significantly reduces the genus of $t \setminus \{f_1 \cup \cdots \cup f_{\frac{\varepsilon}{2}|\mathbf{p}^{n}|}\}$ (see the right of Figure~\ref{pattern}).
    \item[$\bullet$] A positive proportion of the edges $e_j$ have large neighbourhoods that resemble a collection of planar triangulations glued together in a ``linear'' fashion (see left of Figure~\ref{pattern}).
\end{itemize}
This section decomposes as follows:
\begin{itemize}
	\item[$\bullet$] In Section~\ref{first_subsection} we show that the second scenario is very unlikely as $n$ tends to $+\infty$. This is Proposition~\ref{bound_event_C}.
	\item[$\bullet$] In Section~\ref{second_subsection} we  prove that we fall in one of the two scenarios. This is Proposition~\ref{many_events_C}.
	\item[$\bullet$] In Section~\ref{third_subsection} we  exclude the first scenario and thus prove Proposition~\ref{no_event_H_or_O}.
\end{itemize}
The first situation can be controlled using genus-reduction estimates as in the proof of Lemma~\ref{T_is_almost_surely_planar_rooted_middle}.  
The second one is more delicate: a uniform planar triangulation of a polygon is very unlikely to display a ``linear'' structure (see Proposition~\ref{intermediate_lemma}).  
Moreover, since $T_{n,g_n,\mathbf{p}^{n}}$ is chosen uniformly at random, any planar neighbourhood of $(T_{n,g_n,\mathbf{p}^{n}},e^n)$ should itself be uniformly distributed, conditionally on its size. Thus, the left pattern of Figure~\ref{pattern} should be atypical.

\begin{figure}[H]
    \centering
    \includegraphics[scale=0.7]{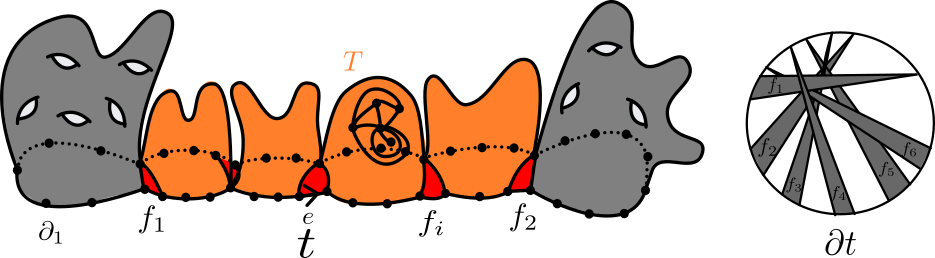}
    \caption{On the left: a triangulation $(t,e)$ satisfying $(\mathcal{C}_{L,p})$.  
    The red triangles $f_i$ correspond to triangles discovered behind edges $e'$ lying on $\partial_1$ such that $(t,e')$ satisfies $(\mathcal{H}_{pL})$.  
    Each of them is $pL$-close to the triangle $t$ behind $e^n$.  
    The planar triangulation $T$ in orange between $f_1$ and $f_2$ is \emph{non-typical} because the red triangles disconnect $T$ into a collection of planar triangulations arranged in a line. On the right: the boundary of a triangulation $t$ where the faces $f_1,\cdots,f_6$ cross each other a lot.}
    \label{pattern}
\end{figure}

\medskip
Let us now formalize this intuition.  
Fix $n,g$, $\mathbf{p} = (p_1,\dots,p_{\ell})$, and $t \in \mathcal{T}_{\mathbf{p}}(n,g)$.  
Let $e_1,e_2$ be two edges on $\partial t$, and let $f_1,f_2$ denote the triangles behind them.  
Denote by $x_1,x_2$ the vertices of $e_1$, by $x_3$ the third vertex of $f_1$, and symmetrically by $y_1,y_2,y_3$ the vertices of $f_2$.  
Assume that both $x_3$ and $y_3$ lie on $\partial t$.

\begin{definition1}\label{def_Lclose}
For $L > 0$, we say that $f_1$ and $f_2$ (see Figure~\ref{Lclose}) are \emph{$L$-close} if:
\begin{itemize}
    \item each vertex $x_1,x_2,x_3$ is at distance at most $L$ from $\{y_1,y_2,y_3\}$ along $\partial t$,
    \item each vertex $y_1,y_2,y_3$ is at distance at most $L$ from $\{x_1,x_2,x_3\}$ along $\partial t$.
\end{itemize}

\begin{figure}[H]
\centering
\includegraphics[scale=0.5]{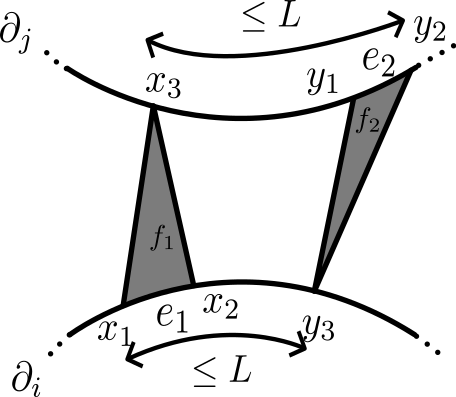}
\caption{Example where the triangles $f_1$ and $f_2$ are $L$-close. Note that it is possible that $i=j$.}
\label{Lclose}
\end{figure}
\end{definition1}

In particular, if $f_1$ and $f_2$ are $L$-close for some $L \ge 0$, then there exist boundaries $\partial_i$ and $\partial_j$ such that all six vertices $x_1,x_2,x_3,y_1,y_2,y_3$ belong to $\partial_i \cup \partial_j$.

\begin{definition1}\label{defClp}
Fix $L,p \ge 0$, a triangulation $t$ of the $(p_1,\dots,p_{\ell})$-gon, and an oriented edge $e$ that lies on $\partial_i$ for $1\le i \le \ell$. We denote by $f$ the triangle that lies behind $e$. We say that $(t,e)$ satisfies $(\mathcal{C}_{L,p})$ (see Figure~\ref{figCLp}) if $(t,e)$ satisfies $(\mathcal{H}_{2pL})$ or $(\mathcal{O})$ and if there exist $p$ distinct edges $e_1,\dots,e_p$ (different from $e$) on $\partial T$ such that, for each $1 \le i \le p$, writing $f_i$ for the triangle behind $e_i$, we have that $f_i$ and $f$ are $pL$-close.
\end{definition1}

\begin{figure}[H]
\centering
\includegraphics[scale=0.4]{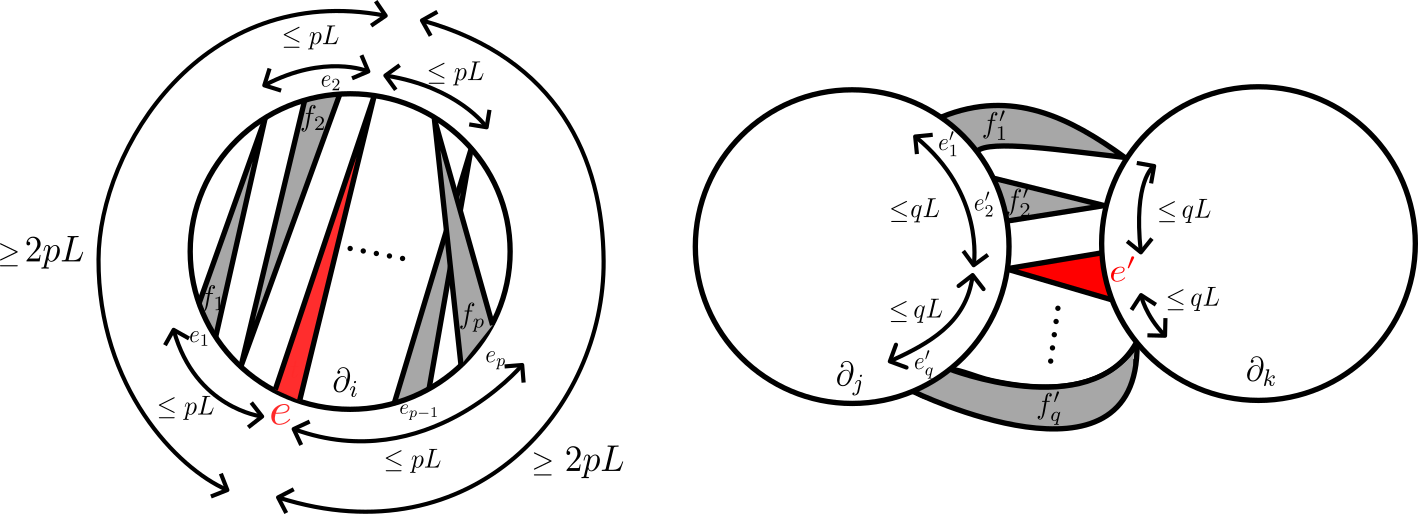}
\caption{Examples of $(t,e)$ and $(t,e')$ satisfying $(\mathcal{C}_{L,p})$.  
On the left: an edge $e \in \partial_i$ such that $(t,e)$ satisfies $(\mathcal{H}_{p})$ and $(\mathcal{C}_{L,p})$.  
On the right: an edge $e' \in \partial_k$ such that $(t,e')$ satisfies $(\mathcal{O})$ and $(\mathcal{C}_{L,q})$.}
\label{figCLp}
\end{figure}

\subsubsection{$(T_{n,g_n,\mathbf{p}^{n}},e^n)$ is unlikely to satisfy $(\mathcal{C}_{L,p})$}\label{first_subsection}

We now show that $(T_{n,g_n,\mathbf{p}^{n}},e^n)$ is unlikely to satisfy $(\mathcal{C}_{L,p})$ for any fixed $L$ when $p$ is large enough.

\begin{proposition}\label{bound_event_C}
For any $\varepsilon > 0$ and any $L \ge 0$, there exists $p \ge 0$ such that, for $n$ large enough
\[
\Pf\big( (T_{n,g_n,\mathbf{p}^{n}},e^n) \text{ satisfies } (\mathcal{C}_{L,p}) \big) \le \varepsilon.
\]
\end{proposition}

We first prove that the ``linear pattern’’ illustrated in Figure~\ref{pattern} is unlikely to occur in a uniform planar triangulation.  
For any $m,p \ge 0$ and $\alpha > 0$, define
\begin{align}\label{defN}
N_{m,p,\alpha} := \#\big\{ e \in \partial T_{m,0,p} : (T_{m,0,p},e) \text{ satisfies } (\mathcal{H}_{\alpha p}) \big\}.
\end{align}
\begin{lemma}\label{intermediate_lemma}
Fix $\alpha_1,\alpha_2 > 0$. There exists a constant $C_{\alpha_1,\alpha_2} >0$ such that for any $m,p \ge 0$ we have
    \begin{align*}
        \Pf(N_{m,p,\alpha_1}\ge \alpha_2 p) \le C_{\alpha_1,\alpha_2}p^{-\frac{6}{5}}.
    \end{align*}
\end{lemma}

\begin{proof}
 Let $(e^i)_{1 \le i \le 5}$ be five distinct edges chosen uniformly at random on $\partial T_{m,0,p}$. We define $c = \min(\frac{\alpha_2}{100},\alpha_1)$. Let $A$ be the event that the triangles $(f_i)_{1 \le i \le 5}$ incident to the edges $(e^i)_{1 \le i \le 5}$ have all their vertices on $\partial T_{m,0,p}$, and such that the six planar triangulations composing $T_{m,0,p} \backslash (f_1 \cup \cdots  \cup f_5)$ have perimeters at least $cp$. The probability of this event satisfies
    \begin{align}\label{proba_event}
        \mathbb{P}(A) \le 120\tau_{p}(m,0)^{-1} \sum_{\substack{m_1+\cdots +m_6 = m-5\\p_1+\cdots +p_6 = p+5\\p_1,\cdots ,p_6 \ge  c p}} \prod_{i=1}^{6}\tau_{p_i}(m_i,0),
    \end{align}
    where the factor $ 120=5!$ counts the order in which the edges $e^1\cdots ,e^5$ appear on the boundary. Using Proposition~\ref{estimate_big_hole} with $\alpha = c$ and $\varepsilon > 0$ such that $-\frac{5}{4}+5\varepsilon \le -\frac{6}{5}$, it follows that for any $m' \ge 1$ and $1 \le p' \le m' + 2$ we have 
    \begin{align*}
    \sum_{\substack{m_1+m_2 = m'-1\\ p_1+p_2 = p'+1 \\p_1,p_2 \ge c p'}} \tau_{p_1}(m_1,0) \tau_{p_2}(m_2,0) \le C_{\alpha_1,\alpha_2}(p')^{-\frac{1}{4}+\varepsilon}\tau_{p'}(m',0).
    \end{align*}
Fix $m_1,\cdots,m_4$ and $p_1,\cdots,p_4$ in \eqref{proba_event} and let us write $m' =  m-m_1-\cdots - m_4$ and $p' = p+4 - p_1-\cdots -p_4$. Note that we have $p' \ge cp$. Then, summing over $m_5$ and $p_5$ in \eqref{proba_event} leads to the sum
\begin{align*}
    \sum_{\substack{m_5+m_6 = m'\\ p_5+p_6 = p' + 1 \\p_5,p_6 \ge c p}} \tau_{p_5}(m_5,0) \tau_{p_6}(m_6,0) &\le C_{\alpha_1,\alpha_2}c^{-\frac{1}{4}+\varepsilon}p^{-\frac{1}{4}+\varepsilon}\tau_{p'}(m',0).
\end{align*}    
We deduce 
\begin{align*}
\sum_{\substack{m_1+\cdots +m_6 = m-5\\p_1+\cdots +p_6 = p+5\\p_1,\cdots ,p_6 \ge  c p}} \prod_{i=1}^{6}\tau_{p_i}(m_i,0) \le  C_{\alpha_1,\alpha_2}c^{-\frac{1}{4}+\varepsilon}p^{-\frac{1}{4}+\varepsilon}\sum_{\substack{m_1+\cdots +m_5 = m-4\\p_1+\cdots +p_5 = p+4\\p_1,\cdots ,p_5 \ge  c p}} \prod_{i=1}^{5}\tau_{p_i}(m_i,0). 
\end{align*}
Iterating this argument leads to the bound
    \begin{align}\label{bound_A}
        \mathbb{P}(A) &\le   C_{\alpha_1,\alpha_2}' p^{-\frac{5}{4}+5\varepsilon} \le C_{\alpha_1,\alpha_2}' p^{-\frac{6}{5}},
    \end{align}
    for some $C_{\alpha_1,\alpha_2}' > 0$. Let $I$ denote the set of edges $e$ on the boundary $\partial T_{m,0,p}$ such that $(T_{m,0,p},e)$ satisfies $(\mathcal{H}_{\alpha_1 p})$. Let $v^1,\cdots,v^5$ denote the third vertices of the faces $f_1,\cdots,f_5$. Let $C$ denote the event that for each $i \in \{1,\cdots,5\}$ the edge $e^i$ lies in $I$ and is at distance at least $\frac{\alpha_2}{100}p$ from $\{e^j, v^j\}_{j \neq i-1}$ along $\partial T_{m,0,p}$. We claim that 
    \begin{align}\label{lower_bound}
        \mathbb{P}(C \mid N_{m,p,\alpha_1}\ge \alpha_2 p) \ge \frac{(\alpha_2)^5}{2^5}.
    \end{align}
    Indeed, this follows from the fact that $ \{N_{m,p,\alpha_1}\ge \alpha_2 p\} = \{\#I \ge \alpha_2 p\}$ and that $e^1,\cdots ,e^5$ are chosen uniformly at random on $\partial T_{m,0,p}$. Thus, for each $ 1 \le i \le 5$, conditionally on $e^1,\dots,e^{i-1}$, the probability that $e^{i}$ lies in $I$ and is at distance at least $\frac{\alpha_2}{100}p$ from $\{e^j, v^j\}_{j \le i}$ is at least $\alpha_2 - \frac{4i}{100}\alpha_2 \ge \frac{\alpha_2}{2}$. 
    
    Finally, under the event $\{N_{m,p,\alpha_1}\ge \alpha_2 p\} \cap C$, the event $ A$ occurs. Combining \eqref{bound_A} and \eqref{lower_bound}, we deduce the result:
    \begin{align*}
        \mathbb{P}(N_{m,p,\alpha_1}\ge \alpha_2 p) \le \left(\frac{(\alpha_2)^5}{2^5}\right)^{-1}C_{\alpha_1,\alpha_2}'p^{-\frac{6}{5}}.
    \end{align*}
\end{proof}

\medskip
We now prove Proposition~\ref{bound_event_C}.
\begin{proof}
Fix $\varepsilon, L > 0$. As in the proof of Proposition~\ref{small_hole_has_small_proportion_faces}, using Proposition~\ref{perimeter_typically_large}, it suffices to assume that $p^n_1 \to +\infty$ and to prove that there exists $p \ge 0$ such that, for $n$ large enough
\begin{align}\label{prove1}
\Pf\big( (T_{n,g_n,\mathbf{p}^{n}},e^n_1) \text{ satisfies } (\mathcal{C}_{L,p}) \big) \le \varepsilon.
\end{align}
If $(T_{n,g_n,\mathbf{p}^{n}},e^n_1)$ satisfies $(\mathcal{C}_{L,p})$, then it satisfies either $(\mathcal{H}_{2pL})$ or $(\mathcal{O})$. We first handle the case of $(\mathcal{H}_{2pL})$ and prove that there exists $p \ge 0$ such that, for $n$ large enough,
\begin{align}\label{prove1H}
\Pf\big( (T_{n,g_n,\mathbf{p}^{n}},e^n_1) \text{ satisfies } (\mathcal{C}_{L,p}) \text{ and } (\mathcal{H}_{2pL}) \big) \le \varepsilon.
\end{align}

Assume that $(T_{n,g_n,\mathbf{p}^{n}},e^n_1)$ satisfies both $(\mathcal{C}_{L,p})$ and $(\mathcal{H}_{2pL})$.  
Let $f$ be the triangle behind $e^n_1$ and let $z$ be its third vertex.  
Let $I_1$ (resp. $I_2$) be the segment of $pL$ edges immediately to the left (resp. right) of $e^n_1$ on $\partial_1$, and let $I_3$ (resp. $I_4$) be the segment of $pL$ edges to the left (resp. right) of $z$ on $\partial_1$. We have $p^n_1 \ge 4pL $, so the segments $I_1, I_2, I_3, I_4$ are pairwise disjoint (see Figure~\ref{proof_bound_Clp}).

For $1 \le j \le 4$, let $A_j$ be the event that there exist at least $\frac{p}{4}$ edges $e'$ in $I_j$ such that the corresponding triangles $f'$ behind $e'$ are $pL$-close to $t$. Since $(T_{n,g_n,\mathbf{p}^n},e)$ satisfies $(\mathcal{C}_{L,p})$, at least one of the events $A_j$ must occur.

\begin{figure}[H]
\centering
\includegraphics[scale=0.16]{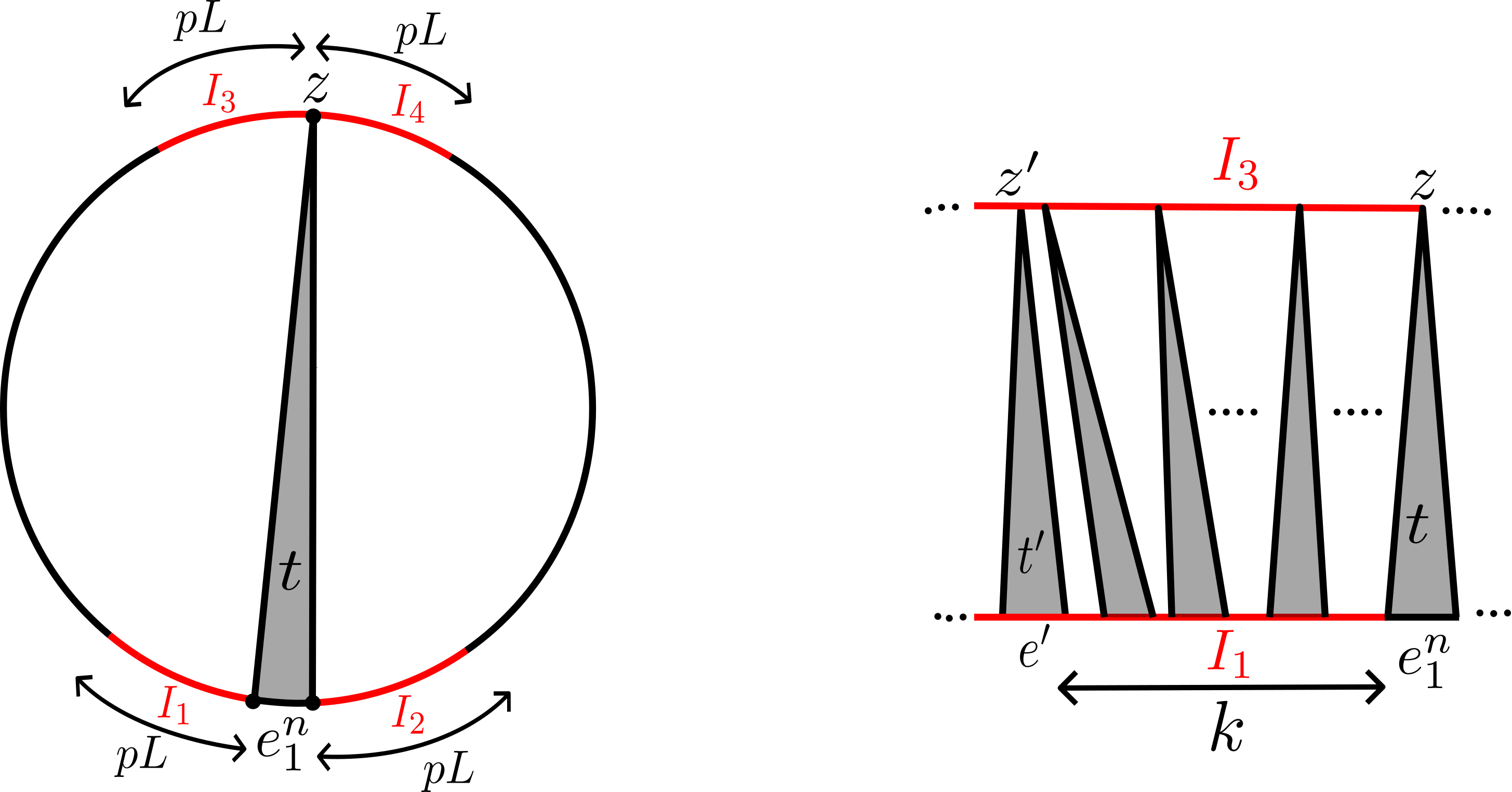}
\caption{Left: the triangle $f$ and the four disjoint segments $I_1, I_2, I_3, I_4$ on $\partial_1$.  
Right: illustration of the case where $A_1$ and $B_{1,k}$ occur.}
\label{proof_bound_Clp}
\end{figure}

Thus,
\begin{align}\label{from_C_to_A}
\Pf\big( (T_{n,g_n,\mathbf{p}^{n}},e^n_1) \text{ satisfies } (\mathcal{C}_{L,p}) \text{ and } (\mathcal{H}_{p}) \big)
    \le \Pf\Big(\bigcup_{i=1}^4 A_i\Big)
    \le \sum_{i=1}^4 \Pf(A_i).
\end{align}
By symmetry, it suffices to bound $\Pf(A_1)$. The same argument applies to the other $A_i$.

Under $A_1$, let $e'$ be the leftmost edge in $I_1$ such that $f'$ (the triangle behind $e'$) is $pL$-close to $f$, and let $z'$ be the third vertex of $f'$ (see Figure~\ref{proof_bound_Clp}).  
Then $z'$ must lie in $I_3 \cup I_4$: otherwise, the vertex $z$ would be at distance at least $pL+1$ from any vertex of $f'$ on $\partial T_{n,g_n,\mathbf{p}^n}$.

Let $B_1$ (resp. $C_1$) denote the event that $A_1$ occurs and that $z'$ is (resp. is not) incident to an edge in $I_3$.  
Thus,
\[
A_1 = B_1 \sqcup C_1.
\]

We first check that the event $C_1$ is very unlikely.  
Under $C_1$, the vertex $z'$ is incident to an edge in $I_4$, so the triangles $f$ and $f'$ cross each other, implying
\[
T_{n,g_n} \setminus \{f \cup f'\} \in \mathcal{T}_{(p_1+2,\dots,p_{\ell_n})}(n,g_n-1).
\]
Since there are at most $pL$ choices for $e'$ and at most $pL+1$ choices for $z$ (since $I_4$ contains $pL$ edges, i.e. $pL+1$ vertices), we get
\begin{align}\label{not_R1_bound}
\Pf(C_1) \le (pL+1)pL\, \frac{\tau_{(p^n_1+2,\dots,p^n_{\ell_n})}(n,g_n-1)}{\tau_{\mathbf{p}^{n}}(n,g_n)} \xrightarrow[n \to \infty]{} 0,
\end{align}
where the limit follows from Lemma~\ref{bounded_ratio_vertices} and Proposition~\ref{estimate_genus_reducing}, since $p,L$ are fixed.

Now, for $\frac{p}{4}-1 \le k \le pL$, let us denote by $B_{1,k}$ the event that $B_1$ occurs and that the segment between $e^n_1$ and $e'$ contains exactly $k$ edges.  
Then $B_1 = \bigsqcup_{k=\frac{p}{4}-1}^{pL} B_{1,k}$.  
Fix such a $k \in \{\frac{p}{4}-1,\dots,pL\}$. Let us show that
\begin{align}\label{bound_R1k}
\Pf(B_{1,k}) \le C_L\, p^{-\frac{6}{5}}.
\end{align}

Under $B_{1,k}$, the vertex $z'$ lies on $I_3$, so the union $ f \cup f'$ splits $\partial_1 t$ into three holes, denoted by $h_1,h_2,h_3$ from left to right, with random perimeters $r_1,r_2,r_3$.  
Let $T$ be the component of $T_{n,g_n,\mathbf{p}^n} \setminus \{f \cup f'\}$ that contains $h_2$. Note that $T$ might also contain $h_1$ or $h_3$. We decompose
\begin{align}\label{rewrite_PR1k}
\Pf(B_{1,k})
&= \mathbb{P}\big( B_{1,k} \text{ and } T \notin \bigcup_{m \ge 0} \mathcal{T}_{r_2}(m,0) \big)
 + \mathbb{P}\big( B_{1,k} \text{ and } T \in \bigcup_{m \ge 0} \mathcal{T}_{r_2}(m,0) \big).
\end{align}
We first show that
\begin{align}\label{planar_mid}
\mathbb{P}\big( B_{1,k} \text{ and } T \notin \bigcup_{m \ge 0} \mathcal{T}_{r_2}(m,0) \big) \xrightarrow[n \to \infty]{} 0.
\end{align}
Let us recall the notations $\mathcal{L}_{n,k}(e)$ and $\mathcal{R}_{n,k}(e)$ defined in the introduction of Section~\ref{section_boundary}. Let $u_i$ (resp. $v_i$) be the $i^{\mathrm{th}}$ edge to the left (resp. right) of $e^n_1$ on $\partial_1$.  
Under $B_{1,k}$, the events $\mathcal{L}_{n,r_1-1}(u_{k+1})$ and $\mathcal{R}_{n,r_3-1}(e^n_1)$ occur (see Figure~\ref{proofboundCLp2}).  
Hence,
\begin{align}\label{step1}
\mathbb{P}\big( B_{1,k}, T \notin \bigcup_{m \ge 0} \mathcal{T}_{r_2}(m,0) \big)
 \le \sum_{r_1,r_2,r_3} \Pf\big( \mathcal{L}_{n,r_1-1}(u_{k+1}), \mathcal{R}_{n,r_3-1}(e^n_1), T \notin \bigcup_{m \ge 0} \mathcal{T}_{r_2}(m,0) \big).
\end{align}

For fixed $r_1,r_2,r_3$, we claim that:
\begin{itemize}
    \item $\Pf(\mathcal{L}_{n,r_1-1}(u_{k+1}) \cap \mathcal{R}_{n,r_3-1}(e^n_1)) = \Pf(\mathcal{L}_{n,r_2-1}(e^n_1) \cap \mathcal{R}_{n,r_3-1}(v_1))$.
    \item the law of $T_{n,g_n,\mathbf{p}^n} \setminus \{f \cup f'\}$ under $\mathcal{L}_{n,r_1-1}(u_{k+1})$ and $\mathcal{R}_{n,r_3-1}(e^n_1)$
    is identical to that of $T_{n,g_n,\mathbf{p}^n} \setminus \{f \cup f^*\}$ under $\mathcal{L}_{n,r_2-1}(e^n_1)$ and $\mathcal{R}_{n,r_3-1}(v_1)$,
    where $f^*$ is the triangle behind $v_1$.
\end{itemize}
Indeed, in both configurations, the two triangles do not cross each other and create holes of perimeters $(r_1,r_2,r_3)$ or $(r_2,r_1,r_3)$ respectively (see Figure~\ref{proofboundCLp2}). Thus, the two items above hold since for each identities considered, both sides of the equation writes as the same ratio of combinatorial quantities.

Under $\mathcal{L}_{n,r_2-1}(e^n_1)$ and $\mathcal{R}_{n,r_3-1}(v_1)$, let $T'$ be the component of $T_{n,g_n,\mathbf{p}^n}\setminus \{f,f^*\}$ filling the hole of perimeter $r_2$.  
Then, the right-hand side of \eqref{step1} can be rewritten
\begin{align}\label{bound_non_planar}
&\sum_{r_1,r_2,r_3}\Pf(\mathcal{L}_{n,r_2-1}(e^n_1), \mathcal{R}_{n,r_3-1}(v_1), T' \notin \bigcup_{m\ge0}\mathcal{T}_{r_2}(m,0))\\
 &\le \sum_{r_2}\Pf(\mathcal{L}_{n,r_2-1}(e^n_1), T' \notin \bigcup_{m\ge0}\mathcal{T}_{r_2}(m,0)) \xrightarrow[n\to\infty]{} 0,
\end{align}
by Proposition~\ref{Finite_holes_filled_by_finite_maps}, since $r_2 \in \{\frac{p}{4},\dots,2pL+2\}$. This proves \eqref{planar_mid}.

\begin{figure}[H]
\centering
\includegraphics[scale=0.5]{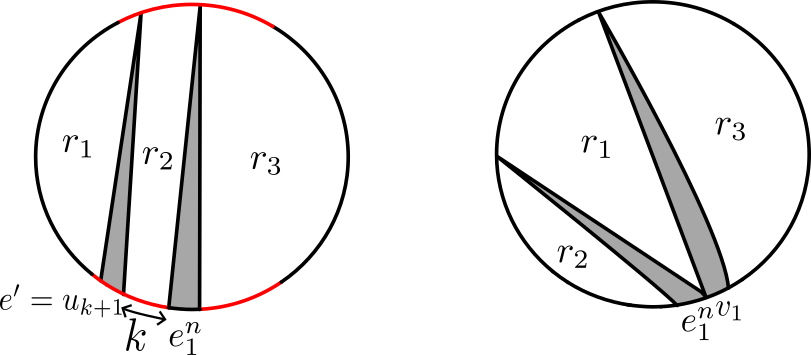}
\caption{Left: the events $\mathcal{L}_{n,r_1-1}(u_{k+1})$ and $\mathcal{R}_{n,r_3-1}(e^n_1)$ occur.  
Right: the events $\mathcal{L}_{n,r_2-1}(e^n_1)$ and $\mathcal{R}_{n,r_3-1}(v_1)$ occur.}
\label{proofboundCLp2}
\end{figure}

Next, we bound the second term in~\eqref{rewrite_PR1k}.  
Under $B_{1,k}$, there exist at least $\frac{p}{30}$ edges $e$ on $\partial T$ such that $(T,e)$ satisfies $(\mathcal{H}_{p/10})$.  
Indeed, from the definition of $A_1$, there are at least $\frac{p}{4}-1$ edges $e$ on $I_1$ such that $t$ and $t(e)$ (the triangle behind $e$) are $pL$-close.  
Discarding the $\frac{p}{10}$ leftmost and rightmost such edges leaves at least $\frac{p}{30}$ edges. For each $e$, by planarity of $T$, the third vertex of the triangle $t(e)$ lies on $I_3$. Hence $(T,e)$ satisfies $(\mathcal{H}_{p/10})$.

Since $T$ has perimeter $r_2$ with $\frac{p}{8} \le r_2 \le 4pL$ and the fact that $T$ is uniform conditionally on its perimeter $m_2$ and its volume $m$, we can apply Lemma~\ref{intermediate_lemma} with $\alpha_1 = 1/40L,\alpha_2 = 1/120L$ and we obtain the bound
\[
 \mathbb{P}\big( B_{1,k} \text{ and } T \in \bigcup_{m \ge 0} \mathcal{T}_{r_2}(m,0) \big) \le C_L\, p^{-\frac{6}{5}}.
\]
Combining this with \eqref{planar_mid} we find that for $n$ large enough we have 
\begin{align*}
 \mathbb{P}\big( B_{1,k} ) \le C_L p^{-\frac{6}{5}}.
\end{align*}
Summing over $k$ and taking $p$ large enough such that $C_L L p^{-\frac{1}{6}} \le \varepsilon/2$, we get for $n$ large enough,
\[
\Pf(B_1) \le \sum_{k=\frac{p}{4}}^{pL}\Pf(B_{1,k}) \le \varepsilon,
\qquad \Pf(A_1) \le \Pf(B_1) + \Pf(C_1) \le 2\varepsilon.
\]
The same bound holds for each $\Pf(A_i)$, completing the proof of~\eqref{from_C_to_A}.

\medskip
Finally, to bound
\begin{align}\label{bound_O}
\Pf\big( (T_{n,g_n,\mathbf{p}^{n}},e^n_1) \text{ satisfies } (\mathcal{C}_{L,p}) \text{ and } (\mathcal{O}) \big),
\end{align}
we reduce to the case already handled.  
If $(T_{n,g_n,\mathbf{p}^{n}},e^n_1)$ satisfies $(\mathcal{C}_{L,p})$ and $(\mathcal{O})$, let $z \in \partial_j$ with $2 \le j \le \ell_n$ be the third vertex of the triangle behind $e^n_1$.  
Then
\[
T_{n,g_n,\mathbf{p}^n} \setminus f \in \mathcal{T}_{(p^n_1+p^n_j+1,p^n_2,\dots,p^n_{j-1},p^n_{j+1},\dots,p^n_{\ell_n})}(n+1,g_n).
\]
Moreover, conditionally on $z \in \partial_j$, the triangulation $T_{n,g_n,\mathbf{p}^n} \setminus f$ is uniformly distributed in $\mathcal{T}_{(p^n_1+p^n_j+1,p^n_2,\dots,p^n_{j-1},p^n_{j+1},\dots,p^n_{\ell_n})}(n+1,g_n)$. Thus it suffices to bound 
$$
\Pf\big( (T_{n,g_n,\mathbf{p}^{n}},e^n_1) \text{ satisfies } (\mathcal{C}_{L,p}) \text{ and } (\mathcal{O}) | z \in \partial_j\big),
$$
uniformly on $j \in \{2,\cdots,\ell_n\}$ for $n$ large enough. For $j \in \{2,\cdots,\ell_n\}$, let $X_j$ be the set of edges in $\partial_1 \cup \partial_j $ whose incident vertices are within distance $pL$ from either $e^n_1$ or $z$ along $\partial_1 \cup \partial_j$.  
Then $\#X_j \le 4pL$. Moreover, discarding the edges having an incident vertex at distance at most $p/8$ from $e$ or $e^n_1$ in $\partial_1 \cup \partial_j$ (see the right part of Figure~\ref{figCLp}), we find
\[
\#\{ e \in X_j : (T_{n,g_n,\mathbf{p}^n}\setminus f, e) \text{ satisfies } (\mathcal{C}_{L,p/16L}) \text{ and } (\mathcal{H}_{p/8}) \} \ge p/32.
\]
By \eqref{prove1H}, if $p$ is chosen large enough, for $n$ large enough, for all $j$ we have
\[
\Pf\big( (T_{n,g_n,(p^n_1+p^n_j+1,p^n_2,\dots,p^n_{\ell_n})},e^n_1) \text{ satisfies } (\mathcal{C}_{L,p/16L}) \text{ and } (\mathcal{H}_{p/8L}) \big)
 \le \frac{\varepsilon}{128L}.
\]
Then using the translation invariance of $(T_{n,g_n,(p^n_1+p^n_j+1,p^n_2,\dots,p^n_{\ell_n})},e^n_1)$ along $\partial_1$, the Markov inequality yields
\[
\Pf\bigg(\#\{ e \in X_j : (T_{n,g_n,\mathbf{p}^n}\setminus f, e) \text{ satisfies } (\mathcal{C}_{L,p/16L}) \text{ and } (\mathcal{H}_{p/8}) \} \ge p/32 \bigg) \le (4pL)\frac{32}{p}\frac{\varepsilon}{128 L} \le \varepsilon.
\]
Thus, for $n$ large enough, we have 
$$
\forall j \in \{2,\cdots,\ell_n\},\text{ }\Pf\big( (T_{n,g_n,\mathbf{p}^{n}},e^n_1) \text{ satisfies } (\mathcal{C}_{L,p}) \text{ and } (\mathcal{O}) | z \in \partial_j\big).
$$
This concludes the proof.
\end{proof}

\subsubsection{Deterministic discovery of triangles} \label{second_subsection}
We start by mentioning that this section is entirely deterministic. It is dedicated to showing that if $(T_{n,g_n,\mathbf{p}^{n}},e)$ satisfies $(\mathcal{H}_{k})$ or $(\mathcal{O})$ for many choices of $e \in \partial T_{n,g_n,\mathbf{p}^{n}}$, we fall into one of the two cases below:
\begin{itemize}
	\item[$\bullet$] There is a positive proportion of edges $e \in \partial T_{n,g_n,\mathbf{p}^{n}}$ such that the triangles lying behind these edges cross each other significantly.
	\item[$\bullet$] There is a positive proportion of edges $e \in \partial T_{n,g_n,\mathbf{p}^{n}}$ such that $(T_{n,g_n,\mathbf{p}^{n}},e)$ satisfies $(\mathcal{C}_{L,p})$ for an appropriate choice of $L$ and $p$.
\end{itemize} 
See the proposition below.
\begin{proposition}\label{many_events_C}
Fix $n,g,\ell \ge 0$ and $(p_1,\cdots,p_{\ell}) \in (\mathbb{N}^{*})^{\ell}$. We also fix $t \in \mathcal{T}_{\mathbf{p}}(n,g)$ and $\varepsilon > 0,p\ge 4$. Suppose that there exist distinct edges $e_1,\cdots,e_{1000\varepsilon |\mathbf{p}|}$ on $\partial_1\cup \cdots \cup \partial_{\ell}$ such that for all $1 \le j \le 1000\varepsilon |\mathbf{p}|$:
\begin{center}
$(t,e_j)$ satisfies $(\mathcal{O})$ or $(\mathcal{H}_{2\varepsilon^{-1} p })$.
\end{center}  
Then one of the following two statements holds:
    \begin{enumerate}
        \item\label{eq1} There exists $I \subset \{1,\cdots,1000\varepsilon |\mathbf{p}|\}$ such that $r:= |I| \ge \varepsilon p^{-1}|\mathbf{p}|$ and denoting by $f_{i}$ for the triangle that lies behind the edge $e_{i}$, we have $t \backslash \bigg(\bigsqcup_{i\in I}f_i\bigg) \in \mathcal{T}_{\mathbf{q}}(n',g')$ for some $\mathbf{q} = (q_1,\cdots,q_{\ell'})$, where $(\ell-\ell')+2(g-g')= r $ and $n' \le n + |\mathbf{p}|$.
        \item\label{eq2} There are more than $\varepsilon|\mathbf{p}|$ edges $e \in \partial t$ such that $(t,e)$ satisfies $(\mathcal{C}_{\varepsilon^{-1},p})$.
    \end{enumerate}
\end{proposition} 
\begin{figure}[H]
    \centering    \includegraphics[scale =0.15]{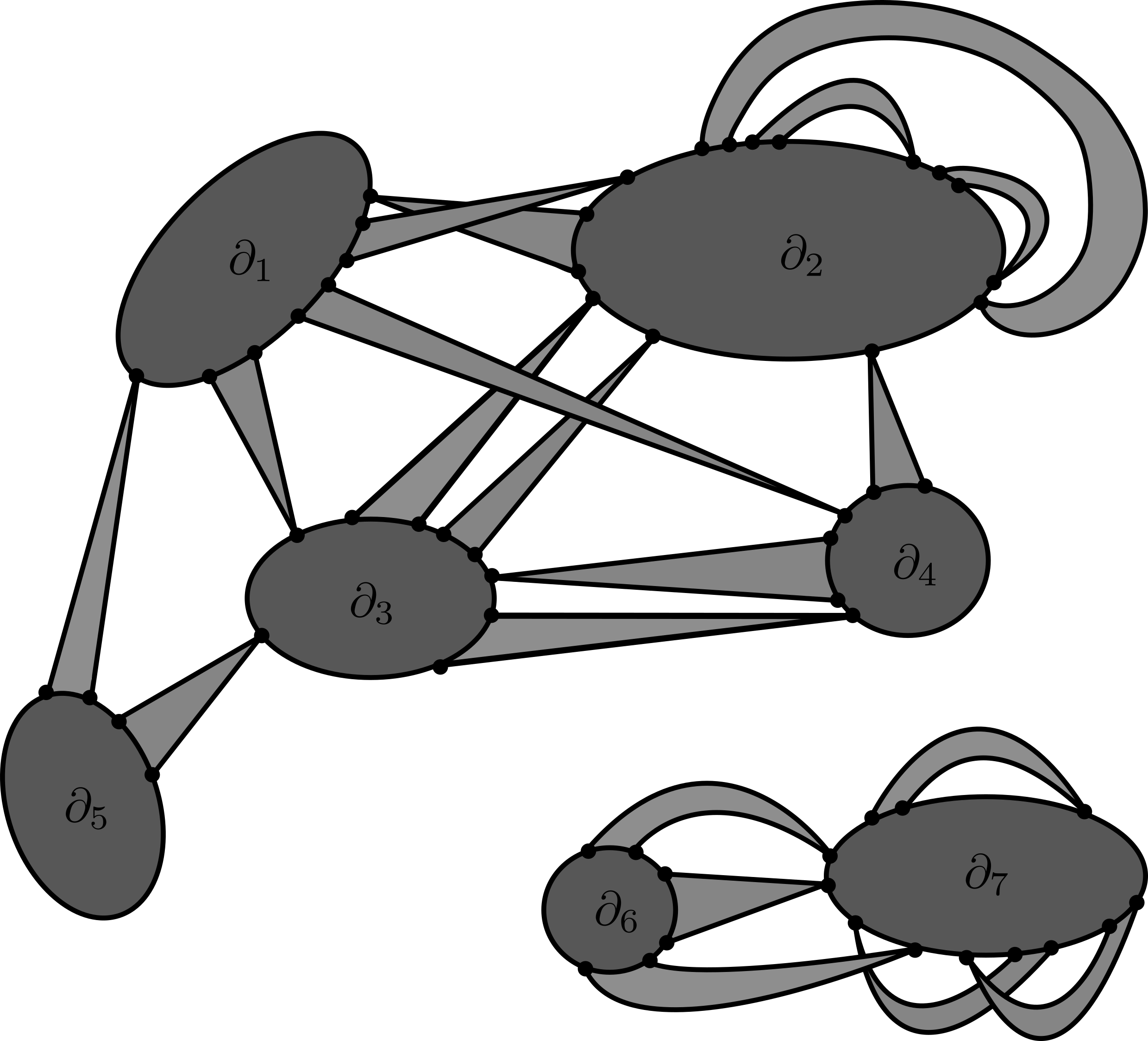}
    \caption{We represent here the boundaries of a triangulation $t$ with $7$ boundaries and the triangles behind the edges $e$ on the boundaries such that $(t,e)$ satisfies $(\mathcal{O})$ or $(\mathcal{H}_{\varepsilon^{-1} p })$.}
    \label{many_triangles}
\end{figure} 
We fix $n,g,\ell > 0$, $\mathbf{p} \in (\mathbb{N}^{*})^{\ell}$, $t \in \mathcal{T}_{\mathbf{p}}(n,g)$, $0 < \varepsilon < \frac{1}{1000}$, $p \ge 4$ and $(e_1,\cdots,e_{1000\varepsilon |\mathbf{p}|})$ as in the statement of the proposition. For $1 \le i \le 1000\varepsilon |\mathbf{p}|$, we denote by $f_i$ the triangle that lies behind $e_i$ in $t$. We assume throughout the whole section that Case $1$ of Proposition~\ref{many_events_C} does not hold. Thus, we aim to prove that Case 2 occurs. Let us explain intuitively why this proposition is true and why it is not obvious. In words, the fact that \eqref{eq1} does not hold means that the triangles $f_1,\cdots,f_{1000\varepsilon |\mathbf{p}|}$ do not “cross each other” much. Thus, the triangles $f_j$ must be “aligned” (see Figure~\ref{pattern}). In other words, there exists a positive proportion of edges $e \in \partial t$ such that $(t,e)$ satisfies $(\mathcal{C}_{\varepsilon^{-1},p})$. However, making this idea precise is not straightforward. In particular, we must be careful when discussing triangles that “cross each other”. Although this notion might seem intuitive (see Figure~\ref{many_triangles}), it is quite technical to make it rigorous. Indeed, we might have four triangles $f_1,f_2,f_3,f_4$ such that $f_1,f_2$ cross each other and so do $f_3,f_4$. However, $f_3$ and $f_4$ may no longer cross each other in $t \backslash (f_1\cup f_2)$ (see Figure~\ref{crossing_to_nocrossing}). To deal with this, we will discover the triangles in a well-chosen order (see Proposition~\ref{discover_nsep_first}). 
\begin{figure}[H]
\centering
\includegraphics[scale=1]{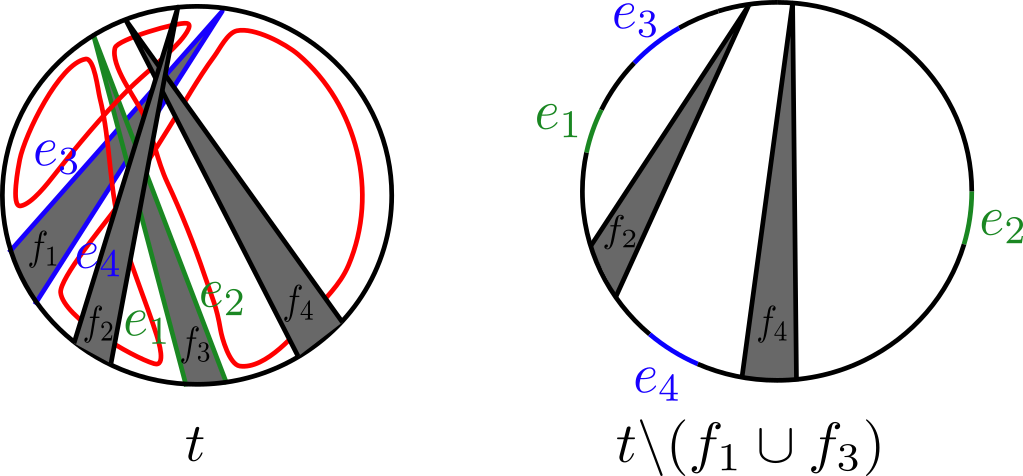}
\caption{On the left, we represent the boundary $\partial_1$ of a triangulation $t$ and four triangles $f_1,f_2,f_3,f_4$ such that $f_1,f_3$ cross each other, and so do $f_2,f_4$. The red path is the hole of $\partial t \cup f_1 \cup f_3$. On the right, we represent $t \backslash (f_1\cup f_3)$. The triangles $f_2,f_4$ no longer cross each other in $t \backslash (f_1\cup f_3)$. More precisely, in both orders $(1,3,2,4)$ and $(2,4,1,3)$, step $1$ is \textbf{SEP} and step $2$ is \textbf{NSEP}. However, for the order $(1,3,2,4)$, both steps $3$ and $4$ are \textbf{SEP}.}
\label{crossing_to_nocrossing}
\end{figure}

To prove this result, we discover the triangles $f_1,\cdots,f_{1000\varepsilon |\mathbf{p}|}$ one by one according to a well-chosen order. 
\begin{definition1}\label{order}
An order is a bijection from $\{1,\cdots,1000\varepsilon |\mathbf{p}|\}$ to itself. We write $i_{1},\cdots,i_{1000\varepsilon |\mathbf{p}|}$ for an order.
A partial order is an injection from $\{1,\cdots,m\}$ to $\{1,\cdots,1000\varepsilon |\mathbf{p}|\}$ with $m \le 1000\varepsilon |\mathbf{p}|$. We write $i_{1},\cdots,i_{m}$ for a partial order.
\end{definition1}

\noindent For any $1 \le j \le 1000\varepsilon |\mathbf{p}|$, we denote by $\mathcal{B}_j = \{\alpha^j_1,\cdots,\alpha^j_{\ell_j}\}$ the set of boundaries of $t \backslash (f_{i_1}\cup\cdots\cup f_{i_j})$. The elements of $\mathcal{B}_j$ are vertex-injective cycles that are vertex-disjoint from each other. The edges of the boundaries $\alpha^j_i$ consist of edges of the initial boundaries of $t$, plus potential edges coming from the already discovered triangles $f_{i_1},\cdots,f_{i_j}$. We refer to the discovery of the triangle $f_{i_{j}}$ as \emph{step $j$}. After discovering $f_{i_1},\cdots,f_{i_j}$, we say that we are at time $j$. The step $j$ is completely determined by the edge $e_{i_{j}}$, which lies on one of the boundaries $\{\alpha^j_1,\cdots,\alpha^j_{\ell_j}\}$, and by the vertex $v_{i_j}$, which lies on one of these boundaries and corresponds to the third vertex of the triangle $f_{i_{j}}$.
\begin{definition1}
For any order $(i_j)$ and any $1 \le j \le 1000 \varepsilon |\mathbf{p}|$, step $j$ is said to be \textbf{SEP} if $e_{i_{j}}\in \alpha^j_{s}$ and $v_{i_{j}} \in \alpha^j_{s}$ for some $s \in \{1,\cdots,\ell_j\}$. Otherwise, it is said to be \textbf{NSEP}. Thus, a \textbf{SEP} step splits the boundary $\alpha_s^j$ containing $e_{i_j}$ into two boundaries. The \textbf{NSEP} occurs when $f_{i_{j}}$ connects two distinct boundaries.\\
\end{definition1}
 
\begin{remark1}
Note that the type of a step $j$ depends on $i_1,\cdots, i_{j-1}$. Therefore, the type of step $j$ is not determined by $i_{j}$ alone. See Figure~\ref{crossing_to_nocrossing}.
\end{remark1}

Let us explain, in words, the order in which we will discover the triangles. We start by choosing \textbf{NSEP} steps as long as possible. When this is no longer possible, we consider pairs of consecutive steps such that the first is \textbf{SEP} and the second \textbf{NSEP}. The pairs (\textbf{SEP},\textbf{NSEP}) have the effect of reducing the genus of $t \backslash (f_{i_1}\cup \cdots \cup f_{i_j})$. When this is no longer possible either, we will show that regardless of the order in which we discover the remaining triangles, all subsequent steps will be \textbf{SEP}.\\
\begin{figure}[H]
    \centering
    \includegraphics[scale =0.15]{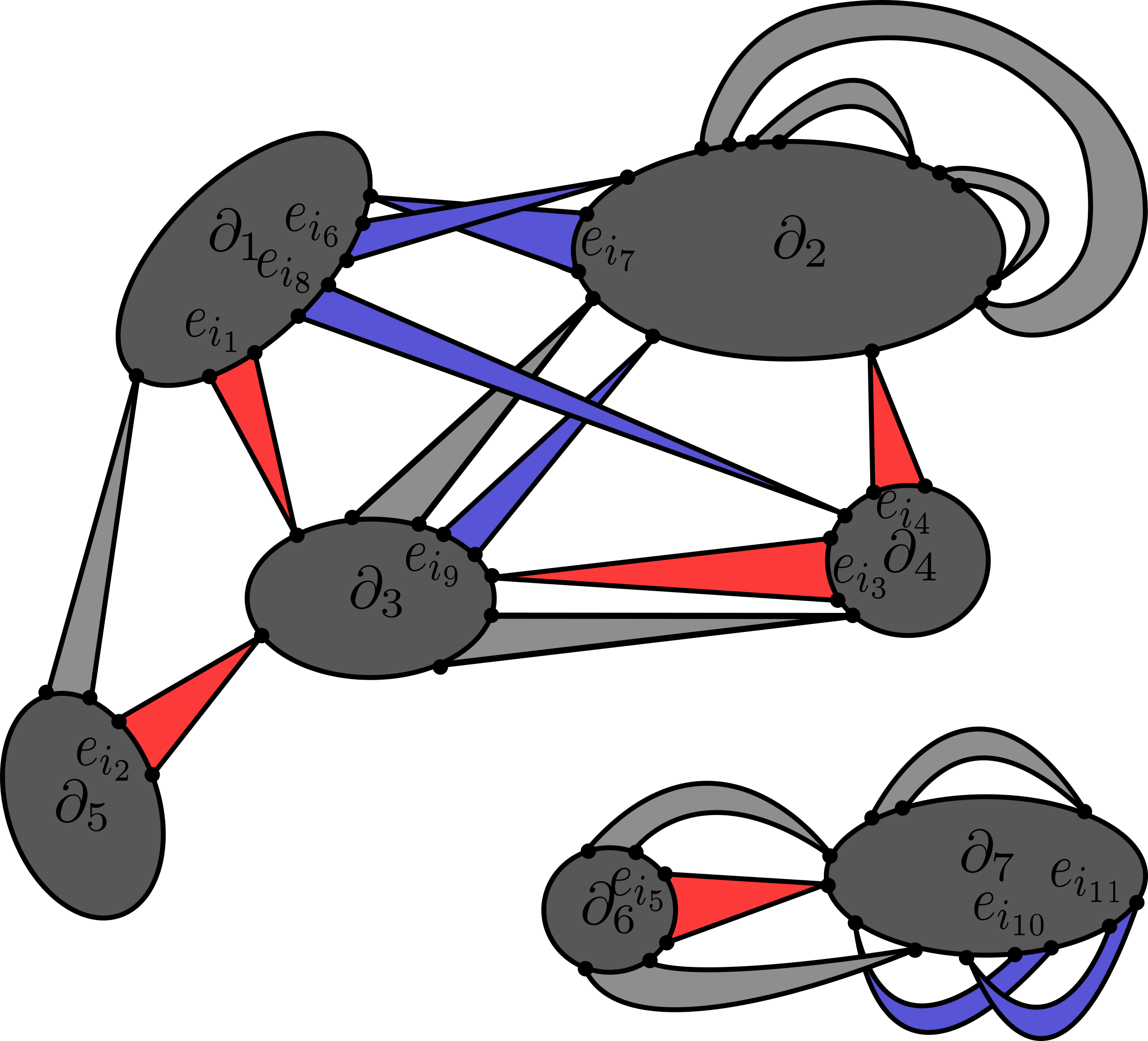}
    \caption{Starting from the example in Figure~\ref{many_triangles}, we first discover $i_1,i_2,i_3,i_4,i_5$ since each of these steps is \textbf{NSEP}. Then we discover $i_6,i_7,i_8,i_9,i_{10},i_{11}$ as these alternate between \textbf{SEP} and \textbf{NSEP} steps. By Proposition~\ref{discover_nsep_first}, for any ordering of the grey triangles, all remaining steps after step $11$ are \textbf{SEP}.}
    \label{discover_first_triangles}
\end{figure}

\begin{proposition}\label{discover_nsep_first}
There exist $m_1,m_2 \ge 0$ and a partial order $i_1,\cdots,i_{m_1+2m_2}$ with $m:=m_1+2m_2 \le \varepsilon p^{-1}|\mathbf{p}|$ such that:
\begin{enumerate}
    \item \label{item1} For any $j \in \{1,\cdots,m_1\}$, step $j$ is $\mathbf{NSEP}$.
    \item\label{item2} For any $j \in \{1,\cdots,m_2\}$, step $m_1+2j-1$ is $\mathbf{SEP}$ and step $m_1+2j$ is $\mathbf{NSEP}$.
    \item \label{item3} For any $i_{m+1},\cdots,i_{1000 \varepsilon |\mathbf{p}|}$ completing the partial order $i_1,\cdots,i_m$ to a full order $i_1,\cdots,i_{1000 \varepsilon |\mathbf{p}|}$ and for any $j \in \{m+1,\cdots,1000 \varepsilon |\mathbf{p}|\}$, step $j$ is $\mathbf{SEP}$.
\end{enumerate}
\end{proposition}

Here we emphasize that our proof is entirely deterministic and depends only on $t$.

\begin{proof}
We first construct $m_1,m_2 \ge 0$ and $i_1,\cdots,i_{m}$, and then verify that items \eqref{item1}, \eqref{item2}, and \eqref{item3} are satisfied. To do so, let us introduce a partial order $(i_1,\cdots,i_{m_1})$ of maximal size such that for $1 \le j \le m_1$, step $j$ is $\textbf{NSEP}$. Thus, for any $j \in \{1,\cdots,m_1\}$, the triangulation $t \backslash (f_1\cup \cdots \cup f_j)$ has exactly one less boundary component than $t \backslash (f_1\cup \cdots \cup f_{j-1})$ and the same genus. We deduce that
\begin{center}
 $t \backslash(f_1\cup \cdots \cup f_{m_1}) \in \mathcal{T}_{(q_1,\cdots,q_{\ell- m_1})}(n+m_1,g)$,
\end{center}
for some $(q_1,\cdots,q_{\ell-m_1})$.

Now, let us construct $i_{m_1+1},\cdots,i_{m_1+2m_2}$. We denote by $Z$ the set of tuples $(i_{m_1+1},\cdots,i_{m_1+2j})$ with $j \ge 0$ satisfying the following conditions:
\begin{enumerate}
	\item \label{cd1} $(i_1,\cdots,i_{m_1+2j})$ is a partial order.
	\item \label{cd2} For any $s \in \{1,\cdots,j\}$, step $m_1+2s-1$ is \textbf{SEP}.
	\item \label{cd3} For any $s \in \{1,\cdots,j\}$, step $m_1+2s$ is \textbf{NSEP}.
\end{enumerate}

We first claim that for any $(i_{m_1+1},\cdots,i_{m_1+2j}) \in Z$:
\begin{itemize}
	\item[a)] We have $t\backslash (f_1\cup \cdots \cup f_{m_1+2j}) \in \mathcal{T}_{(q_1,\cdots,q_{\ell - m_1})}(n+m_1,g-j)$ where $(q_1,\cdots,q_{\ell - m_1}) \in (\mathbb{N}^{*})^{\ell-m_1}$.
	\item[b)] For any $i_{m_1+2j+1} \in \{1,\cdots,1000\varepsilon |\mathbf{p}|\}\backslash \{i_1,\cdots,i_{m_1+2j}\}$, in the partial order $(i_1,\cdots,i_{m_1+2j},i_{m_1+2j+1})$, step $m_1+2j+1$ is \textbf{SEP}.
\end{itemize}
 
We prove this claim by induction on $j \ge 0$.
\begin{itemize}
	\item[$\bullet$] For $j = 0$, we have $t\backslash (f_1 \cup \cdots \cup f_{m_1}) = t \in \mathcal{T}_{(q_1,\cdots,q_{\ell - m_1})}(n+m_1,g)$, so item (a) holds. Item (b) follows from the definition of $m_1$.
	\item[$\bullet$] Suppose the result holds for some $j \ge 0$. Let us fix $(i_{m_1+1},\cdots,i_{m_1+2j+2}) \in Z$. Then $(i_{m_1+1},\cdots,i_{m_1+2j}) \in Z$, so
	\begin{center}
	$t\backslash (f_1 \cup \cdots \cup f_{m_1+2j}) \in \mathcal{T}_{(q_1,\cdots,q_{\ell - m_1})}(n+m_1,g-j)$.
	\end{center} 
Moreover, by item~\eqref{cd2}, step $m_1+2j+1$ is \textbf{SEP}. Thus, the triangle $f_{i_{m_1+2j+1}}$ splits a boundary $\alpha \in \mathcal{B}_{m_1+2j}$ into two boundaries $\beta,\gamma \in \mathcal{B}_{m_1+2j+1}$. We claim that 
\begin{align}\label{disconnect_connect}
 f_{i_{m_1+2j+2}} \text{ connects } \beta \text{ and } \gamma \text{ in } t\backslash (f_1 \cup \cdots \cup f_{m_1+2j} \cup f_{m_1+2j+1}).
\end{align} 
Otherwise, in the partial order $(i_1,\cdots,i_{m_1+2j},i_{m_1+2j+2})$ (removing $i_{m_1+2j+1}$), step $m_1+2j+1$ would be \textbf{NSEP} (see Figure~\ref{SEPNSEP}), contradicting item (b) of the induction hypothesis.
\begin{figure}[H]
\centering
\includegraphics[scale=0.45]{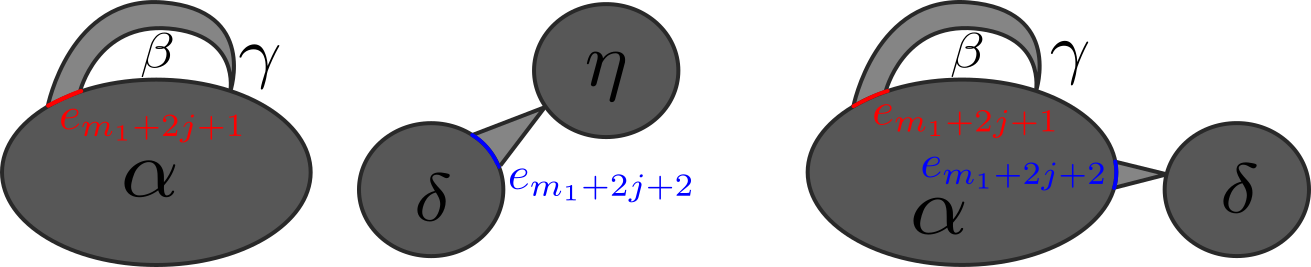}
\caption{On the left, step $m_1+2j+2$ is \textbf{NSEP} and connects two boundaries $\delta,\eta \in \mathcal{B}_{m_1+2j+1} \backslash \{\beta,\gamma\}$. On the right, step $m_1+2j+2$ is \textbf{NSEP} and connects boundary $\gamma$ with a boundary $\delta \in \mathcal{B}_{m_1+2j+1} \backslash \{\beta,\gamma\}$. In both cases, in the partial order $(i_1,\cdots,i_{m_1+2j},i_{m_1+2j+2})$, step $m_1+2j+1$ becomes \textbf{NSEP}.}
\label{SEPNSEP}
\end{figure} 
Hence,
\begin{center}
	$t\backslash (f_1 \cup \cdots \cup f_{m_1+2j} \cup f_{m_1+2j+1} \cup f_{m_1+2j+2}) \in \mathcal{T}_{(q_1,\cdots,q_{\ell - m_1})}(n+m_1,g-j-1)$,
\end{center} 
so item (a) is proved. Item (b) also follows. Indeed, suppose that there exists $i_{m_1+2j+3} \in \{1,\cdots,1000\varepsilon |\mathbf{p}|\}\backslash \{i_1,\cdots,i_{m_1+2j+2}\}$ such that in the partial order $(i_1,\cdots,i_{m_1+2j+2},i_{m_1+2j+3})$, step $m_1+2j+3$ is \textbf{NSEP}. Then, removing $i_{m_1+2j+1}$ and $i_{m_1+2j+2}$ would make step $m_1+2j+1$ discovering $f_{i_{m_1+2j+3}}$, thus it would connect two boundaries and thus be of type \textbf{NSEP}, contradicting item (b) of the induction hypothesis.
\end{itemize} 

We have $\emptyset \in Z$, hence we can define $m_2$ as the maximal size of an element of $Z$. Let us fix $(i_{m_1+1},\cdots,i_{m_1+2m_2}) \in Z$ (potentially empty). By item (a) we have
\begin{center}
$t\backslash (f_1 \cup \cdots \cup f_{m_1+2m_2}) \in \mathcal{T}_{(q_1,\cdots,q_{\ell - m_1})}(n+m_1,g-m_2)$.
\end{center}
Since Item~\ref{eq1} of Proposition~\ref{many_events_C} does not hold, it follows that $(\ell -(\ell-m_1))+2(g-(g-m_2))  \le \varepsilon p^{-1} |\mathbf{p}|$, i.e.
\[
m_1+2m_2 \le \varepsilon p^{-1} |\mathbf{p}|.
\]
Items \eqref{item1} and \eqref{item2} of Proposition~\ref{discover_nsep_first} follow from the construction of $m_1$ and $m_2$. Finally, we prove item \eqref{item3}. We prove this result by induction. For any $j \in \{m+1,\cdots,1000 \varepsilon |\mathbf{p}|\}$, define:
\begin{center}
    $\mathcal{P}_j$ := For any $(i_{m+1},\cdots, i_{j})$ such that $(i_1,\cdots,i_m,\cdots,i_j)$ is a partial order, and for any $s \in \{m+1,\cdots,j\}$, step $s$ is \textbf{SEP}.
\end{center}
We prove this property by induction on $j$.
\begin{itemize}
	\item[$\bullet$] Statement $\mathcal{P}_{m+1}$ holds by item (b). Statement $\mathcal{P}_{m+2}$ holds by definition of $m_2$.
	\item[$\bullet$] Suppose $\mathcal{P}_{j}$ holds for some $j \ge m+2$. Let us prove $\mathcal{P}_{j+1}$. Fix $i_{m+1},\cdots, i_{j},i_{j+1}$ such that $(i_1,\cdots,i_m,\cdots,i_{j+1})$ is a partial order. By induction, for any $s \in \{m+1,\cdots,j\}$, step $s$ is \textbf{SEP}. Suppose that step $j+1$ is \textbf{NSEP}. Since step $j$ is \textbf{SEP}, the triangle $f_{i_j}$ splits a boundary $\alpha$ into two boundaries $\beta$ and $\gamma$. Then $f_{i_{j+1}}$ must connect $\beta$ and $\gamma$, as in \eqref{disconnect_connect}. But then, in the partial order $(i_1,\cdots,i_{j-2},i_{j},i_{j+1})$ (removing $i_{j-1}$), step $j$ would be \textbf{NSEP}, contradicting $\mathcal{P}_{j}$.
\end{itemize}
This concludes the induction. Item \eqref{item3} of Proposition~\ref{discover_nsep_first} is precisely $\mathcal{P}_{1000 \varepsilon |\mathbf{p}|}$, which concludes the proof.
\end{proof}

For the rest of this section, we fix $(i_1,\cdots,i_{m})$ given by Proposition~\ref{discover_nsep_first}. We write $t' = t \backslash (f_{i_1} \cup \cdots\cup f_{i_m})$. Let us write $t' \in \mathcal{T}_{(q_1,\cdots,q_{\ell'})}(n',g')$. We denote by $\{f'_1,\cdots,f'_{1000 \varepsilon |\mathbf{p}| - m}\} := \{f_{1},\cdots,f_{1000 \varepsilon |\mathbf{p}|}\} \backslash \{f_{i_1},\cdots,f_{i_m}\}$ the set of triangles that remain to be discovered.  

Although we are now working in $t'$, the notion of being $L-close$ is still defined with respect to $t$. 

We introduce $\mathcal{E}$ as the set of edges lying on the boundaries of $t'$ that are incident to one of the already discovered triangles $f_{i_1},\cdots,f_{i_m}$. Item~3 of Proposition~\ref{discover_nsep_first} ensures that the triangles $\{f_1',\cdots,f_{1000 \varepsilon |\mathbf{p}|-m}'\}$ are embedded in a planar way on $\partial t'$ (see Figure~\ref{forest_structure_embedding}).  

We introduce $\mathcal{H}$ as the set of holes $h$ of $\partial t' \cup f_1'\cup \cdots \cup f_{1000 \varepsilon |\mathbf{p}|-m}'$ that are incident to one of the triangles in $\{f_1',\cdots,f_{1000 \varepsilon |\mathbf{p}|-m}'\}$. 

We can now define a forest structure $\mathcal{F}$ (see Figure~\ref{forest_structure_embedding}) as follows:
\begin{itemize}
	\item[$\bullet$] The set of vertices $V$ consists of the edges $e \in \mathcal{E}$ and the holes $h \in \mathcal{H}$.
	\item[$\bullet$] The set of edges $E$ consists of:
	\begin{itemize}
	\item[$\bullet$] Edges $u = \{h,e\}$ where $h \in \mathcal{H}$, $e \in \mathcal{E}$, and $e$ lies on $h$.
	\item[$\bullet$] Edges $u = \{h,h'\}$ where $h,h' \in \mathcal{H}$ and $h,h'$ are separated by a single triangle $f_i'$.
	\end{itemize}
\end{itemize}

\begin{figure}[H]
\centering
\includegraphics[scale=0.65]{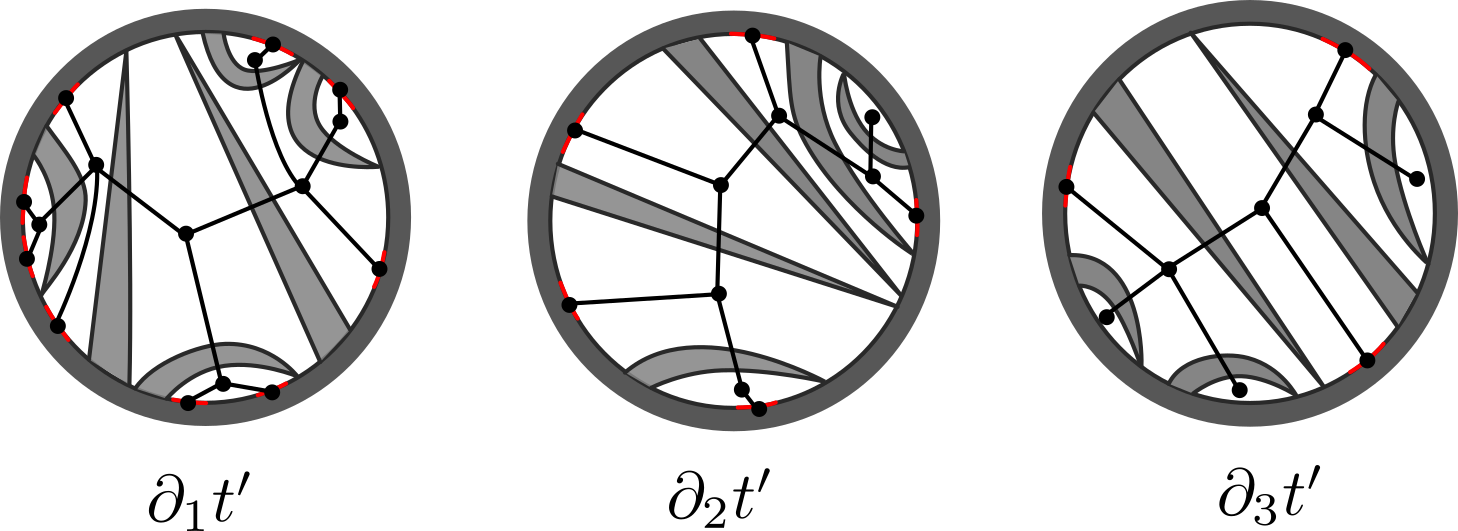}
\caption{We represent the triangles $\{f_1',\cdots,f_{1000 \varepsilon |\mathbf{p}|-m}'\}$ that remain to be discovered, embedded in the boundaries of $t'$. The edges in red correspond to $\mathcal{E}$. The forest $\mathcal{F}$ is also represented.}
\label{forest_structure_embedding}
\end{figure}

This forest encapsulates, in a more concise way, how the triangles $f_i'$ embed on $\partial t'$. In the next proposition, we aim to show that most of the vertices of the forest $\mathcal{F}$ have degree $2$. Thus, when choosing a uniform vertex $x \in \mathcal{F}$, its neighbourhood typically looks like a “line.” We relate this structure to the pattern observed in Figure~\ref{pattern} and the event $(\mathcal{C}_{L,p})$.

The number of vertices in the forest $\mathcal{F}$ is at least $2m+(1000 \varepsilon |\mathbf{p}|-m +1) = 1000 \varepsilon |\mathbf{p}| + m +1$, where $2m$ accounts for the number of edges in $\mathcal{E}$ and $1000 \varepsilon |\mathbf{p}|-m +1$ is a crude lower bound for the number of holes $h$ in $\partial t' \cup (f_1' \cup \cdots \cup f_{1000 \varepsilon |\mathbf{p}|-m}')$ that are incident to one of the triangles in $\{f_1',\cdots,f_{1000 \varepsilon |\mathbf{p}|-m}'\}$.

This is made more precise in the next proposition. For any $h \in \mathcal{H}$, we write $\phi(h)$ for the perimeter of $h$.  

\begin{proposition}\label{properties_forest}
The forest $\mathcal{F}$ satisfies the following properties:
\begin{enumerate}
    \item \label{it1forest} We have the bound $\displaystyle\sum_{h \in \mathcal{H}} \phi(h) \le 2|\mathbf{p}|$.
    \item\label{it2forest} The number of leaves plus isolated vertices is bounded by $3\varepsilon p^{-1}|\mathbf{p}|$. There are more than $1000\varepsilon|\mathbf{p}|$ vertices in $\mathcal{F}$ and more than $999\varepsilon|\mathbf{p}|$ edges.
    \item \label{it3forest} We have the bound $\displaystyle \sum_{\substack{v \in \mathcal{F}\\\deg(v)\ge 3}} \deg_{\mathcal{F}}(v) \le 18 \varepsilon p^{-1}|\mathbf{p}|$.
    \item\label{it4forest} Let $h$ be a vertex in $\mathcal{F}$ with degree $2$. Then $h \in \mathcal{H}$. Moreover, suppose that the neighbours of $h$ are $h_1,h_2 \in \mathcal{H}$. Write $f_1$ for the triangle separating $h_1$ and $h$, and $f_2$ for the triangle separating $h_2$ and $h$. Then $f_1$ and $f_2$ are $\phi(h)$-close in $t$.
\end{enumerate}
\end{proposition}

\begin{proof}
Let us prove item~\ref{it1forest}. This follows from the fact that each discovered triangle increases the total perimeter by at most $1$. Since we have discovered $1000 \varepsilon |\mathbf{p}| < |\mathbf{p}|$ triangles, we obtain
\[
\sum_{h \in \mathcal{H}} \phi(h) \le |\mathbf{p}| + 1000 \varepsilon |\mathbf{p}| \le 2|\mathbf{p}|.
\]

\noindent Now we prove item~\ref{it2forest}. The union of the leaves and isolated vertices of $\mathcal{F}$ is given by the vertices $e \in \mathcal{E}$ and the holes $h \in \mathcal{H}$ having degree $1$ in $\mathcal{F}$. We have $\# \mathcal{E} \le 2m \le 2\varepsilon p^{-1}|\mathbf{p}|$.  
Fix $h \in \mathcal{H}$ with exactly one neighbour in $\mathcal{F}$. Then, we claim that
\begin{align}\label{holes_degree_one_phi_large}
 \phi(h) \ge 2\varepsilon^{-1} p.
\end{align}
Indeed, let $f_0'$ be the unique triangle to which $h$ is incident, and $e_0$ the edge of $\partial t'$ behind which $f_0'$ lies. Let $e_1$ be the edge of $f_0'$ lying in $h$. Denote by $a,b$ the two endpoints of $e_1$. Then $h \backslash e_1$ defines a path $\gamma$ between $a$ and $b$. Since no edge of $\mathcal{E}$ lies on $h$, the path $\gamma$ is composed entirely of edges of $\partial t$ (see Figure~\ref{forest_proof}).  

By the hypothesis of Proposition~\ref{many_events_C}, $(t,e_0)$ satisfies either $(\mathcal{H}_{\varepsilon^{-1}p})$ or $(\mathcal{O})$. It cannot satisfy $(\mathcal{O})$, since $a$ and $b$ lie on the same boundary of $\partial t$. Thus, $(t,e_0)$ satisfies $(\mathcal{H}_{2\varepsilon^{-1}p})$, and hence $d_{\partial t}(a,b) \ge 2\varepsilon^{-1}p - 1$. In particular, the path $\gamma$ has length at least $2\varepsilon^{-1}p - 1$, which implies $\phi(h) \ge 2\varepsilon^{-1} p$.  

Combining this with item~\ref{it1forest}, we find that the number of leaves plus isolated vertices is bounded by
\[
2\varepsilon p^{-1}|\mathbf{p}| + (2\varepsilon^{-1}p)^{-1} 2|\mathbf{p}| = 3 \varepsilon p^{-1}|\mathbf{p}|.
\]
The total number of vertices in $\mathcal{F}$ is at least $1000 \varepsilon |\mathbf{p}|+m+1 \ge 1000 \varepsilon |\mathbf{p}|$. In a forest, the total number of edges equals $\#V - C$, where $C$ denotes the number of connected components. Since the number of connected components is bounded by the number of leaves, the total number of edges is bounded below by 
\[
1000 \varepsilon |\mathbf{p}| - 4 \varepsilon p^{-1}|\mathbf{p}| \ge 999 \varepsilon |\mathbf{p}|.
\]

Now we prove item~\ref{it3forest}. Since $\mathcal{F}$ is a forest, we have $\sum_{v \in \mathcal{F}} (\deg_{\mathcal{F}}(v)-2) \le 0$. Thus,
\begin{align*}
\sum_{\substack{v \in \mathcal{F}\\\deg_{\mathcal{F}}(v)\ge 3}} \deg_{\mathcal{F}}(v)
&\le 3\sum_{\substack{v \in \mathcal{F}\\\deg_{\mathcal{F}}(v)\ge 3}} (\deg_{\mathcal{F}}(v)-2)
\le 3\sum_{\substack{v \in \mathcal{F}\\\deg_{\mathcal{F}}(v)\in \{0,1\} }} (2-\deg_{\mathcal{F}}(v))
\underset{\mathrm{Item}~\ref{it2forest}}{\le} 18\varepsilon p^{-1}|\mathbf{p}|.
\end{align*}

\begin{figure}[H]
\centering
\includegraphics[scale=0.8]{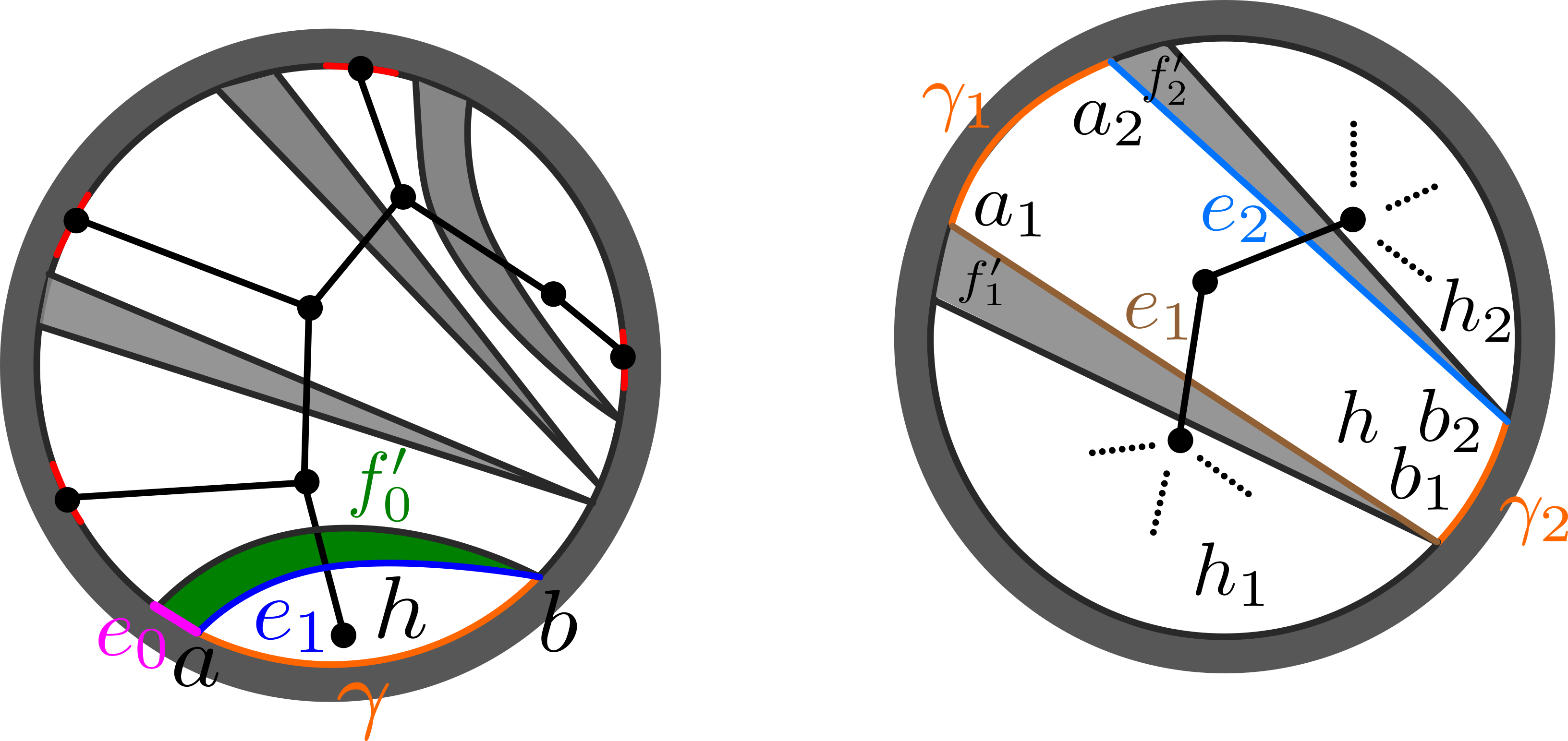}
\caption{Left: the hole $h$ has degree $1$ in $\mathcal{F}$. The edges $e\in\mathcal{E}$ are shown in red. The path $\gamma$ does not use any red edge since $h$ has degree $1$; thus, it only uses edges of $\partial t$.  
Right: the hole $h$ has degree $2$ in $\mathcal{F}$ and its two neighbours are holes $h_1$ and $h_2$. Since $h$ has degree $2$, the paths $\gamma_1$ and $\gamma_2$ only use edges of $\partial t$. Hence $d_{\partial t}(a_1,a_2) \le \phi(h) -1$ and $d_{\partial t}(b_1,b_2) \le \phi(h) -1$, so $f_1$ and $f_2$ are $\phi(h)$-close in $t$.}
\label{forest_proof}
\end{figure}

Finally, we prove item~\ref{it4forest}.  
Let $h$ be a vertex in $\mathcal{F}$ with degree $2$. Since the vertices in $\mathcal{E}$ are leaves, we deduce that $h \in \mathcal{H}$. Suppose the two neighbours $h_1,h_2$ of $h$ in $\mathcal{F}$ are elements of $\mathcal{H}$. Write $f_1$ for the triangle separating $h_1$ and $h$, and $f_2$ for the triangle separating $h_2$ and $h$. Let $e_1$ (resp. $e_2$) be the edge of $f_1$ (resp. $f_2$) that lies in $h$. Denote by $a_1,b_1$ (resp. $a_2,b_2$) the two endpoints of $e_1$ (resp. $e_2$). Then $h \backslash \{e_1,e_2\}$ consists of two paths $\gamma_1$ and $\gamma_2$ connecting $a_1,a_2$ and $b_1,b_2$.  

These paths are composed of edges of $\partial t$, since $h$ has degree $2$ (there is no edge of $\mathcal{E}$ lying on $h$). They have length at most $\phi(h)-1$. This proves that $f_1$ and $f_2$ are $\phi(h)$-close in $t$ (see Figure~\ref{forest_proof}).
\end{proof}

Now we can prove Proposition~\ref{many_events_C}.
\begin{proof}
For any edge $u$ of $\mathcal{F}$ such that both endpoints $a$ and $b$ of $u$ have degree $2$ in $\mathcal{F}$, the endpoints $a,b$ are holes in $\mathcal{H}$, and we can associate to $u$ the triangle $f(u)$ that separates $a$ and $b$. We also define $e(u)$ as the edge on $\partial t$ such that $f(u)$ lies behind $e(u)$.

Let $X \subset \mathcal{F}$ be the subset of edges $u \in \mathcal{F}$ such that the endpoints $x,y$ of $u$ are at graph distance (in $\mathcal{F}$) at least $p$ from any vertex of degree different from $2$ in $\mathcal{F}$. In words, around $u$, the forest $\mathcal{F}$ looks like a line segment of length at least $2p$ (see Figure~\ref{line_forest}).
\begin{figure}[H]
\centering
\includegraphics[scale=0.3]{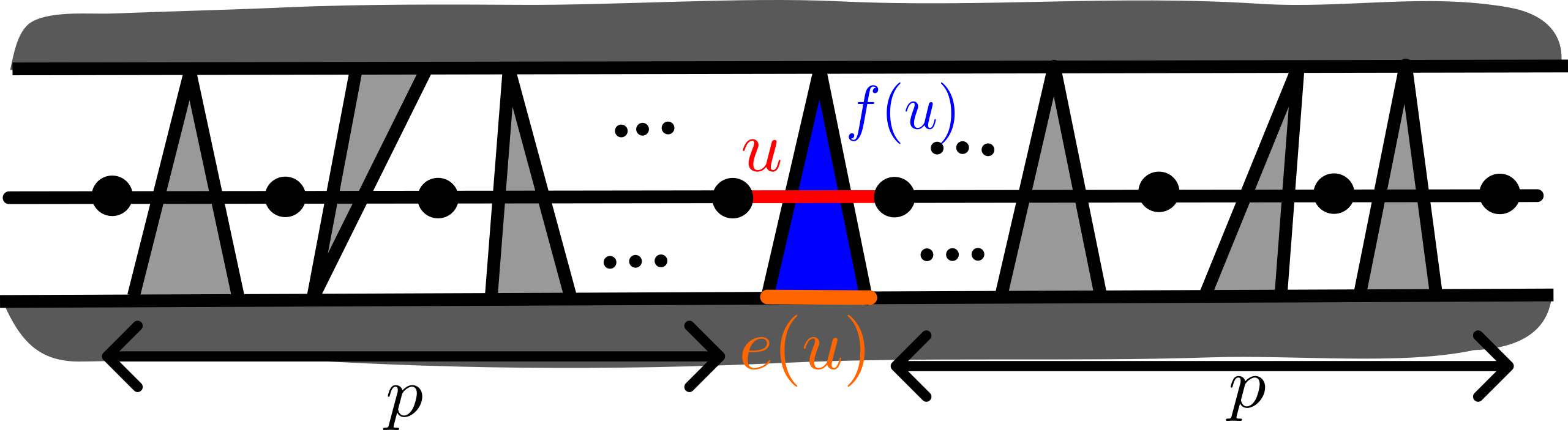}
\caption{We represent an edge $u \in X$. Thus around $u$, the forest $\mathcal{F}$ looks like a line of length at least $2p$. In particular, for any $u' \in \mathcal{E}(u)$, the triangles $t(u)$ and $t(u')$ are $\somme{h \in \mathcal{V}(u)}{}{\phi(h)}$-close.}
\label{line_forest}
\end{figure}
We denote by $\mathcal{V}(u)$ the set of vertices of $\mathcal{F}$ at distance at most $p$ from an endpoint of $u$, and by $\mathcal{E}(u)$ the set of edges of $\mathcal{F}$ whose two endpoints both belong to $\mathcal{V}(u)$. Hence, for any $u \in X$ and any $v \in \mathcal{V}(u)$, we have $\deg_{\mathcal{F}}(v) = 2$.

By item~\eqref{it4forest} of Proposition~\ref{properties_forest}, for any $u' \in \mathcal{E}(u)$, the triangles $f(u)$ and $f(u')$ are $\sum_{h \in \mathcal{V}(u)} \phi(h)$-close (see Figure~\ref{line_forest}). It follows that if 
\[
\sum_{h \in \mathcal{V}(u)} \phi(h) \le \varepsilon^{-1}p,
\]
then $(t,e(u))$ satisfies $(\mathcal{C}_{\varepsilon^{-1},p})$ (see Definition~\ref{defClp}).  

Let us therefore define
\[
X' = \bigg\{u \in X : \sum_{h \in \mathcal{V}(u)} \phi(h) \le \varepsilon^{-1} p \bigg\}.
\]
To prove Proposition~\ref{many_events_C}, it suffices to show that $\#E(X') \ge \varepsilon|\mathbf{p}|$.

Let $B$ denote the set of vertices in $\mathcal{F}$ with degree different from $2$. For any $r \ge 0$, let $C_r$ be the set of edges in $\mathcal{F}$ having at least one endpoint at graph distance at most $r-1$ from $B$.  
From items~\eqref{it2forest} and~\eqref{it3forest} of Proposition~\ref{properties_forest}, we obtain
\[
\#E(C_1) \le 18 \varepsilon p^{-1}|\mathbf{p}|.
\]
By induction, it follows that for any $r \ge 1$, we have
\[
\#E(C_{r+1}\backslash C_r) \le 18\varepsilon  p^{-1}|\mathbf{p}|.
\]
It follows that for any $r \ge 1$ we have 
$$
\#E(C_r) \le 18 r\varepsilon p^{-1}|\mathbf{p}|.
$$
Note that $E(\mathcal{F}) \backslash E(C_{2p}) \subset E(X)$. Using again item~\eqref{it2forest} of Proposition~\ref{properties_forest}, we deduce that
\[
\# E(X) \ge 963 \varepsilon |\mathbf{p}|.
\]
Now, we estimate the total contribution of the perimeters:
\[
\sum_{u \in X} \sum_{h \in \mathcal{V}(u)} \phi(h)
\le 2p \sum_{h \in \mathcal{H}} \phi(h)
\underset{\mathrm{Item}~\ref{it1forest}}{\le} 4p|\mathbf{p}|,
\]
where we used the fact that a hole $h \in \mathcal{H}$ can belong to at most $2p$ subsets $\mathcal{V}(u)$ with $u \in X$.  

It follows that there are at most $4\varepsilon|\mathbf{p}|$ edges $u \in X$ such that $\sum_{h \in \mathcal{V}(u)} \phi(h) \ge \varepsilon^{-1}p$. Therefore,
\[
\# X' \ge \#X - 4\varepsilon |\mathbf{p}| \ge 959 \varepsilon |\mathbf{p}| \ge \varepsilon |\mathbf{p}|.
\]
This concludes the proof.
\end{proof}

\subsubsection{Proof of Proposition~\ref{no_event_H_or_O}}\label{third_subsection}
This section concludes the proof of Proposition~\ref{no_event_H_or_O}. The proof consists in adapting the argument of Section~\ref{planarity_middle}.
\begin{proof}
In this proof, we will treat Case $1$ of Proposition~\ref{many_events_C}. Case $2$ is handled by Proposition~\ref{bound_event_C}. Fix $\varepsilon > 0$. By Proposition~\ref{bound_event_C}, fix a constant $p > 0$ such that, for $n$ large enough, we have 
\begin{align}\label{first_equation}
\Pf\Big((T_{n,g_n,\mathbf{p}^{n}},e^n)\text{ satisfies }(\mathcal{C}_{\varepsilon^{-1},p})\Big) \le \frac{\varepsilon^2}{20000}.
\end{align}

We reason by contradiction. Suppose that there exist arbitrarily large values of $n$ such that
\begin{center}
$\Pf\Big((T_{n,g_n,\mathbf{p}^{n}},e^n) \text{ satisfies }(\mathcal{H}_{2\varepsilon^{-1}p}) \text{ or }(\mathcal{O})\Big) \ge \varepsilon > 0.$
\end{center} 
Since $e^n$ is chosen uniformly at random on $\partial T_{n,g_n,\mathbf{p}^n}$, there is a proportion at least $\frac{\varepsilon}{2}$ of triangulations $t \in \mathcal{T}_{\mathbf{p}^{n}}(n,g_n)$ such that there exist distinct edges $(e_{1},\cdots,e_{\frac{\varepsilon}{2}|\mathbf{p}^{n}|})$ on the boundaries of $t$ with
\begin{center}
$(t,e_i)$ satisfying $(\mathcal{H}_{2\varepsilon^{-1}p})$ or $(\mathcal{O})$.
\end{center}  

By \eqref{first_equation}, there is a proportion at most $\frac{\varepsilon}{4}$ of triangulations $t \in \mathcal{T}_{\mathbf{p}^{n}}(n,g_n)$ such that there exist distinct edges $(e_{1}',\cdots,e_{\frac{\varepsilon}{4000}|\mathbf{p}^n|}')$ on $\partial t$ for which $(t,e'_j)$ satisfies $(\mathcal{C}_{\varepsilon^{-1},p})$.  
Combining this with Proposition~\ref{many_events_C}, we deduce that there is a proportion at least $\frac{\varepsilon}{4}$ of triangulations $t \in \mathcal{T}_{\mathbf{p}^{n}}(n,g_n)$ satisfying the following property:  
there exist distinct edges $e_{1},\cdots,e_{r}$ with $\varepsilon p^{-1}|\mathbf{p}^{n}| \le r \le |\mathbf{p}^{n}|$ on $\partial t$ such that, writing $f_1,\cdots,f_{r}$ for the triangles behind $e_1,\cdots,e_{r}$, we have
\begin{center}
$t \backslash (f^1 \cup \cdots \cup f^{r}) \in \mathcal{T}_{(q_1,\cdots,q_{\ell'})}(n',g')$
\end{center}
for some $(q_1,\cdots,q_{\ell'})$, where $(\ell_n-\ell') + 2(g_n-g') = r$.  
We also have $n' \le n + r \le n + |\mathbf{p}^n|$.  

For such triangulations $t$, define the pair of triangulations with holes
\[
\varphi(t) = \big(t \backslash (f^1\cup \cdots\cup f^{r}) , \partial t \cup f^1\cup \cdots \cup f^{r}\big).
\]
The function $\varphi$ is injective since one can recover $t$ by gluing $t \backslash (f^1\cup \cdots\cup f^{r})$ back to $\partial t \cup f^1\cup \cdots \cup f^{r}$.  

Using Lemma~\ref{bounded_ratio_vertices} and Lemma~\ref{add_new_boundary} (whose assumptions can be verified directly), the number of inputs is at least
\begin{align}\label{bound_input}
\exp(-C_{\theta}|\mathbf{p}^n|)\, n^{\ell_n -1}\,\frac{\varepsilon}{4}\,\tau(n,g_n).
\end{align}
On the other hand, the number of outputs is bounded by 
\[
\sum_{\substack{r,f^1,\cdots,f^r}} \tau_{\mathbf{q}}(n',g'),
\]
where $\mathbf{q} = (q_1,\cdots,q_{\ell'}),n',g'$ depend on $r,f^1,\cdots,f^r$.  

Fix $r,f^1,\cdots,f^r$ in the sum.  
Applying first Lemma~\ref{bounded_ratio_vertices} using $n' \le n + r \le n + |\mathbf{p}^n|$, and then Lemma~\ref{add_new_boundary}, we deduce the sequence of inequalities
\begin{align}\label{serie_inequality}
\tau_{\mathbf{q}}(n',g') &\le  \tau_{\mathbf{q}}(n+|\mathbf{p}^n|,g') \le (6n)^{\ell'-1}\tau(n+|\mathbf{p}^n|,g').
\end{align}

The Goulden–Jackson recursion formula \cite[Theorem~4]{GOULDEN2008932} reads:
\begin{align*}
\tau(n,g) = \frac{4}{n+1}\Big(n(3n-2)(3n-4)\tau(n-2,g-1)
+ \sum_{\substack{i+j=n-2\\h+k = g}} (3i+2)(3j+2)\tau(i,h)\tau(j,k)\Big).
\end{align*}
For $n \ge 2$, this implies
\begin{align*}
\tau(n,g) \ge n^2\tau(n-2,g-1).
\end{align*}
Hence, the right-hand side of \eqref{serie_inequality} is bounded by
\begin{align}\label{int2}
C(6n)^{\ell'-1}n^{-2(g_n-g')}\tau(n+|\mathbf{p}^n|+2(g_n-g'),g_n),
\end{align}
where $C > 0$ is a constant.  
Note that this is the only place in the paper where we use the Goulden–Jackson recursion formula.  

Iterating Lemma~\ref{bounded_ratio_vertices} and using the fact that $g_n-g' \le |\mathbf{p}^n|$, we further bound the last term by 
\begin{align}\label{int}
C\exp(C_{\theta}|\mathbf{p}^n|)n^{\ell'-2(g_n-g')-1}\tau(n,g_n).
\end{align}

Since the number of inputs of $\varphi$ is bounded above by the number of outputs, we obtain
\begin{align*}
\exp(-C_{\theta}|\mathbf{p}^n|)n^{\ell_n -1}\frac{\varepsilon}{4}\tau(n,g_n)
\le 
\sum_{\substack{r,f^1,\cdots,f^r}} C\exp(C_{\theta}|\mathbf{p}^n|)n^{\ell'-2(g_n-g')-1}\tau(n,g_n),
\end{align*}
or equivalently,
\begin{align}\label{bound1}
1 \le \frac{4C}{\varepsilon}
\sum_{\substack{r,f^1,\cdots,f^r}} \exp(C_{\theta}|\mathbf{p}^n|)n^{-r},
\end{align}
where we recall that $(\ell_n-\ell')+2(g_n-g') = r$.

For a fixed $r$ satisfying $\varepsilon p^{-1} |\mathbf{p}^n| \le r \le |\mathbf{p}^n|$, the number of possible choices for $\{f^1,\cdots,f^{r}\}$ is bounded by
\begin{align}\label{bound_choices_ti}
2^{|\mathbf{p}^{n}|}\prod_{i=0}^{r-1}(|\mathbf{p}^{n}|+i).
\end{align}

Indeed, the factor $2^{|\mathbf{p}^{n}|}$ accounts for the number of possible sets of edges $\{e^1,\cdots,e^{r}\}$.  
Then, for $f^1$ we must choose a third vertex, given $|\mathbf{p}^{n}|$ choices.  
For each subsequent $f^{i+1}$, we must choose a third vertex among the vertices on the boundaries of $t\backslash (f^1 \cup \cdots \cup f^{i})$, whose total perimeter is at most $|\mathbf{p}^{n}|+i$.  
This gives the factor $\prod_{i=0}^{r-1}(|\mathbf{p}^{n}|+i)$.  
Using $|\mathbf{p}^n|+i \le 2|\mathbf{p}^n|$, we deduce that the right-hand side of \eqref{bound1} is bounded by
\begin{align}\label{bound2_output}
\frac{4C}{\varepsilon}\exp(C_{\theta}'|\mathbf{p}^n|)\sum_{r=\varepsilon p^{-1} |\mathbf{p}^n|}^{|\mathbf{p}^n|}\Big(\frac{|\mathbf{p}^n|}{n}\Big)^{r}
\le 2C\exp(C_{\theta}'|\mathbf{p}^n|)\Big(\frac{|\mathbf{p}^n|}{n}\Big)^{\varepsilon p^{-1} |\mathbf{p}^n|}.
\end{align}

Since $|\mathbf{p}^{n}| = o(n)$, the right-hand side tends to $0$ as $n \to +\infty$.  
This leads to a contradiction in \eqref{bound1}. This completes the proof.
\end{proof}
\subsection{Subsequential limits are triangulations of the half-plane}\label{conv_subsequences}
This section is devoted to proving the following Proposition. We recall that weak Markov triangulations are defined in Definition~\ref{weak_Markovian_halfplane}.
\begin{proposition}\label{tightness_d_loc}
    The sequence $(T_{n,g_n,\mathbf{p}^{n}},e^n)_{n\ge 0}$ is tight for $d_{\mathrm{loc}}$. Moreover, any subsequential limit is an infinite half-plane triangulation that is weakly Markovian.
\end{proposition}
\noindent To prove this result, we explore the neighbourhood of $e^n$ using a suitably chosen \emph{peeling exploration}. See Section~\ref{peeling_def} for the definition of the peeling exploration. Here, we consider the filled-in peeling, which we now define. For a peeling algorithm $\mathcal{A}$, $t \in \mathcal{T}_{\mathbf{p}^n}(n,g_n)$, and $e \in \partial t$, we define $\bar{\mathcal{E}}^{\mathcal{A}}_{0}(t,e)$ as the polygon $\partial_i$ on which $e$ lies, rooted at $e$. For $j \ge 0$, we define $\widehat{\mathcal{E}}^{\mathcal{A}}_{j+1}(t,e)$ as the triangulation with holes obtained from $\bar{\mathcal{E}}^{\mathcal{A}}_{j}(t,e)$ by revealing the triangle behind $\mathcal{A}(\bar{\mathcal{E}}^{\mathcal{A}}_{j}(t,e))$ in $t$. Then, we define $\bar{\mathcal{E}}^{\mathcal{A}}_{j+1}(t,e)$ as the map obtained by gluing to $\widehat{\mathcal{E}}^{\mathcal{A}}_{j+1}(t,e)$ the components of $(t,e) \backslash \mathcal{E}^{\mathcal{A}}_{j+1}(t,e)$ that are planar triangulations. The discovery of the $j^{th}$ triangle is referred to as step $j$. 
 
We recall the definitions of $\mathcal{L}_k, \mathcal{R}_k,(\mathcal{H}_k),(\mathcal{O})$ introduced in the introduction of Section~\ref{section_boundary}. Propositions~\ref{Finite_holes_filled_by_finite_maps} and~\ref{no_event_H_or_O} ensure that for any peeling algorithm $\mathcal{A}$ with high probability the pathological cases \eqref{list_pathological_cases1}, \eqref{list_pathological_cases2}, \eqref{list_pathological_cases3}, and \eqref{list_pathological_cases4} do not occur at the first step of the peeling exploration started from the root edge $e^n$, chosen uniformly at random on $\partial T_{n,g_n,\mathbf{p}^n}$. We prove that with high probability, these pathological cases still do not occur as long as we consider only a constant number of peeling steps. Let us make this more precise
\begin{definition1}\label{Defbad}
  Fix $A,a > 0$ and $t \in \mathcal{T}_{\mathbf{p}^n}(n,g_n)$. Let $i \in \{1,\cdots,\ell_n\}$ and $e \in \partial_i t$. Fix a peeling algorithm $\mathcal{A}$. Fix $j \ge 0$ and assume that $\bar{\mathcal{E}}^{\mathcal{A}}_{j}(t,e)$ is a planar triangulation with one hole. We say that step $j+1$ is $\mathbf{BAD}^{\mathcal{A}}(A,a)$ (see Figure~\ref{badcases}) if it satisfies one of the conditions below:
\begin{enumerate}
    \item \label{bad1} The triangulation $\bigg((t,e) \backslash \bar{\mathcal{E}}^{\mathcal{A}}_{j}(t,e),\mathcal{A}(\bar{\mathcal{E}}^{\mathcal{A}}_{j}(t,e))\bigg)$ satisfies $\mathcal{L}_k^{\mathrm{nsep}}$ or $\mathcal{R}_k^{\mathrm{nsep}}$ with $k \in \{0,\cdots,A-1\}$.
    \item \label{bad2} The triangulation $\bigg((t,e) \backslash \bar{\mathcal{E}}^{\mathcal{A}}_{j}(t,e),\mathcal{A}(\bar{\mathcal{E}}^{\mathcal{A}}_{j}(t,e))\bigg)$ satisfies $\mathcal{L}_k^{\mathrm{sep}}$ (resp. $\mathcal{R}_k^{\mathrm{sep}}$) with $k \in \{0,\cdots,A-1\}$, and writing $t_1$ for the component filling the left hole (resp. the right hole), we have $t_1 \notin \bigsqcup_{i=0}^{a}\mathcal{T}_{k+1}(i,0)$.
     \item \label{bad3} The triangulation $\bigg((t,e) \backslash \bar{\mathcal{E}}^{\mathcal{A}}_{j}(t,e),\mathcal{A}(\bar{\mathcal{E}}^{\mathcal{A}}_{j}(t,e))\bigg)$ satisfies $(\mathcal{H}_{A+1})$.
    \item \label{bad4} The triangulation $\bigg((t,e) \backslash \bar{\mathcal{E}}^{\mathcal{A}}_{j}(t,e),\mathcal{A}(\bar{\mathcal{E}}^{\mathcal{A}}_{j}(t,e))\bigg)$ satisfies $(\mathcal{O})$.
\end{enumerate}
 Otherwise, it is said to be $\mathbf{GOOD}^{\mathcal{A}}(A,a)$. Note that if step $j+1$ is $\mathbf{GOOD}^{\mathcal{A}}(A,a)$, then $\bar{\mathcal{E}}^{\mathcal{A}}_{j+1}(t,e)$ is again a planar triangulation with one hole. 
\end{definition1} 

\begin{figure}[H]
\includegraphics[scale=0.45]{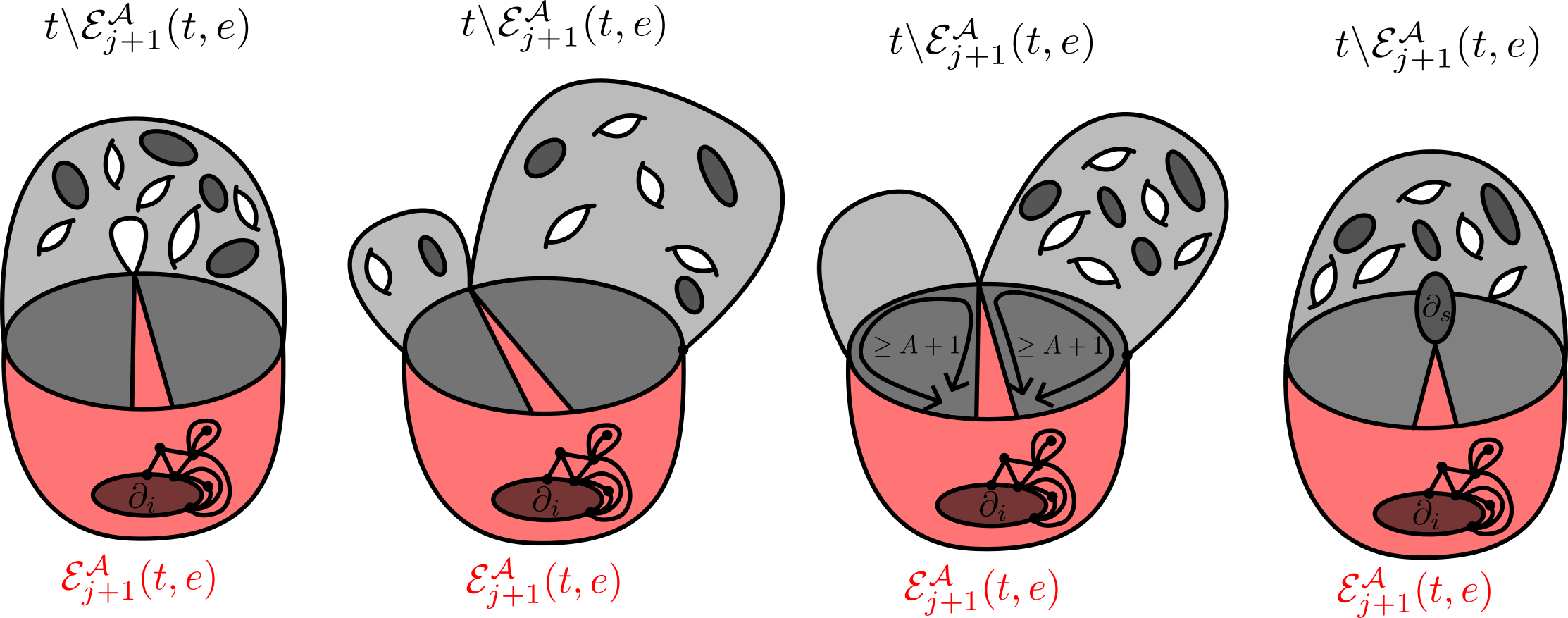}

\caption{Illustration of Definition~\ref{Defbad}.}
\label{badcases}
\end{figure} 
 
We define the event $\mathbf{GOOD}^{\mathcal{A}}_{n}(A,a,k)$ as the event that the first $k$ steps of the peeling exploration $(\bar{\mathcal{E}}^{\mathcal{A}}_j(T_{n,g_n,\mathbf{p}^n},e^n))_{j \ge 0}$ are $\mathbf{GOOD}^{\mathcal{A}}(A,a)$. We write $\mathbf{BAD}^{\mathcal{A}}_n(A,a,k) = \mathbf{GOOD}^{\mathcal{A}}_{n}(A,a,k)^{c}$. The next proposition shows that for any fixed $k \ge 1$, for suitably chosen values of $A,a \ge 0$, and for $n$ large enough, the event $\mathbf{GOOD}^{\mathcal{A}}_{n}(A,a,k)$ occurs with high probability.
 
\begin{proposition}\label{no_BAD_steps}
    For any $k \ge 1$ and any $\varepsilon > 0$, there exist $A,a > 0$ such that, for $n$ large enough and for any peeling algorithm $\mathcal{A}$, we have
    \begin{align*}
        \Pf\bigg(\mathbf{GOOD}^{\mathcal{A}}_{n}(A,a,k) \bigg) \ge 1-\varepsilon. 
    \end{align*}
\end{proposition}

\begin{proof}
    Observe that, due to the invariance under translation along the boundary, the probability of the event considered is independent of the peeling algorithm $\mathcal{A}$. We proceed by induction on $k \ge 1$. The base case $k = 1$ follows directly from Proposition~\ref{no_event_H_or_O} and Proposition~\ref{Finite_holes_filled_by_finite_maps}.

    Assume that the result holds for all $ 1 \le j \le k$. We now prove it for $k+1$. Fix $\varepsilon > 0$. By the induction hypothesis, there exist $A, a > 0$ such that, for any peeling algorithm $\mathcal{A}$ and for sufficiently large $n$, we have 
    \begin{align}\label{first_step_good}
        \mathbb{P}\left(\mathbf{GOOD}^{\mathcal{A}}_{n}(A,a,k)\right) \ge 1-\frac{\varepsilon}{3} \quad \text{and} \quad \mathbb{P}\left(\mathbf{GOOD}^{\mathcal{A}}_{n}(A,a,1)\right) \ge 1-\frac{\varepsilon}{3}.
    \end{align}
        We now introduce a specifically constructed peeling algorithm. For any oriented edge $e \in \partial T_{n,g_n,\mathbf{p}^n}$, let $\theta(e)$ denote the oriented edge obtained by shifting $e$ by one position to the right along the boundary $\partial T_{n,g_n,\mathbf{p}^n}$. For each $j \ge 0$, we define a peeling algorithm $\mathcal{A}_j$ as follows. For a triangulation with holes $t$ rooted at an oriented boundary edge $e_0$, we define:\footnote{If the root edge $e_0$ does not lie on a boundary, we define $\mathcal{A}(t)$ arbitrarily.}
    \begin{enumerate}
        \item \label{case1} If the edge $\theta^{3jA}(e_0)$ lies on the boundary of a hole in $t$, we set $\mathcal{A}_j(t) = \theta^{3jA}(e_0)$.
        \item \label{case2} Otherwise, the edge $\mathcal{A}_j(t)$ is chosen arbitrarily.
    \end{enumerate}
    We now define a peeling algorithm $\mathcal{A}$ which, intuitively, applies $\mathcal{A}_k$ at the $k$-th step of the exploration. More precisely, for a triangulation with holes $t$:
    \begin{itemize}
        \item If $t = \bar{\mathcal{E}}^{\mathcal{A}}_j(t_0,e_0)$ for some $j \ge 0$, then we define $\mathcal{A}(t) = \mathcal{A}_j(t)$.
        \item Otherwise, we define $\mathcal{A}(t)$ arbitrarily. 
    \end{itemize}
    Observe that, by induction on $i \in \{0,\dots,j\}$, the triangulation with holes $\bar{\mathcal{E}}^{\mathcal{A}}_j(t_0,e_0)$ uniquely determines the sequence $(\bar{\mathcal{E}}^{\mathcal{A}}_0(t_0,e_0),\dots, \bar{\mathcal{E}}^{\mathcal{A}}_i(t_0,e_0))$. Consequently, a triangulation with holes $t$ cannot satisfy $t = \bar{\mathcal{E}}^{\mathcal{A}}_i(t_0,e_0) =  \bar{\mathcal{E}}^{\mathcal{A}}_j(t_0,e_0)$ with $j \neq i$. This uniqueness of the step index ensures that $\mathcal{A}$ is well-defined.
    
    We now seek to bound
    \begin{align*}
        \mathbb{P}\left(\mathbf{GOOD}^{\mathcal{A}}_n(A,a,k) \cap \mathbf{BAD}^{\mathcal{A}}_n(A,a,k+1)\right).
    \end{align*}
    Let $I \in \{1,\dots,\ell_n\}$ be the random index such that $e^n \in \partial_I$. Note that if $p^n_I > 6Ak$ and the event $\mathbf{GOOD}^{\mathcal{A}}_n(A,a,k)$ occurs, then the edge $\theta^{3kA}(e^n)$ belongs to the boundary of a hole of $\bar{\mathcal{E}}^{\mathcal{A}}_{k}(T_{n,g_n,\mathbf{p}^n},e^n)$. Moreover, if the event $\mathbf{BAD}^{\mathcal{A}}_n(A,a,k+1)$ also occurs, then it implies the occurrence of $\mathbf{BAD}^{\mathcal{A}_{k}}_n(A,a,1)$ (see Figure~\ref{badcasek}).

    \begin{figure}[H]
        \centering
        \includegraphics[scale=0.3]{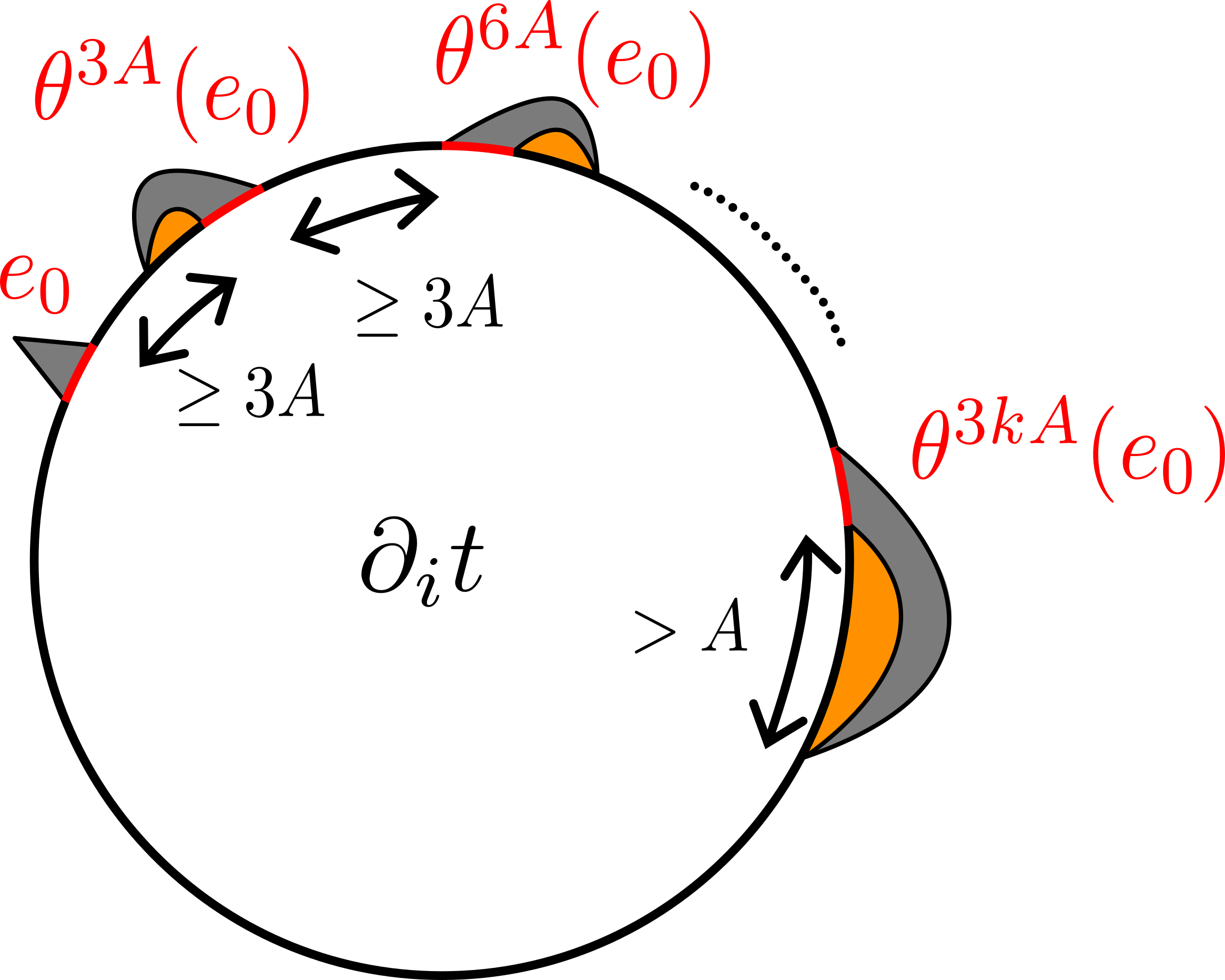}
        \caption{In this example, we depict a triangulation $(t,e_0)$ and the first $k+1$ steps of its peeling exploration using algorithm $\mathcal{A}$. The orange components represent planar triangulations filling the holes. The first $k$ steps satisfy $\mathbf{GOOD}^{\mathcal{A}}(A,a)$, whereas the $(k+1)$-th step is $\mathbf{BAD^{\mathcal{A}}}(A,a)$. Indeed, the $(k+1)$-th discovered triangle encloses more than $A$ edges to its right, which satisfies condition~\ref{bad3} in Definition~\ref{Defbad}. Step $1$ is also $\mathbf{BAD}^{\mathcal{A}_k}(A,a)$ since the first $k$ triangles do not touch any edge at distance at most $A$ from $\theta^{3kA}(e_0)$.}
        \label{badcasek}
    \end{figure} 

    Recall the event $\mathcal{P}_n := \left\{ e^n \text{ belongs to a boundary component with perimeter at least } \left(\frac{|\mathbf{p}^n|}{\ell_n}\right)^{1/2} \right\}$. Using Proposition~\ref{perimeter_typically_large}, we deduce that for sufficiently large $n$,
    \begin{align*}
        \mathbb{P}\left(\mathbf{GOOD}^{\mathcal{A}}_n(A,a,k) \cap \mathbf{BAD}^{\mathcal{A}}_n(A,a,k+1)\right) \le  \mathbb{P}\left(\mathbf{BAD}^{\mathcal{A}_{k}}_n(A,a,1)\right)+ \mathbb{P}(\mathcal{P}_n^{c}).
    \end{align*}
    Using the induction hypothesis and the fact that the second term tends to $0$ as $n \to +\infty$, we conclude that for $n$ large enough,
    \begin{align*}
        \mathbb{P}\left(\mathbf{GOOD}^{\mathcal{A}}_n(A,a,k) \cap \mathbf{BAD}^{\mathcal{A}}_n(A,a,k+1)\right) \le \frac{\varepsilon}{2}.
    \end{align*}
    Finally, for sufficiently large $n$,
    \begin{align*}
        \mathbb{P}\left(\mathbf{GOOD}^{\mathcal{A}}_n(A,a,k+1)\right) 
        &\ge \mathbb{P}\left(\mathbf{GOOD}^{\mathcal{A}}_n(A,a,k)\right) - \mathbb{P}\left(\mathbf{GOOD}^{\mathcal{A}}_n(A,a,k) \cap \mathbf{BAD}^{\mathcal{A}}_n(A,a,k+1)\right) \\
        &\ge 1 - \varepsilon.
    \end{align*}
    This concludes the proof.
\end{proof}

Now we move on to the proof of Proposition~\ref{tightness_d_loc}. To do so, we will apply Proposition~\ref{no_BAD_steps} to a peeling algorithm that discovers the balls around the root edge. Fix $t$ a planar triangulation with one hole\footnote{If $t$ does not satisfy this assumption, we define $\mathcal{A}_{\mathrm{metric}}(t,e)$ arbitrarily.} and $e$ an oriented edge of $t$. Let $\rho$ denote the origin of $e$. We define $\mathcal{A}_{\mathrm{metric}}(t,e)$ as follows:
\begin{itemize}
	\item[$\bullet$] Choose $v \in \partial^{*}t$ such that $d_{t}(\rho,v)$ is minimal, taking $v$ having the leftmost geodesic from $\rho$ to $v$ in case of equality.
	\item[$\bullet$] Define $\mathcal{A}_{\mathrm{metric}}(t,e)$ as the edge immediately to the right of $v$ on the hole.
\end{itemize}

For any $n,k \ge 0$, we set 
\[
t_k^n := \bar{\mathcal{E}}^{\mathcal{A}_{\mathrm{metric}}}_{k}(T_{n,g_n,\mathbf{p}^n},e^n),
\]
the rooted triangulation with holes obtained after $k$ steps of the peeling exploration with the algorithm $\mathcal{A}_{\mathrm{metric}}$. We also denote by $v^n_k$ the left vertex of the peeled edge $\mathcal{A}_{\mathrm{metric}}(t^n_k)$. For any $n,k \ge 0$, we define
\begin{align}\label{def_sigma}
\sigma_n(k) := \min_{x \in \partial^{*}t_k^n} d_{t_k^n}(x,\rho^n) = d_{t^n_k}(v^n_k,\rho^n).
\end{align} We aim to prove that, for any $r \ge 0$, with high probability and for $n,k$ large enough, $\sigma_n(k) \ge r$ (see Proposition~\ref{discover_balls}). The argument is now classical and goes back all the way to the study of the UIPT in \cite{Ang_growth_UIPT}. Nevertheless, we provide a detailed proof since we are working with non-planar triangulations and must ensure that the argument applies in this context. This will be guaranteed by Proposition~\ref{no_BAD_steps}.

\begin{figure}[H]
\centering  
\includegraphics[scale=0.24]{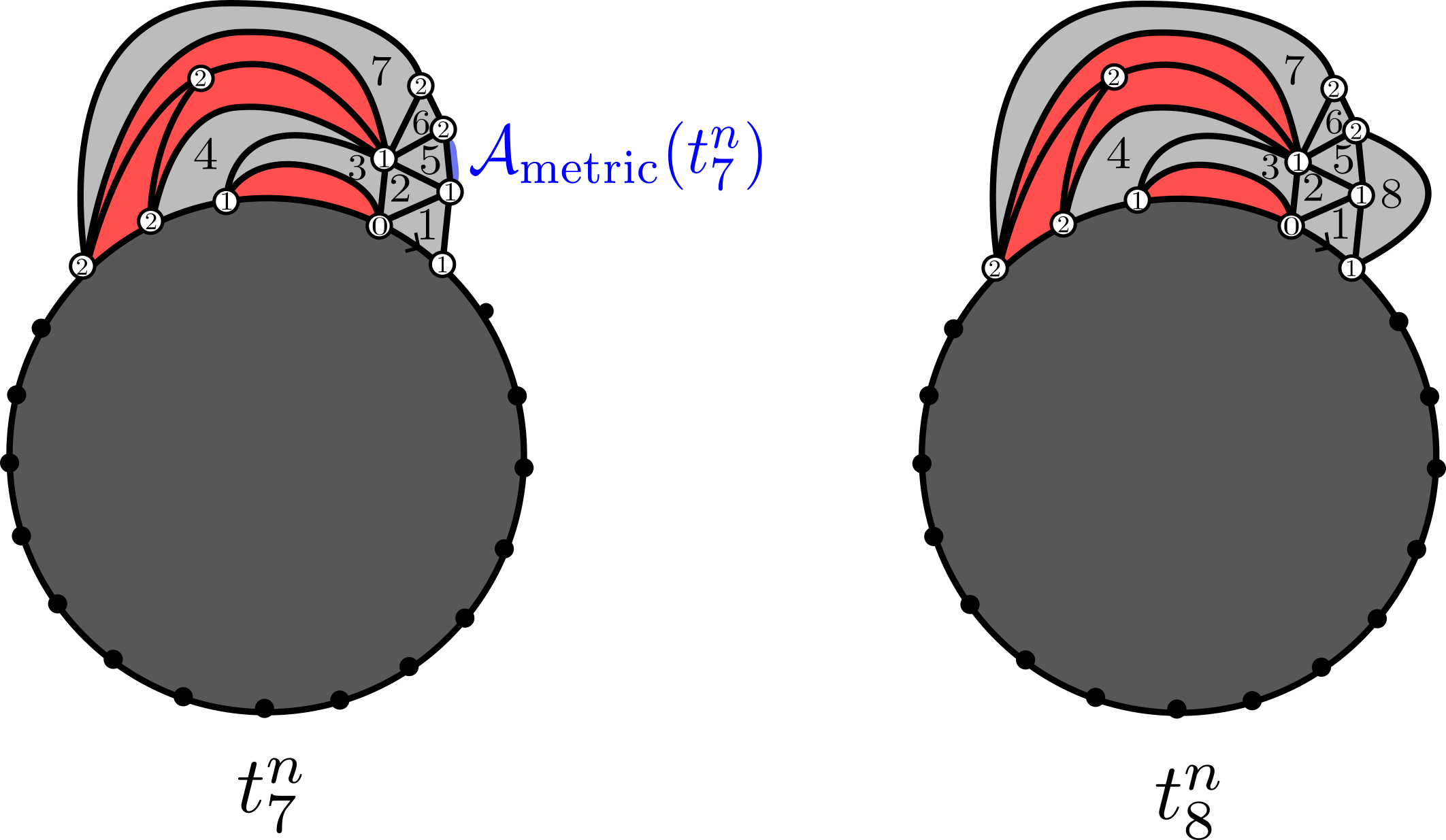}
\caption{On the left, an instance of $t_7^n$. The labels inside the white vertices indicate the distances to the origin in $t_7^n$. The blue edge represents the one to be peeled, $\mathcal{A}_{\mathrm{metric}}(t_7^n)$. On the right, we show the triangulation with holes $t_8^n$ obtained after peeling this edge. Note that the vertex $v^n_7$ to the right of the peeled edge becomes internal in $t_8^n$. In this example, the event $\mathcal{C}_8$ occurs.}
\label{algo}
\end{figure}

\begin{proposition}\label{discover_balls}
    For any $r \ge 0$ and any $\varepsilon > 0$, there exists $k > 0$ such that, for $n$ large enough, 
    \begin{align*}
        \Pf(\sigma_n(k) \ge r)\ge 1-\varepsilon.
    \end{align*}
\end{proposition}

\begin{proof}
We proceed by induction on $r$. The case $r=0$ is trivial. Suppose the result holds for some $r \ge 0$. Fix $\varepsilon >0$ and choose $j \ge 0$ given by the induction hypothesis such that, for $n$ large enough,
\begin{align}\label{induction_prop}
\Pf(\sigma_n(j) \ge r)\ge 1-\frac{\varepsilon}{3}.
\end{align}
 Fix $ j \le k$. We aim to bound $\Pf(\sigma_n(k) < r+1)$ for $n$ large enough. By Proposition~\ref{no_BAD_steps}, there exist $A,a > 0$ such that, for $n$ large enough,
\begin{align}\label{GOOD}
\Pf\big(\mathbf{GOOD}_{n}^{\mathcal{A}_{\mathrm{metric}}}(A,a,k)\big) \ge 1- \frac{\varepsilon}{3}.
\end{align}
Let us recall $\mathcal{P}_n := \Big\{ e^n \text{ lies on a boundary with perimeter at least } \Big(\frac{|\mathbf{p}^n|}{\ell_n}\Big)^{1/2} \Big\}$. For $j \le i$, define the event $F_n(j,i,A,a)$ as the event where $\mathbf{GOOD}_{n}^{\mathcal{A}_{\mathrm{metric}}}(A,a,i)$ holds, $\sigma_n(j) \ge r$ and $\mathcal{P}_n$ occurs. The family of events $(F_n(j,i,A,a))_{j \le i}$ is decreasing in $i$. Moreover, each $F_n(j,i,A,a)$ is measurable with respect to $t^n_i$. By Proposition~\ref{perimeter_typically_large}, \eqref{induction_prop} and \eqref{GOOD}, for $n$ large enough we have 
\begin{align}\label{FnjkAa}
\Pf(F_n(j,k,A,a)) \ge 1-\varepsilon.
\end{align}

For $j \le i$, we introduce the set of vertices on the holes that are exactly at graph distance $r$ from $\rho^n$ in $t^n_i$:
\begin{align}\label{vertex_at_distance_r}
B_{n}(i,r) = \{x \in \partial^{*}t_i^n : d_{t_i^n}(x,\rho^n) = r\}.
\end{align}
We first check that
\[
\#B_n(0,r) \le 2 \quad \text{and} \quad \#(B_n(i+1,r) \setminus B_n(i,r)) \le 1 \text{ , }\forall i \ge 0.
\]
The first statement is obvious, and the second follows because at each peeling step we add at most one new vertex to the hole and we peel to the right of a vertex at minimal distance from the root. In particular, $  j+2\le \#B_n(j,r)$.  
Moreover, under $F_n(j,i,A,a)$, we have
\[
 B_n(i+1,r)\subset B_n(i,r) ,
\]
which follows from the fact that the algorithm $\mathcal{A}_{\mathrm{metric}}$ peels the edge to the left of a vertex minimizing the distance to $\rho^n$. Now we write 
\begin{align}\label{bound_Nrk}
\Pf(\sigma_n(k) \le r) \le \Pf(F_n(j,k,A,a)^{c}) +\Pf(F_n(j,k,A,a),\, B_{n}(k,r) \neq \emptyset).
\end{align}

For $j+1 \le i \le k$, let $C_i$ be the event that the $i^{\mathrm{th}}$ peeling step creates a hole of size $2$ to the left of the peeled edge and that this hole is filled by an empty map (see Figure~\ref{algo}). Then, under $C_i$ and if $\# B_n(i,r) > 0$, we have $\#B_n(i,r) - \# B_n(i-1,r) \le -1$.

We now note that, conditionally on $t^n_i$ such that $F_n(j,i,A,a)$ holds, the triangulation $(T_{n,g_n,\mathbf{p}^n},e^n) \backslash t^n_i$ is uniformly distributed in $\mathcal{T}_{(p^n_1,\cdots,p^n_{s}+q ,\cdots,p^n_{\ell_n})}(n+m,g_n)$, where $s,q,m \in \mathbb{Z}$ are $t^n_i$-measurable. Moreover, $q \in \{-Ai,\cdots,i\}$ and $m \in \{-ai,\cdots,0\}$ and $s \in \{1,\cdots,\ell_n\}$ is such that $p^n_s \ge (\frac{|\mathbf{p}^n|}{\ell_n})^{\frac{1}{2}}$.  

Using Lemma~\ref{bounded_ratio_vertices}, we obtain, for $n$ large enough,
\[
\Pf(C_i \mid t_i) = \frac{\tau_{(p^n_1,\cdots,p^n_{s}+q-1,\cdots,p^n_{\ell_n})}(n+m-1,g_n)}{\tau_{(p^n_1,\cdots,p^n_{s}+q,\cdots,p^n_{\ell_n})}(n+m,g_n)} \ge c_{\theta} > 0.
\]
From this, we can construct, for $n$ large enough, a family of i.i.d. Bernoulli random variables $(Z_i)_{i \ge 0}$ with parameter $c_{\theta}$ such that, for any $j+1 \le i \le k$, under $F_n(j,i,A,a)$, we have $\mathbf{1}_{C_i} \ge Z_i$. It follows that the second term in~\eqref{bound_Nrk} is bounded by 
\[
\Pf\Bigg(\sum_{i=j+1}^{k} Z_i \le j-2\Bigg) = \Pf(\mathrm{Bin}(k-j,c_{\theta}) \le j+2).
\]
Finally, for $n$ large enough, we can write 
\[
\Pf(\sigma_n(k) \le r) \le \Pf(F_n(j,k,A,a)^{c}) + \Pf(\mathrm{Bin}(k-j,c_{\theta}) \le j+2).
\]
Choosing $k$ large enough so that $\Pf(\mathrm{Bin}(k-j,c_{\theta}) \le j+2) \le \varepsilon$ and using~\eqref{FnjkAa} concludes the proof.
\end{proof}

We can now complete the proof of Proposition~\ref{tightness_d_loc}.

\begin{proof}
Fix $r \ge 0$ and $\varepsilon > 0$. Let us recall 
\begin{center}
$\mathcal{P}_n := \Big\{ e^n \text{ lies on a boundary with perimeter at least } \Big(\frac{|\mathbf{p}^n|}{\ell_n}\Big)^{1/2} \Big\}$.
\end{center} 
Using Proposition~\ref{discover_balls}, choose $k > 0$ such that, for $n$ large enough,
\[
\Pf(\sigma_n(k) \ge r)\ge 1-\frac{\varepsilon}{3}.
\]
Using Proposition~\ref{no_BAD_steps}, choose $A,a \ge 0$ such that, for $n$ large enough,
\[
\Pf\big(\mathbf{GOOD}^{\mathcal{A}_{\mathrm{metric}}}_{n}(A,a,k)\big) \ge 1-\frac{\varepsilon}{3}.
\]
Define $\mathcal{F}_{n}(A,a,k,r)$ as the event on which $\mathbf{GOOD}_{n}^{\mathcal{A}_{\mathrm{metric}}}(A,a,k)$ holds, $\sigma_n(k) \ge r$, and the event $\mathcal{P}_n$ occurs. Combining these facts with Proposition~\ref{perimeter_typically_large}, we obtain, for $n$ large enough,
\begin{align}\label{A}
\Pf(\mathcal{F}_n(A,a,k,r)) \ge 1-\varepsilon.
\end{align}

Under $\mathcal{F}_{n}(A,a,k,r)$ and for $n$ large enough, define $\widetilde{t}_k^n$ as the triangulation with holes $t_k^n$ where we remove the edges $e \in \partial t_k^n \cap \partial^{*}t_k^n$.  
In $\widetilde{t}_k^n$, the boundary consists of a segment colored in blue (coming from the boundary of $t_k^n$) and a segment colored in green (coming from the hole of $t_k^n$). Then the ball $B_r(T_{n,g_n,\mathbf{p}^{n}},e^n)$ is contained in $\widetilde{t}_k^n$ (see Figure~\ref{limit_halfplane}).  

Since the first $k$ steps are $\mathbf{GOOD}^{\mathcal{A}_{\mathrm{metric}}}(A,a)$, the triangulation $\widetilde{t}_k^n$ contains at most $k(A+a)$ vertices, ensuring that under $\mathcal{F}_{n}(A,a,k,r)$, the ball $B_r(T_{n,g_n,\mathbf{p}^{n}},e^n)$ can take only finitely many possible values. This establishes tightness.

\begin{figure}[H]
\centering 
\includegraphics[scale=0.2]{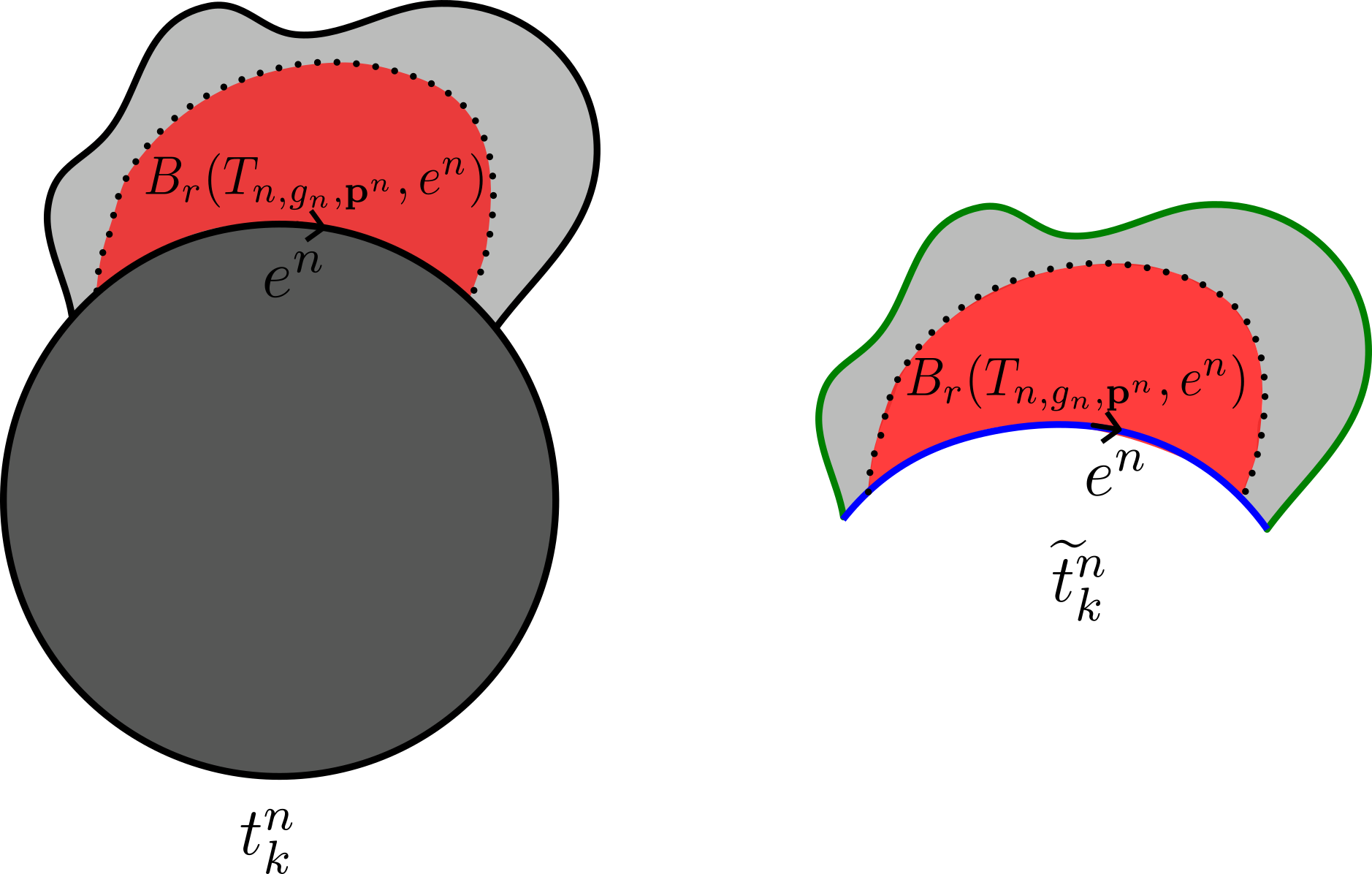}
\caption{Left: the triangulation with holes $t_k^n$. Under $B_n(A,a,k,r)$, it contains the ball of radius $r$. Right: the triangulation $\widetilde{t}_k^n$ obtained after removing edges $e \in \partial t_k^n \cap \partial^{*} t_k^n$.}
\label{limit_halfplane}
\end{figure}

Now fix $(H,e)$ a subsequential limit of $(T_{n,g_n,\mathbf{p}^{n}},e^n)$.  
On $B_n(A,a,k,r)$, the blue part of the boundary of $\widetilde{t}_k^n$ has size greater than $k$, ensuring that $H$ has an infinite boundary. Moreover, since $\widetilde{t}_k^n$ is planar, the limit is also planar. Finally, $\widetilde{t}_k^n$ contains the ball of radius $r$ and $T_{n,g_n,\mathbf{p}^n} \backslash t^n_k$ is connected, which implies that $H$ is one-ended.  

Since $(T_{n,g_n,\mathbf{p}^{n}},e^n)$ is weakly Markovian, the limit $(H,e)$ is also weakly Markovian. This completes the proof.
\end{proof}

\subsection{$\mathbb{H}_{\lambda(\theta)}$ is the limit of $(T_{n,g_n,\mathbf{p}^{n}},e^n)$}\label{last_section}
  
The goal of this section is to prove Theorem~\ref{local_limit_boundary}.  

We recall the definitions of $\mathbb{H}_{\lambda}$ and $\widetilde{\mathbb{H}}_{\lambda}$ from Section~\ref{halfplane_triangulation}. We recall that the triangulations $\widetilde{\mathbb{H}}_{\lambda}$ are subcritical while the triangulations $\mathbb{H}_{\lambda}$ are hyperbolic. We also recall that for $0 < \lambda \le \lambda_c$, the value $h(\lambda)$ is defined by the equation $\lambda = \frac{h}{(1+8h)^{3/2}}$.  

For the remainder of this section, by Proposition~\ref{tightness_d_loc}, we fix a subsequential limit $H$ of $(T_{n,g_n,\mathbf{p}^n},e^n)$.  
Using Proposition~\ref{classification_halfplane}, we fix $\Lambda$, a probability measure on $\mathbb{R}_{+}^2$, supported on
\begin{align*}
\big\{(\lambda,\beta) : 0 < \lambda \le \lambda_c \text{ and } \beta \in \{32h+4,\,8+\tfrac{1}{h}\}\big\} \cup \{(0,4)\},
\end{align*}
such that for any triangulation with holes $t$ having an infinite boundary and one infinite hole, we can write
\begin{align*}
\Pf(t \subset H ) = \int_{\mathbb{R}_{+}^2} \beta^{|\partial^{*}t|-|\partial t|}\lambda^{|t_{\mathrm{in}}|}\, \Lambda(\mathrm{d}\beta,\mathrm{d}\lambda).
\end{align*}
We recall the definition
\begin{align*}
d(\lambda) = \frac{h \log\!\big(\frac{1+\sqrt{1-4h}}{1-\sqrt{1-4h}}\big)}{(1+8h)\sqrt{1-4h}},
\end{align*} 
and that $\lambda(\theta)$ is the unique solution to
\begin{align*}
d(\lambda(\theta)) = \frac{1-2\theta}{6}.
\end{align*}
It remains to show that, $\Lambda$-almost surely,
\begin{center}
$\lambda = \lambda(\theta)$ and $\beta = 8+ \dfrac{1}{h(\lambda(\theta))}$.
\end{center}

We begin by proving that $\lambda = \lambda(\theta)$ almost surely.

\begin{lemma}\label{lambda_deter}
$\Lambda$-almost surely, $\lambda = \lambda(\theta)$.
\end{lemma}

\begin{proof}
For $r \ge 0$, let $t$ be a triangulation with holes having one infinite boundary, one infinite hole, exactly $r$ internal vertices and such that $|\partial^{*}t|=|\partial t|$. Then
\begin{align*}
\Pf(t \subset H) = \int_{\mathbb{R}_{+}^2} \lambda^{r}\, \Lambda(\mathrm{d}\beta,\mathrm{d}\lambda).
\end{align*}
On the other hand, since $H$ is a local limit of $(T_{n,g_n,\mathbf{p}^{n}},e^n)$ along a subsequence $(n_k)_{k \ge 0}$, and using Proposition~\ref{application_cv_ratio}, we can write
\begin{align*}
\Pf(t \subset H) = \lim_{k\to +\infty} \frac{\tau_{\mathbf{p}_{n_k}}(n_k-r,g_{n_k})}{\tau_{\mathbf{p}_{n_k}}(n_k,g_{n_k})} = \lambda(\theta)^r.
\end{align*}
Thus, the marginal distribution of $\lambda$ under $\Lambda$ has the same moments (of all orders) as $\delta_{\lambda(\theta)}$. Since $\delta_{\lambda(\theta)}$ has compact support, the two measures coincide. This proves the lemma.
\end{proof}

Finally, we prove that $\beta = 8+\frac{1}{h(\lambda(\theta))}$ almost surely.  
By Lemma~\ref{lambda_deter}, we can write $H = x\,\widetilde{\mathbb{H}}_{\lambda(\theta)} +(1-x)\,\mathbb{H}_{\lambda(\theta)}$ for some $0\le x\le 1$, i.e. $H$ is a mixture of $\widetilde{\mathbb{H}}_{\lambda(\theta)}$ and $\mathbb{H}_{\lambda(\theta)}$.  
If $\theta = 0$, using the first Item of Proposition~\ref{classification_halfplane}, we have $\widetilde{\mathbb{H}}_{\lambda(\theta)} = \mathbb{H}_{\lambda(\theta)}$, which immediately yields Theorem~\ref{local_limit_boundary}.  
Hence, we now assume $\theta > 0$ and reason by contradiction, assuming $x > 0$.  

Our main tool will be a surgery operation on the peeling diagrams inspired from \cite{Contat2022LastCD}. In words, the peeling diagram associated to a peeling exploration is a decorated diagram that encapsulates the genealogy of the peeling steps. See Section~\ref{peeling_def} for the formal definitions.

We now define a suitable peeling algorithm, denoted by $\mathcal{A}_{\mathrm{left}}$.  
Informally, this algorithm always peels on the “leftmost” boundary.  
We now make this notion precise.  
The definition of $\mathcal{A}_{\mathrm{left}}$ is recursive.  
Fix $t \in \mathcal{T}_{\mathbf{p}^n}(n,g_n)$ and an oriented edge $e$ lying on $\partial_1$. Then $\mathcal{E}^{\mathcal{A}_{\mathrm{left}}}_0(t,e)$ is the polygon $\partial_1$ rooted at $e$. We define $\mathcal{A}_{\mathrm{left}}(\mathcal{E}^{\mathcal{A}_{\mathrm{left}}}_0(t,e))$ as any edge chosen arbitrarily on $\partial_1$.  
Suppose that for $k \ge 0$, we have constructed $\mathcal{E}^{\mathcal{A}_{\mathrm{left}}}_k(t,e)$.  
Then, consider the peeling diagram $\mathcal{D}^{\mathcal{A}_{\mathrm{left}}}_k(t,e)$.  
For each hole $h$ of $\mathcal{E}^{\mathcal{A}_{\mathrm{left}}}_k(t,e)$, denote by $\gamma(h)$ the leftmost oriented path in $\mathcal{D}^{\mathcal{A}_{\mathrm{left}}}_k(t,e)$ from the root to $h$.  
The notion of “leftmost” is well-defined via the labelling $\leftarrow$ and $\rightarrow$ of the edges in $\mathcal{D}^{\mathcal{A}_{\mathrm{left}}}_k(t,e)$.  
Let $h_0$ denote the hole whose associated path $\gamma(h_0)$ is the leftmost among $\{\gamma(h)\}_h$.  
We then define $\mathcal{A}^{\mathrm{left}}(\mathcal{E}^{\mathcal{A}_{\mathrm{left}}}_k(t,e))$ by choosing an arbitrary edge on $h_0$.  
As will be seen later, the use of this specific peeling algorithm is crucial.\\
 
The main idea behind the argument will be that the half-plane triangulation $\widetilde{\mathbb{H}}_{\lambda(\theta)}$ has a boundary that “folds into itself”. Indeed, in $\widetilde{\mathbb{H}}_{\lambda(\theta)}$ the probability to have the first peeling step that swallows at least $k$ edges on the boundary is of order $r^{-\frac{1}{2}}$, while in $\mathbb{H}_{\lambda}$ it is of order $\exp(-Ck)$ for some constant $C > 0$. Then, with probability of order $r^{-\frac{1}{2}}$, the peeling diagram of $\widetilde{\mathbb{H}}_{\lambda(\theta)}$ might split into two disconnected components of size more than $r$. We show in Proposition~\ref{no_fattrees_locally} that the decay of this event is at least $r^{-1}$ in $T_{n,g_n,\mathbf{p}^n}$ which will exclude $\widetilde{\mathbb{H}}_{\lambda(\theta)}$.

 We now formalize this notion.

\begin{definition1}
Fix $t \in \mathcal{T}_{\mathbf{p}^n}(n,g_n)$ and an oriented edge $e$ lying on $\partial_1$.  
Fix a peeling algorithm $\mathcal{A}$.  
For any $r \ge 0$, we say that $\mathcal{D}^{\mathcal{A}}(t,e)$ satisfies $(\mathcal{S}_{r})$ if the following conditions hold:
\begin{enumerate}
 	\item \label{condi1} The first step of the peeling exploration is of type IV (the discovered triangle splits a boundary into two holes) and disconnects the peeling diagram $\mathcal{D}^{\mathcal{A}}(t,e)$ into two connected components.
 	\item \label{condi2} The right component (the one induced by the edge labelled $\rightarrow$) is a tree with more than $r$ vertices in total.
\end{enumerate}
\end{definition1} 

We recall that $e^n_1$ denotes the distinguished edge on $\partial_1$. We now show that for large $r$, there are few $t \in \mathcal{T}_{\mathbf{p}^{n}}(n,g_n)$ satisfying the above condition.

\begin{lemma}\label{no_fattrees_locally}
For $\theta > 0$, there exists a constant $C_{\theta} > 0$ such that for any $r \ge 0$, we have
\begin{align*}
\#\{t \in \mathcal{T}_{\mathbf{p}^{n}}(n,g_n) : \mathcal{D}^{\mathcal{A}_{\mathrm{left}}}(t,e^n_1) \text{ satisfies } (\mathcal{S}_r)\}\le C_{\theta}r^{-1}\tau_{\mathbf{p}^{n}}(n,g_n).
\end{align*}
\end{lemma}

\begin{proof}
Fix $r \ge 0$ and $t \in \mathcal{T}_{\mathbf{p}^{n}}(n,g_n)$ such that $\mathcal{D}^{\mathcal{A}_{\mathrm{left}}}(t,e^n_1)$ satisfies $(\mathcal{S}_r)$.  
The root of $\mathcal{D}^{\mathcal{A}_{\mathrm{left}}}(t,e^n_1)$ has two children: one labelled $p_1^n-k$ (connected by the edge labelled $\leftarrow$) and one labelled $k+1$ (connected by the edge labelled $\rightarrow$).  
Let $\mathcal{L}$ (resp. $\mathcal{R}$) denote the descendants of the left child (resp. right child) of $\mathcal{D}^{\mathcal{A}_{\mathrm{left}}}(t,e^n_1)$ at the root.  
We also choose:
\begin{itemize}
	\item[$\bullet$] a vertex $x \in \mathcal{R}$.
	\item[$\bullet$] a vertex $y \in \mathcal{L}$ labelled $a \in \mathbb{N}^{*}$ such that the edge outgoing from $y$ is labelled by a number $s$, i.e. the corresponding peeling step is of type V (the discovered triangles connect two distinct holes). Let $z$ denote the child of $y$ and $e$ the oriented edge from $y$ to $z$.
\end{itemize} 

Let $\gamma$ be the leftmost oriented path from the root of the diagram to $y$ (see the orange path in Figure~\ref{peeling_diagram_surgery}).  
We now define a new peeling diagram obtained by a surgery operation, denoted $\widetilde{\mathcal{D}}^{\mathcal{A}_{\mathrm{left}}}(t)$.  
Care must be taken to preserve the correct labelling (see Figure~\ref{peeling_diagram_surgery}):
\begin{enumerate}
\item remove the root vertex and the edges emanating from it.
\item along the oriented path $\gamma$, increase the labels of the encountered vertices by $k$.
\item add an oriented edge labelled $\leftarrow$ starting at $y$ and connecting to $\mathcal{R}$.  
Split $e$ into two consecutive oriented edges: the first labelled $\rightarrow$, the second labelled $s$, with an intermediate vertex labelled $a$.
\end{enumerate} 

\begin{figure}[H]
\centering
\includegraphics[scale=0.3]{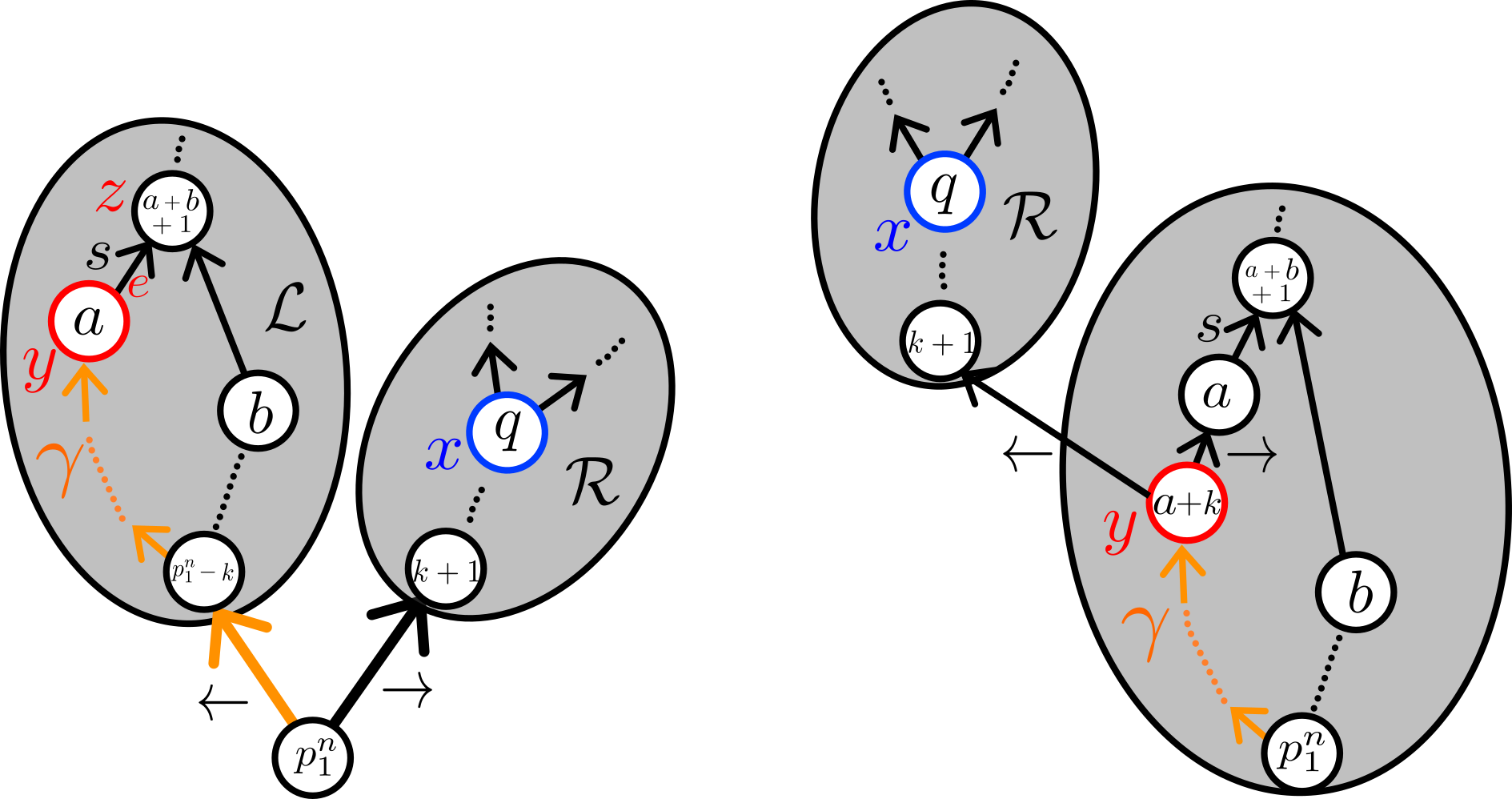}
\caption{Left: the peeling diagram $\mathcal{D}^{\mathcal{A}_{\mathrm{left}}}(t,e^n_1)$ with distinguished vertices $x$ (blue) and $y$ (red).  
The oriented path $\gamma$ is shown in orange.  
Right: the new peeling diagram $\widetilde{\mathcal{D}}^{\mathcal{A}_{\mathrm{left}}}(t,e^n_1)$ obtained after surgery.  
Note that the labels along $\gamma$ are all increased by $k$.}
\label{peeling_diagram_surgery}
\end{figure}

We claim that this new diagram is of the form $\mathcal{D}^{\mathcal{A}_{\mathrm{left}}}(\widetilde{t},e^n_1)$ for some $\widetilde{t} \in \mathcal{T}_{\mathbf{p}^{n}}(n,g_n)$.  
Indeed, given $\widetilde{\mathcal{D}}^{\mathcal{A}_{\mathrm{left}}}(t,e)$, one can reconstruct all peeling steps of a peeling exploration with algorithm $\mathcal{A}_{\mathrm{left}}$. Note that it is crucial that we are working with algorithm $\mathcal{A}_{\mathrm{left}}$. Indeed if one uses another peeling algorithm $\mathcal{A}'$ and copy the argument, it could turn out that the diagram obtained after the surgery operation is no longer a peeling diagram with algorithm $\mathcal{A}'$.

Define $\varphi(t,x,y) = (\widetilde{t},x)$.  
We claim that $\varphi$ is injective.  
Indeed, given $\varphi(t,x,y)$, and since $\mathcal{R}$ is a tree, one can start from $x$ and backtrack until reaching the first vertex $v$ that has a descendant that is a type~V step.  
This vertex is $y$.  
From $y$, one can then recover the modified region and reconstruct $(t,x,y)$.\\

Next, we compare the number of inputs and outputs of this injection.  
For any $t \in \mathcal{T}_{\mathbf{p}^{n}}(n,g_n)$ such that $\mathcal{D}^{\mathcal{A}_{\mathrm{left}}}(t,e^n_1)$ satisfies $(\mathcal{S}_r)$, the component $\mathcal{R}$ has more than $r$ vertices, giving at least $r$ choices for $x$.  
Moreover, since there are $g_n$ steps of type~V in $\mathcal{D}^{\mathcal{A}_{\mathrm{left}}}(t,e^n_1)$, there are $g_n$ choices for $y$.  
Thus, the number of inputs is at least
\begin{align*}
g_n r \#\{t \in \mathcal{T}_{\mathbf{p}^{n}}(n,g_n) : \mathcal{D}^{\mathcal{A}_{\mathrm{left}}}(t,e^n_1) \text{ satisfies } (\mathcal{S}_r)\}.
\end{align*}

For any $\widetilde{t}$, the number of possible $x$ is bounded by the number of vertices in $\mathcal{D}^{\mathcal{A}_{\mathrm{left}}}(\widetilde{t})$.  
Each internal vertex (except the root) in $\mathcal{D}^{\mathcal{A}_{\mathrm{left}}}(\widetilde{t})$ corresponds to a peeling step that introduced it: this mapping is at most two-to-one since type~IV steps introduce two new vertices in $\mathcal{D}^{\mathcal{A}_{\mathrm{left}}}(\widetilde{t})$. Thus, the number of internal vertices in $\mathcal{D}^{\mathcal{A}_{\mathrm{left}}}(\widetilde{t})$ is bounded by the number of internal faces in $t$. Each leaf corresponds to an edge in the triangulation.  
Hence, the number of vertices in the diagram is bounded by $2\#F(\widetilde{t})+\#E(\widetilde{t})+1$, where $\#E(\widetilde{t}) = 3n - \sum_{i=1}^{\ell}(p^n_i-3)$ and $\#F(\widetilde{t}) = 2n - \sum_{i=1}^{\ell}(p^n_i-2)$.  
Therefore, the number of outputs of $\varphi$ is bounded by $10n\tau_{\mathbf{p}^{n}}(n,g_n)$.  
Using the injectivity of $\varphi$ and Lemma~\ref{bounded_ratio_vertices}, we obtain  
\begin{align*}
g_n r \#\{t \in \mathcal{T}_{\mathbf{p}^{n}}(n,g_n) : \mathcal{D}^{\mathcal{A}_{\mathrm{left}}}(t) \text{ satisfies } (\mathcal{S}_r)\}
\le 10n\tau_{\mathbf{p}^{n}}(n,g_n).
\end{align*}
Dividing by $g_n r$ and noting that for $n$ large enough $g_n  \ge \frac{\theta}{2}n$ completes the proof.
\end{proof}

We can now conclude the proof of Theorem~\ref{local_limit_boundary}.

\begin{proof}
We have 
\[
(T_{n,g_n,\mathbf{p}^n},e^n) \underset{n \to +\infty}{\overset{(d)}{\longrightarrow}} x\,\widetilde{\mathbb{H}}_{\lambda(\theta)} + (1-x)\,\mathbb{H}_{\lambda(\theta)}.
\]
Let us assume $x > 0$. For $j \ge 0$, $t \in \mathcal{T}_{\mathbf{p}^n}(n,g_n)$ and $e$ an edge that lies on a boundary of $t$, we recall that $(t,e)$ satisfies $\mathcal{R}^{\mathrm{sep}}_j$ if the triangle behind $e$ in $t$ has its third vertex that lies on the same boundary as $e$ and the right hole created has perimeter $j+1$ and disconnects $t$ in two disconnected components. For any $k \in \{1,\cdots,\ell_n\}$, recalling that $e^n_k$ denotes the distringuished edge on $\partial_k$, let $\mathcal{A}_{n}(r,k)$ denote the event that $(T_{n,g_n,\mathbf{p}^n},e^n_k)$ satisfies $\mathcal{R}^{\mathrm{sep}}_j$ for some $j \ge r$, and that the triangulation $T^2$ filling the hole of perimeter $j+1$ to the right is planar.  
Let $\mathcal{A}(r)$ be the event that $(\widetilde{\mathbb{H}}_{\lambda(\theta)},e)$ satisfies $\mathcal{R}^{\mathrm{sep}}_j$ with $j \ge r$, where $e$ is the root edge of $\widetilde{\mathbb{H}}_{\lambda(\theta)}$.  

we get, for $n$ large enough,
\begin{align}
\sum_{k=1}^{\ell_n}\frac{p^n_k}{|\mathbf{p}^{n}|}\Pf(\mathcal{A}_n(r,k)) \ge \frac{x}{2}\Pf(\mathcal{A}(r)).
\end{align}

Using \eqref{formula_wlambda}, we have
\begin{align*}
\Pf(\mathcal{A}(r)) = (32h+4)^{-r}w_{\lambda(\theta)}(r+1) \underset{r \to +\infty}{\sim} C_{\lambda(\theta)}r^{-1/2}.
\end{align*}
Hence, for $r$ and $n$ large enough, there exists $k \in \{1,\cdots,\ell_n\}$ such that 
\[
\Pf(\mathcal{A}_n(r,k)) \ge \frac{x}{2}C_{\lambda(\theta)}r^{-1/2}.
\]
Now, we claim that on $\mathcal{A}_n(r,k)$, the peeling diagram $\mathcal{D}^{\mathcal{A}_{\mathrm{left}}}\big(T_{n,g_n,\mathbf{p}^n},e^n_k\big)$ satisfies $(\mathcal{S}_r)$.

Indeed, the first peeling step is of type~IV.  
Moreover, since the right hole has perimeter $r+1$ and is filled by a planar triangulation of the $(r+1)$-gon, the component $\mathcal{R}$ induced by the edge $\rightarrow$ at the root of $\mathcal{D}^{\mathcal{A}_{\mathrm{left}}}\big(T_{n,g_n,\mathbf{p}^n},e_k^n\big)$ is a tree with at least $r+1$ vertices.  
However, by Lemma~\ref{no_fattrees_locally}, we also have 
\[
\Pf\big(\mathcal{D}^{\mathcal{A}_{\mathrm{left}}}(T_{n,g_n,\mathbf{p}^n},e^n_k) \text{ satisfies } (\mathcal{S}_r)\big) \le C_{\theta}r^{-1}.
\]
We therefore deduce
\[
\frac{x}{4}C_{\lambda(\theta)}r^{-1/2} \le C_{\theta}r^{-1},
\]
which leads to a contradiction for large $r$.  
This concludes the proof.
\end{proof}

\appendix
\section{Martin boundaries and harmonic functions}

In this appendix, we provide a classification of harmonic functions for a specific class of Markov chains. The arguments used in this part are probably standard. However, we did not find a reference that works in our context. We fix a probability measure $\mu$ on 
\[
S = \{(1,1)\} \cup \{(-k,m) : (k,m) \in \mathbb{N}^2\}.
\]
We write 
\[
X = \{(p,v) \in \mathbb{Z}^2 : v \ge \max(0,p)\}.
\]
Let us introduce the vectors $e_1 = (-1,0)$ and $e_2 = (1,1)$.  
In this way, any element $(x,y) \in X$ can be uniquely decomposed as $(x,y) = k_1 e_1 + k_2 e_2$ with $k_1,k_2 \ge 0$.  
The set $X$ is stable under addition by an element of $S$.  
We assume that $\mu(e_1) > 0$ and $\mu(e_2) > 0$.  
For any $x \in X$, the random walk starting from $x$ with i.i.d.\ increments distributed according to $\mu$ takes its values in $X$.  

A function $h$ on $X$ is said to be \emph{$\mu$-harmonic} if
\begin{align}\label{harmonicity}
\forall x \in X,\quad h(x) = \sum_{y \in S} \mu(y) \, h(x+y).
\end{align}
A nonnegative function $h$ on $X$ is said to be \emph{minimal $\mu$-harmonic} if it is $\mu$-harmonic and if, for any other nonnegative $\mu$-harmonic function $w$ such that $0 \le w \le h$, there exists $C \ge 0$ with $w = C h$.  

Following the construction of \cite[Section~4]{martin}, there exists a metric space $\partial_m X_M$, called the \emph{minimal Martin boundary}, with the following properties:
\begin{enumerate}
	\item \label{it1app} For any $\alpha \in \partial_m X_M$, there exists a minimal $\mu$-harmonic function $K(\cdot,\alpha) \ge 0$.
	\item \label{it2app} For any positive $\mu$-harmonic function $h$, there exists a unique measure $\nu_h$ on $\partial_m X_M$ such that, for any $x \in S$,
	\[
	h(x) = \int_{\alpha \in \partial_m X_M} K(x,\alpha)\,\nu_h(\mathrm{d}\alpha).
	\]
\end{enumerate} 
We leave it to the reader to verify that the function defined by 
\[
h_0(p,v) = \mu(e_2)^{-v}\mathbf{1}_{\{p=v\}}
\]
is a minimal $\mu$-harmonic function on $X$.  
The next lemma classifies all minimal $\mu$-harmonic functions.

\begin{lemma}\label{minimal_harmonic}
Let $h$ be a minimal $\mu$-harmonic function on $S$ such that $h(0,0) = 1$, distinct from $h_0$.  
Then, there exist constants $\eta,\alpha \ge 0$ such that
\[
\forall (x,y) \in X, \quad h(x,y) = \eta^x \alpha^y.
\]
\end{lemma}

\begin{proof}
Fix a minimal $\mu$-harmonic function $h$ on $S$ with $h(0,0) = 1$ and $h \neq h_0$.  
For any $(p,v) \in X$, define the shifted function $h_{(p,v)}$ by $h_{(p,v)}(a,b) = h(p+a, v+b)$.  
Fix $n \ge 0$ such that 
\[
c_{p,v} := \mu^{n}((0,0),(p,v)) > 0,
\]
which is always possible.  
By harmonicity, we have 
\[
0 \le h_{(p,v)}(a,b)c_{p,v} \le h(a,b).
\]
By minimality, there exists a constant $d_{p,v} \ge 0$ such that $h_{(p,v)} = d_{p,v} h$.  
Evaluating at $(0,0)$ gives $d_{p,v} = h(p,v)$, hence
\[
h(p+a,v+b) = h(p,v)h(a,b).
\]
We may therefore write 
\[
h(p,v) = h(e_1)^{v-p}h(e_2)^{v}.
\]
Since $h \neq h_0$, we must have $h(e_1) > 0$.  
The result follows by setting $\eta = h(e_1)^{-1}$ and $\alpha = h(e_1)h(e_2)$.
\end{proof}

As a consequence, we obtain the following corollary.

\begin{corollary}\label{harmonic}
For any $\mu$-harmonic function $h$, there exist a measure $\Lambda_0$ on $(\mathbb{R}_{+})^2$ and a constant $\lambda \ge 0$ such that, for any $(x,y) \in X$,
\[
h(x,y) = \int_{(\eta,\alpha) \in (\mathbb{R}_{+})^2} \eta^x \alpha^y \, \Lambda_0(\mathrm{d}\eta,\mathrm{d}\alpha) + \lambda\, h_0(x,y).
\]
\end{corollary}

\printbibliography
\end{document}